\documentclass[12pt]{amsart}

\usepackage[active]{srcltx}

\makeatletter \@addtoreset{equation}{section}
\@addtoreset{figure}{section}

\makeatother

\usepackage{pstcol}
\usepackage{pstricks, pst-node}
\usepackage{xypic,amscd}
\usepackage{tikz-cd}
\usepackage{tikz}
\usepackage{amsmath}
\usepackage{amssymb}
\usepackage{latexsym}
\newcommand{\Supp}{\operatorname{Supp}}
\newcommand{\Spec}{\operatorname{Spec}}

\newcommand{\PP}{\mathbb{P}}

\newcommand{\FF}{\mathbb{F}}
\newcommand{\RR}{\mathbb{R}}
\newcommand{\CC}{\mathbb{C}}

\newcommand{\AAA}{\mathbb{A}}

\newcommand{\NE}{\mathbb{NE}}

\newtheorem{theorem}[equation]{Theorem}

\newtheorem{lemma}[equation]{Lemma}

\theoremstyle{definition}

\newtheorem{definition}[equation]{Definition}

\theoremstyle{remark}

\date{}

\begin{document}

\title[]{Cylinders in del Pezzo surfaces with du Val singularities}
\author[]{Grigory Belousov and Nivedita Viswanathan}
\maketitle

\begin{abstract}
We consider del Pezzo surfaces $X$ with du Val singularities. Assume that $X$ has a $-K_X$-polar cylinder and $\deg X=1$. Let $H$ be an ample divisor. We'll prove that $X$ has a $H$-polar cylinder.
\end{abstract}

\section{intoduction}

A \emph{log del Pezzo surface} is a projective algebraic surface $X$
with only quotient singularities and ample anti-canonical divisor
$-K_{X}$. In this paper we assume that $X$ has only du Val singularities and we work over complex number field $\CC$. Note that a del Pezzo surface with only du Val singularities is rational.

\begin{definition}[see. {\cite{KPZ}}]
Let $M$ be a $\RR$-divisor on a projective normal variety $X$. An \emph{$M$-polar cylinder} in $X$ is an open subset $U=X\backslash\Supp(D)$ defined by an effective $\RR$-divisor $D$ such that $D\equiv M$ and $U\cong Z\times\AAA^1$ for some affine variety $Z$.
\end{definition}

In this paper, we consider del Pezzo surfaces with du Val singularities over complex number field $\CC$. Our interest is a connection between existence of a $-K_X$-polar cylinder in the del Pezzo surface and existence of a $H$-polar cylinder, where $H$ is an arbitrary ample divisor on $X$.

The existence of a $H$-polar cylinder in $X$  is important due to the following fact.

\begin{theorem}[see {\cite{KPZ1}}, Corollary 3.2]
Let $Y$ be a normal algebraic variety over $\CC$ projective over an affine
variety $S$ with $\dim_S Y\geq 1$. Let $H\in  Div(Y)$ be an ample divisor on $Y$, and let $V =
\Spec A(Y,H)$ be the associated affine quasicone over $Y$. Then $V$ admits an effective
$G_a$-action if and only if $Y$ contains an $H$-polar cylinder.
\end{theorem}

When $H=-K_{S_d}$, these have been studied in {\cite{Ch1}}, {\cite{Ch2}}.

\begin{theorem}\cite[Theorem 1.5]{Ch1}
\label{theorem:cylinders:-K_S} Let $X_d$ be a del Pezzo surface of degree $d$ with at most
du Val singularities.
\begin{itemize}
\item[I.] The surface $X_d$ does not admit a $(-K_{X_d})$-polar cylinder when
\begin{enumerate}
\item $d=1$ and  the surface $X_d$ allows only singular points of types $\mathrm{A}_1$, $\mathrm{A}_2$, $\mathrm{A}_3$,
$\mathrm{D}_4$ if any;
\item $d=2$ and $S_d$ allows only singular points of type $\mathrm{A}_1$ if any;
\item $d=3$ and $S_d$ allows no singular point.%
\end{enumerate}
\item[II.] The surface $X_d$ has a $(-K_{X_d})$-polar cylinder if it is not one of the del Pezzo surfaces listed in I.
\end{itemize}
\end{theorem}

The main result of this paper is the followings.

\begin{theorem}
\label{glav}
Let $X$ be a del Pezzo surface with du Val singularities. Assume that $X$ has a $-K_X$-polar cylinder. Let $H$ be an ample divisor. Then $X$ has a $H$-polar cylinder.
\end{theorem}

In {\cite{Saw1}, {\cite{Saw2} proved this theorem for cases, when $\deg X\geq 2$. So, we may assume that $\deg X=1$.

The authors is grateful to professor I. A. Cheltsov for suggesting
me this problem and for his help.

\section{Preliminary results}

We work over complex number field $\CC$.
We employ the following notation:
\begin{itemize}
\item
$(-n)$-curve is a smooth rational curve with self intersection
number $-n$.
\item
$K_{X}$: the canonical divisor on $X$.
\item
$\rho(X)$: the Picard number of $X$.
\end{itemize}

Let $X$ be a del Pezzo surface with du Val singularities and let $H$ an arbitrary ample divisor on $X$. Assume that $\rho(X)=1$. Then $H\equiv -K_X$. So, $X$ has an $H$-polar cylinder if and only if $X$ has a $-K_X$-polar cylinder. So we may assume that $\rho(X)>1$.

\begin{definition}
Put $$\mu_H:=\inf\left\{\lambda\in\RR_{>0}\mid\text{ the $\RR$-divisor $K_X+\lambda H$ is pseudo-effective}\right\}.$$ The number $\mu_H$ is called Fujita invariant of $(X,H)$. The smallest extremal face $\Delta_H$ of the Mori cone $\overline{\NE(X)}$ that contains $K_X+\mu_H H$ is called the Fujita face of $H$. The Fujita rank of $H$ is defined by $r_H:=\dim\Delta_H$. Note that $r_H=0$ if and only if $-K_X\equiv \mu_H H$.
\end{definition}

Let $\phi_H X\rightarrow Y$ be the contraction given by the $\Delta_H$. Then either $\phi_H$ is a birational morphism or a conic bundle with $Y\cong\PP^1$.
In the former case, the $\RR$-divisor $H$ is said to be of type $B(r_H)$ and in the latter case it is said
to be of type $C(r_H)$.

Suppose that $H$ is of type $B(r_H)$. Then $$K_X+\mu_H H\equiv\sum\limits_{i=1}^{r_H} a_i E_i,$$ where $E_1,\ldots,E_{r_H}$ are curves contained in $\Delta_H$ and $a_1,\ldots,a_{r_H}$ are positive real numbers such that $a_i<1$ for every $i$.

Suppose that $H$ is of type $C(r_H)$. Note that $r_H=9-d$. There are a $0$-curve $B$ and $(8-d)$ disjoint curves $E_1,\ldots,E_{8-d}$, each of which is contained in a distinct fiber of $\phi_H$, such that $$K_X+\mu_H H\equiv aB+\sum\limits_{i=1}^{8-d} a_i E_i,$$
for some positive real number $a$ and non-negative real numbers $a_1,\ldots,a_{8-d}<1$.

\section{The proof of theorem \ref{glav}}

\begin{lemma}
\label{Lem1}
Let $X$ be a del Pezzo surface with du Val singularities and let $H$ be an ample divisor on $X$. Assume that $X$ has a singular point of type $E$. Then $X$ has an $H$-polar cylinder.
\end{lemma}

\begin{proof}
Let $\varphi\colon\tilde{X}\rightarrow X$ be the minimal resolution of singularities of $X$ and let $D=\sum_{i=1}^n D_i$ be the exceptional divisor of $\varphi$.

Note that $H$ is of type $B(r_H)$. If $X$ has a singular point of type $E_8$ then $\rho(X)=1$. Assume that $X$ has a singular point $P$ of type $E_7$ and $\rho(X)=2$. Then $P$ is a unique singular point and $$H\equiv-K_X+aE,$$ where $E$ is a $(-1)$-curve and $0<a<1$.  So, we have the following configuration on $\tilde{X}$.

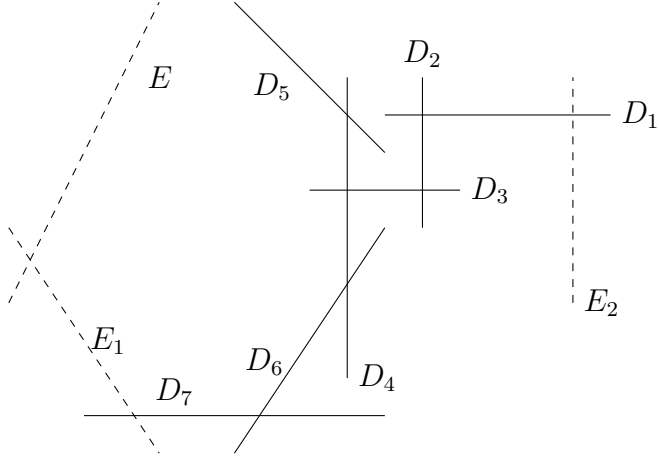
\begin{tikzpicture}
\draw  (1,0.5) -- (5,0.5); \draw [dashed] (0,3) -- (2,0); \draw [dashed] (0,2) -- (2,6); \draw (3,0) -- (5,3); \draw  (4.5,1) -- (4.5,5); \draw (3,6) -- (5,4);
\draw  (4,3.5) -- (6,3.5) (5.5, 5)-- (5.5,3); \draw (5,4.5)-- (8,4.5); \draw [dashed] (7.5,2) -- (7.5,5); \draw (1.7,5) node [right]  {$E$}  (1.7,1.5) node [left] {$E_1$}
(7.5,2) node [right] {$E_2$}  (2.2,0.5) node [above] {$D_7$} (3.4,0.9) node [above] {$D_6$} (4.5,1) node [right] {$D_4$} (3.5,5.2) node [below] {$D_5$} (6,3.5) node [right] {$D_3$} (5.5,5) node [above] {$D_2$} (8,4.5) node [right] {$D_1$};
\end{tikzpicture}\\
where $E, E_1, E_2$ are $(-1)$-curves and $D_1,\ldots, D_7$ are $(-2)$-curves. Note that \begin{gather*}
\varphi^*(H)\equiv L= 2D_7+(3+\epsilon)D_6+(2+\epsilon)D_5+(4+2\epsilon)D_4+(3+2\epsilon)D_3+\\
+(2+2\epsilon)D_2+(1+2\epsilon)D_1+2\epsilon E_2+(1-\epsilon)E_1+(a-\epsilon)E.
\end{gather*}
We see that for sufficiently small $\epsilon$, $L$ is effective. Moreover, there exists $\PP^1$-fibration $g\colon\tilde{X}\rightarrow\PP^1$ with two singular fibers $$C_1=E+E_1,\quad C_2=2E_2+2D_1+2D_2+2D_3+2D_4+D_5+D_6.$$ Then $U\cong\tilde{X}\setminus\Supp(L)\cong\FF_2\setminus(M+F_1+F_2)$, where $M$ is a unique $(-2)$-curve and $F_1, F_2$ are fibers. So, $X$ has an $H$-polar cylinder.

Assume that $X$ has a singular point $P$ of type $E_6$ and $\rho(X)=2$. Since $\deg(X)=1)$, we see that $X$ has one more singular point $P'$ of type $A_1$. In
the same way, $$H\equiv-K_X+aE,$$ where the proper transform of $E$ on $\tilde{X}$ is a $(-1)$-curve $\hat{E}$ and $0<a<2$. Moreover, by classification of del Pezzo surface with du Val singularities and Picard number is equal one, we see that $E$ contain the singular point $P'$. We have $\varphi^*(E)=\hat{E}+\frac{1}{2}D_1$, where $D_1$ is a unique $(-2)$-curve that correspond to $P'$. So, we have the following configuration on $\tilde{X}$.

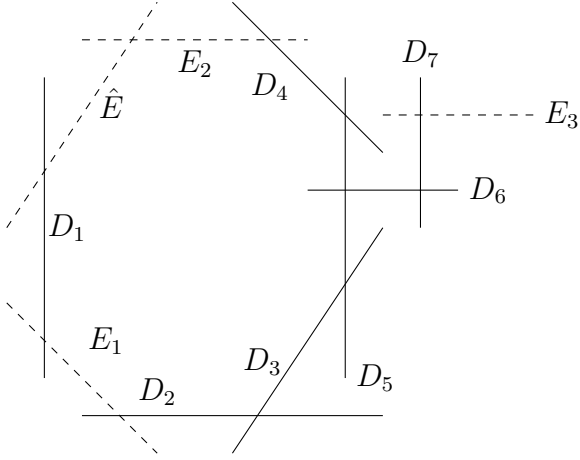
\begin{tikzpicture}
\draw  (1,0.5) -- (5,0.5); \draw [dashed] (0,2) -- (2,0); \draw  (0.5,1) -- (0.5,5); \draw [dashed] (0,3) -- (2,6); \draw (3,0) -- (5,3); \draw  (4.5,1) -- (4.5,5); \draw (3,6) -- (5,4);
\draw  (4,3.5) -- (6,3.5) (5.5, 5)-- (5.5,3); \draw [dashed] (5,4.5)-- (7,4.5) (1,5.5) -- (4,5.5); \draw (1.7,5) node [below left]  {$\hat{E}$} (1.2,3) node [left] {$D_1$} (1.7,1.5) node [left] {$E_1$}
(2.5,5.5) node [below] {$E_2$} (7,4.5) node [right] {$E_3$} (2,0.5) node [above] {$D_2$} (3.4,0.9) node [above] {$D_3$} (4.5,1) node [right] {$D_5$} (3.5,5.2) node [below] {$D_4$} (6,3.5) node [right] {$D_6$} (5.5,5) node [above] {$D_7$};
\end{tikzpicture}\\
where $\hat{E}, E_1, E_2, E_3$ are $(-1)$-curves and $D_1,\ldots, D_7$ are $(-2)$-curves. Note that $$\hat{E}+D_1+E_1\sim 2E_3+2D_7+2D_6+2D_5+D_4+D_3,$$ and there exists $\PP^1$-fibration $g\colon\tilde{X}\rightarrow\PP^1$ with two singular fibers $$C_1=2E_1+D_1+D_2,\quad C_2=E_2+D_4+D_5+D_6+D_7+E_3.$$ Then \begin{gather*}\varphi^*(H)\equiv L=(2+a)D_3+(3+\epsilon)D_5+(2+\epsilon)(D_4+D_6)+\\
+(1+\epsilon)(D_7+E_2)+\epsilon E_3+(3a-2\epsilon)E_1+(\frac{3}{2}a-\epsilon)D_1+(1+\frac{a}{2})D_2.
\end{gather*}
We see that for sufficiently small $\epsilon$, $L$ is effective. Then $$U\cong\tilde{X}\setminus\Supp(L)\cong\FF_2\setminus(M+F_1+F_2),$$ where $M$ is a unique $(-2)$-curve and $F_1, F_2$ are fibers. So, $X$ has an $H$-polar cylinder.

Assume that $X$ has a singular point $P$ of type $E_6$ and $\rho(X)=3$. Then $P$ is a unique singular point. So, $H\equiv-K_X+a_1E_1+a_2 E_2$. We may assume that $a_1\leq a_2$ (maybe $a_1=0$). Put $\hat{E}_1, \hat{E}_2$ are the proper transform of $E_1$ and $E_2$. We have the following configuration on $\tilde{X}$.\\
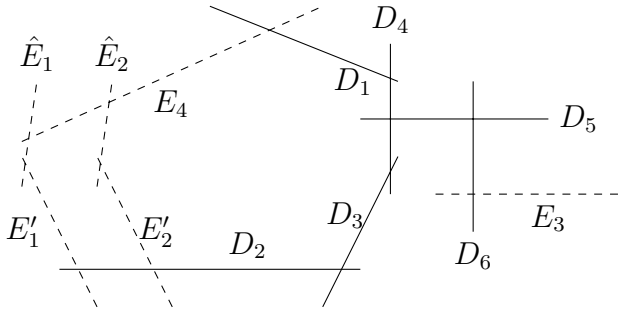
\begin{tikzpicture}
\draw  (0.5,0.5) -- (4.5,0.5); \draw [dashed] (1,0) -- (0,2) (0,1.6) -- (0.2,3) (2,0) -- (1,2) (1,1.6) -- (1.2,3) (5.5,1.5)--(8,1.5) (0,2.2) -- (4,4); \draw  (4,0) -- (5,2) (4.9,1.5) -- (4.9,3.5) (2.5,4) -- (5,3) (4.5,2.5)-- (7,2.5) (6,3) -- (6,1);
\draw (0.2,3) node [above] {$\hat{E}_1$} (1.2,3) node [above] {$\hat{E}_2$}(0.4,1) node [left] {$E'_1$} (1.4,1) node [right] {$E'_2$} (1.6,2.7) node [right] {$E_4$} (4.8,3) node [left] {$D_1$} (4.9,3.5) node [above] {$D_4$}
(3,0.5) node [above] {$D_2$} (4.3,0.9) node [above] {$D_3$} (7,2.5) node [right] {$D_5$} (6,1) node [below] {$D_6$} (7,1.5) node [below] {$E_3$};
\end{tikzpicture}\\
where $\hat{E}_1, E'_1,\hat{E}_2, E'_2, E_3, E_4$ are $(-1)$-curves and $D_1,\ldots, D_6$ are $(-2)$-curves. Note that $$\hat{E}_i+E'_i\sim 2E_3+2D_6+2D_5+2D_4+D_3+D_1\quad(i=1,2).$$ So, $$\hat{E}_i\sim  -E'_i + 2E_3+2D_6+2D_5+2D_4+D_3+D_1.$$ We see that \begin{gather*}\varphi^*(H)\equiv L=(2+a_1+a_2)D_3+(3+\epsilon)D_4+(2-a_1-a_2+\epsilon)D_1+\\
+(1-2a_1-2a_2+\epsilon)E_4+(2+\epsilon)D_5+(1+\epsilon)D_6+\epsilon E_3+\\
+(1+2a_1+2a_2-\epsilon)D_2+(a_1+2a_2-\epsilon)E'_1+(2a_1+a_2-\epsilon)E'_2.\end{gather*} Assume that $a_1+a_2\leq\frac{1}{2}$. Then $L$ is an effective divisor for sufficiently small $\epsilon$. Note that there exists $\PP^1$-fibration $g\colon\tilde{X}\rightarrow\PP^1$ with two singular fibers $$C_1=E'_1+D_2+E'_2,\quad C_2=E_4+D_1+D_4+D_5+D_6+E_3.$$ Hence, $\tilde{X}\backslash\Supp(L)\cong\FF_2\backslash(M+F_1+F_2)$, where $M$ is a unique $(-2)$-curve and $F_1, F_2$ are fibers. So, $X$ has an $H$-polar cylinder. Assume that $a_1+a_2>\frac{1}{2}$. Then \begin{gather*}\varphi^*(H)\equiv L=2D_2+(\frac{5}{2}+\epsilon)D_3+(\frac{3}{2}+\epsilon)D_1+(3+2\epsilon)D_4+\\
+(2+2\epsilon)D_5+(1+2\epsilon)D_6+2\epsilon E_3+(1-a_1+\epsilon)E'_1+\epsilon\hat{E}_1+\\
+(\frac{1}{2}+a_1-2\epsilon)E'_2+(a_1+a_1-\frac{1}{2}-2\epsilon)\hat{E}_2.\end{gather*} We see that $L$ is an effective divisor for sufficiently small $\epsilon$. Note that there exists a $\PP^1$-fibration $g\colon\tilde{X}\rightarrow\PP^1$ with three singular fibers $$C_1=E'_1+\hat{E}_1,\quad C_2=E'_1+\hat{E}_1,\quad C_3=2E_3+2D_6+2D_5+2D_4+D_1+D_3.$$ Hence, $\tilde{X}\backslash\Supp(L)\cong\FF_2\backslash(M+F_1+F_2+F_3)$, where $M$ is a unique $(-2)$-curve and $F_1, F_2, F_3$ are fibers. So, $X$ has an $H$-polar cylinder.
\end{proof}

\begin{lemma}
\label{Lemm1}
Let $X$ be a del Pezzo surface with du Val singularities and let $H$ be an ample divisor on $X$. Assume that $X$ has a singular point of type $D_m$ ($m\geq 5$). Then $X$ has an $H$-polar cylinder.
\end{lemma}

\begin{proof}
Let $\varphi\colon\tilde{X}\rightarrow X$ be the minimal resolution of singularities of $X$ and let $D=\sum_{i=1}^n D_i$ be the exceptional divisor of $\varphi$. If $n=8$ then $\rho(X)=1$ and every divisor $H\equiv a(-K_X)$. So, we may assume that $n\leq 7$.

Assume that $m=7$. Since there no exist del Pezzo surface $X$ with $\rho(X)=1$ and a singular point of type $D_7$ (see, e. g., {\cite{Fur}}, {\cite{Ma}}, {\cite{Ye}}), we see that $H$ is of type $C(r_H)$. Hence, $H\equiv-K_X+aC$, where $C$ is a $(0)$-curve and $a>0$. We have the following configuration on $\tilde{X}$.\\
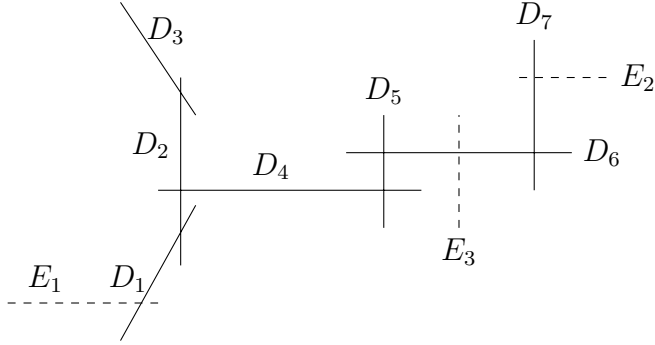
\begin{tikzpicture}
\draw  (1.5,0) -- (2.5,1.8) (2.3,1)--(2.3,3.5) (1.5,4.5)--(2.5,3) (2,2)--(5.5,2) (5,1.5)--(5,3) (4.5,2.5)--(7.5,2.5) (7,2)--(7,4); \draw [dashed] (0,0.5) -- (2,0.5) (6.8,3.5)--(8,3.5) (6,1.5)--(6,3);
\draw (0.5,0.5) node [above] {$E_1$} (1.6,0.5) node [above] {$D_1$} (2.3,2.6) node [left] {$D_2$} (2.1,3.8) node [above] {$D_3$} (3.5,2) node [above] {$D_4$} (5,3) node [above] {$D_5$} (7.5,2.5) node [right] {$D_6$} (7,4) node [above] {$D_7$} (8,3.5) node [right] {$E_2$} (6,1.5) node [below] {$E_3$};
\end{tikzpicture}\\
where $E_1,E_2,E_3$ are $(-1)$-curves, $D_1,\ldots, D_2$ are $(-2)$-curves. We have $$\tilde{C}\sim R= 2E_2+2D_7+2D_6+2D_5+2D_4+2D_2+D_1+D_3,$$ where $\tilde{C}$ is the proper transform of $C$ on $\tilde{X}$. We see that
\begin{gather*}\varphi^*(H)\equiv L= (2+2a)D_5+(2+2a+\epsilon)D_4+(1+a+\epsilon)D_3+\\
+(2+2a+2\epsilon)D_2+(1+a+2\epsilon)D_1
+2\epsilon E_1+(2+2a-\epsilon)D_6+\\
+(1+2a-\epsilon)D_7+(1-\epsilon)E_3+(2a-\epsilon) E_2.\end{gather*}
So, $L$ is an effective divisor for sufficiently small $\epsilon$. Note that there exists a $\PP^1$-fibration $g\colon\tilde{X}\rightarrow\PP^1$ with two singular fibers $$C_1=E_3+D_6+D_7+E_2,\quad C_2=2E_1+2D_1+2D_2+D_3+D_4.$$ Hence, $\tilde{X}\backslash\Supp(L)\cong\FF_2\backslash(M+F_1+F_2)$, where $M$ is a unique $(-2)$-curve and $F_1, F_2$ are fibers. So, $X$ has an $H$-polar cylinder.

Assume that $m=6$. Since $n\leq 7$, we see that there exist at most one singular point of type $A_1$. Assume that there exist a singular point of type $A_1$. Then $H$ is of type $B(r_H)$ and $H\equiv-K_X+aE$, where $E$ is a $(-1)$-curve and $0<a<1$. We have the following configuration on $\tilde{X}$.\\

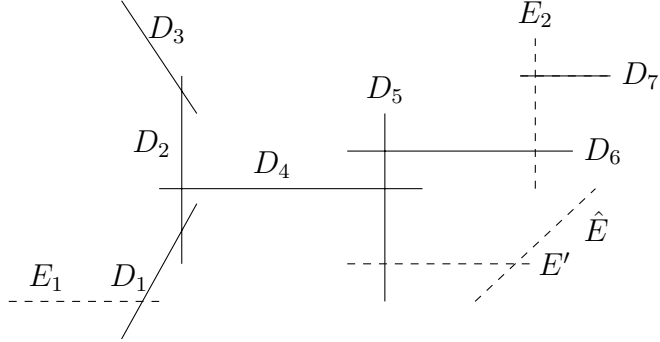
\begin{tikzpicture}
\draw  (1.5,0) -- (2.5,1.8) (2.3,1)--(2.3,3.5) (1.5,4.5)--(2.5,3) (2,2)--(5.5,2) (5,0.5)--(5,3) (4.5,2.5)--(7.5,2.5) (6.8,3.5)--(8,3.5); \draw [dashed] (0,0.5) -- (2,0.5) (6.8,3.5)--(8,3.5)  (7,2)--(7,4) (4.5,1) -- (7,1) (6.2,0.5)--(7.8,2);
\draw (0.5,0.5) node [above] {$E_1$} (1.6,0.5) node [above] {$D_1$} (2.3,2.6) node [left] {$D_2$} (2.1,3.8) node [above] {$D_3$} (3.5,2) node [above] {$D_4$} (5,3) node [above] {$D_5$} (7.5,2.5) node [right] {$D_6$} (7,4) node [above] {$E_2$} (8,3.5) node [right] {$D_7$} (6.9,1) node [right] {$E'$} (7.5,1.5) node [right] {$\hat{E}$};
\end{tikzpicture}\\
where $\hat{E}$ is the proper transform of $E$ on $\tilde{X}$, $\hat{E}, E', E_1, E_2$ are $(-1)$-curves $D_1,\ldots,D_6$ are $(-2)$-curves. Note that \begin{gather*}\varphi^*(H)\equiv L= 2D_5+ (2+\epsilon)D_4+(1+\epsilon)D_3+(2+2\epsilon)D_2+(1+2\epsilon)D_1+\\
+2\epsilon E_1+(1+\epsilon)D_6+\epsilon D_7+2\epsilon E_2+(1-2\epsilon)E'+(a-2\epsilon)\hat{E}.\end{gather*}
So, $L$ is an effective divisor for sufficiently small $\epsilon$. Note that there exists a $\PP^1$-fibration $g\colon\tilde{X}\rightarrow\PP^1$ with three singular fibers $$C_1=2E_2+D_6+D_7,\quad C_2=2E_1+2D_1+2D_2+D_3+D_4,\quad C_3=E'+\hat{E}.$$ Hence, $\tilde{X}\backslash\Supp(L)\cong\FF_2\backslash(M+F_1+F_2+F_3)$, where $M$ is a unique $(-2)$-curve and $F_1, F_2,F_3$ are fibers. So, $X$ has an $H$-polar cylinder.

Assume that $m=n=6$. Then $H\equiv-K_X+aC+bE$, where $C$ is a $(0)$-curve, $E$ is a $(-1)$-curve and $a\geq 0, 1>b\geq 0$. We have the following configuration on $\tilde{X}$.\\

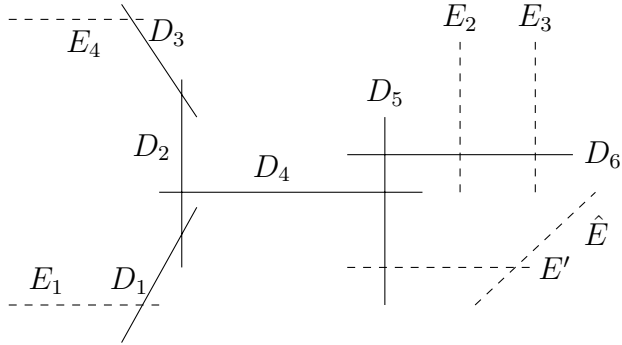
\begin{tikzpicture}
\draw  (1.5,0) -- (2.5,1.8) (2.3,1)--(2.3,3.5) (1.5,4.5)--(2.5,3) (2,2)--(5.5,2) (5,0.5)--(5,3) (4.5,2.5)--(7.5,2.5); \draw [dashed] (0,0.5) -- (2,0.5) (0,4.3)-- (1.8,4.3) (7,2)--(7,4) (6,2) -- (6,4) (4.5,1) -- (7,1) (6.2,0.5)--(7.8,2);
\draw (0.5,0.5) node [above] {$E_1$} (1.6,0.5) node [above] {$D_1$} (2.3,2.6) node [left] {$D_2$} (2.1,3.8) node [above] {$D_3$} (3.5,2) node [above] {$D_4$} (5,3) node [above] {$D_5$} (7.5,2.5) node [right] {$D_6$} (6,4) node [above] {$E_2$} (7,4) node [above] {$E_3$}  (6.9,1) node [right] {$E'$} (7.5,1.5) node [right] {$\hat{E}$} (1,4.3) node [below] {$E_4$};
\end{tikzpicture}\\
where $\hat{E}$ is the proper transform of $E$ on $\tilde{X}$, $\hat{E}, E', E_1,\ldots, E_4$ are $(-1)$-curves $D_1,\ldots,D_6$ are $(-2)$-curves. We have $$\tilde{C}\sim R= 2E_2+2D_6+2D_5+2D_4+2D_2+D_1+D_3,$$ where $\tilde{C}$ is the proper transform of $C$ on $\tilde{X}$. Assume that $b>0$. We see that
\begin{gather*}\varphi^*(H)\equiv L= (2+2a)D_5+(2+2a+\epsilon)D_4+(1+a+\epsilon)D_3+\\+(2+2a+2\epsilon)D_2+(1+a+2\epsilon)D_1+\\
+2\epsilon E_1+(1+2a+\epsilon)D_6+(2a+\epsilon) E_2+\epsilon E_3+(1-2\epsilon) E'+(b-2\epsilon)\hat{E}.\end{gather*}
So, $L$ is an effective divisor for sufficiently small $\epsilon$. Note that there exists a $\PP^1$-fibration $g\colon\tilde{X}\rightarrow\PP^1$ with three singular fibers $$C_1=2E_1+2D_1+2D_2+D_3+D_4,\quad C_2=E_2+D_6+E_3,\quad C_3=E'+\hat{E}.$$ Hence, $\tilde{X}\backslash\Supp(L)\cong\FF_2\backslash(M+F_1+F_2+F_3)$, where $M$ is a unique $(-2)$-curve and $F_1, F_2,F_3$ are fibers. So, $X$ has an $H$-polar cylinder. Assume that $b=0$. Then \begin{gather*}\varphi^*(H)\equiv L= (2+2a)D_4+\epsilon E_4+ (1+a+\epsilon)D_3+(2+2a+\epsilon)D_2+\\
+(1+a+\epsilon)D_1+\epsilon E_3+\\
+(2+2a-\epsilon)D_5+(1+2a-\epsilon)D_6+(1-\epsilon)E'+(2a-\epsilon) E_2.\end{gather*}
Then $L$ is an effective divisor for sufficiently small $\epsilon$. Note that there exists a $\PP^1$-fibration $g\colon\tilde{X}\rightarrow\PP^1$ with two singular fibers $$C_1=E'+D_5+D_6+E_2,\quad C_2=E_1+D_1+D_2+D_3+E_4.$$ Hence, $\tilde{X}\backslash\Supp(L)\cong\FF_2\backslash(M+F_1+F_2)$, where $M$ is a unique $(-2)$-curve and $F_1, F_2$ are fibers. So, $X$ has an $H$-polar cylinder.

Assume that $m=5$ and $n=7$. So, we have two cases for singularities $D_5+A_2$ and $D_5+2A_1$. Consider the case $D_5+A_2$. Then $H$ is of type $B(r_H)$. Hence, $H\equiv-K_X+a E$, where $E$ ia a $(-1)$-curve and $0<a<3$. Moreover, $E$ contains singular point of type $A_2$. So, we have the following configuration on $\tilde{X}$.\\

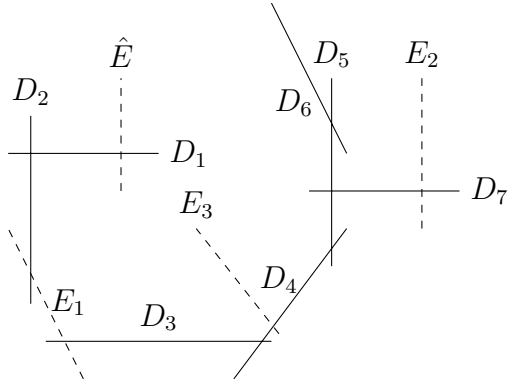
\begin{tikzpicture}
\draw  (0.5,0.5) -- (3.5,0.5) (3,0)--(4.5,2) (4.3,1.5)--(4.3,4) (3.5,5)--(4.5,3) (4,2.5) -- (6,2.5) (0.3,1)--(0.3,3.5) (0,3)-- (2,3); \draw [dashed] (1,0) -- (0,2) (5.5,2) -- (5.5,4) (1.5,2.5)--(1.5,4) (3.6,0.6)--(2.5,2);
\draw (2,0.5) node [above] {$D_3$} (2,3) node [right] {$D_1$} (0.3,3.5) node [above] {$D_2$} (1.5,4) node [above] {$\hat{E}$} (3.6,1) node [above] {$D_4$} (4.3,4) node [above] {$D_5$} (3.8,4) node [below] {$D_6$} (6,2.5)  node [right] {$D_7$} (5.5,4) node [above] {$E_2$} (0.8,0.7) node [above] {$E_1$} (2.5,2) node [above] {$E_3$};
\end{tikzpicture}\\
where $\hat{E}$ is the proper transform of $E$ on $\tilde{X}$, $\hat{E}, E_1, E_2$ are $(-1)$-curves $D_1,\ldots,D_7$ are $(-2)$-curves. Moreover, $\varphi^{*}(E)=\hat{E}+\frac{2}{3}D_2+\frac{1}{3}D_1$. We see that $$\hat{E}+D_1+D_2+E_1\sim 2E_2+2D_7+2D_5+D_6+D_4.$$ So, $$\varphi^{*}(E)\equiv-E_1-\frac{2}{3}D_1-\frac{1}{3}D_2+2E_2+2D_7+2D_5+D_6+D_4.$$ We have \begin{gather*}\varphi^*(H)\equiv L=x D_3+(x-\frac{1}{3}-\frac{1}{3}a)D_1+(2x-\frac{2}{3}-\frac{2}{3}a)D_2+\\
+(3x-1-a)E_1+(y+a)D_4+(y-1)E_3+(y-1+a)D_6+\\
+(2y-2+2a)D_5+(2y-3+2a)D_7+(2y-4+2a)E_2,\end{gather*}
where $x+y=3$. So, there exist $x,y$ such that $L$ is an effective divisor. Let $h\colon\bar{X}\rightarrow\tilde{X}$ be the blow-ups the point of intersection $D_3$ and $D_4$ and let $\bar{E}_1,\bar{E}_2,\bar{E}_3,\bar{D}_1,\ldots,\bar{D}_7$ be the proper transform of $E_1,E_2,E_3,D_1,\ldots,D_7$ correspondingly and $F$ is an exceptional $(-1)$-curve. Note that there exists a $\PP^1$-fibration $g\colon\bar{X}\rightarrow\PP^1$ with two singular fibers $$C_1=3\bar{E}_1+2\bar{D}_2+\bar{D}_1+\bar{D}_3,\quad C_2=\bar{E}_3+\bar{D}_4+2\bar{D}_5+\bar{D}_6+2\bar{D}_7+2\bar{E}_2.$$ Hence, $\tilde{X}\backslash\Supp(L)\cong\FF_1\backslash(M+F_1+F_2)$, where $M$ is a unique $(-1)$-curve and $F_1, F_2$ are fibers. So, $X$ has an $H$-polar cylinder.

Consider the case $D_5+2A_1$. Then $H$ is  of type $C(r_H)$. Hence, $H\equiv-K_X+aC$, where $C$ ia a $(0)$-curve and $a>0$. Moreover, there exist a curve $C'\sim C$ such that $C'$ contain both singular point of type $A_1$. So, we have the following configuration on $\tilde{X}$.\\

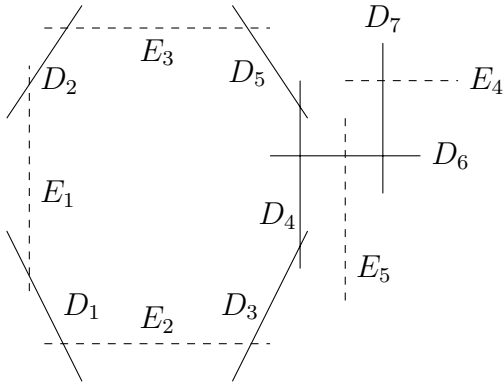
\begin{tikzpicture}
\draw  (1,0) -- (0,2) (0,3.5)--(1,5) (3,0)--(4,2) (3.9,1.5)--(3.9,4) (4,3.5)--(3,5) (3.5,3)--(5.5,3) (5,2.5)--(5,4.5); \draw [dashed] (0.5,0.5) -- (3.5,0.5) (0.3,1.2) -- (0.3,4.2) (4.5,4)--(6,4) (0.5,4.7)--(3.5,4.7) (4.5,3.5)-- (4.5,1);
\draw (2,0.5) node [above] {$E_2$} (0.6,1) node [right] {$D_1$} (0.3,2.5) node [right] {$E_1$} (0.3,4) node [right] {$D_2$} (2,4.7) node [below] {$E_3$} (3.5,1) node [left] {$D_3$} (3.6,4.1) node [left] {$D_5$} (4,2.2) node [left] {$D_4$} (5.5,3) node [right] {$D_6$} (5,4.5) node [above] {$D_7$} (6,4) node [right] {$E_4$} (4.5,1.5) node [right] {$E_5$};
\end{tikzpicture}\\
where $E_1, E_2, E_3, E_4$ are $(-1)$-curves $D_1,\ldots,D_7$ are $(-2)$-curves. Moreover, $$\varphi^{*}(C)\sim R=2E_4+2D_7+2D_6+2D_4+D_3+D_5.$$
We have
\begin{gather*}\varphi^*(H)\equiv L= (2+2a)D_4+\epsilon D_1+(1+a+\epsilon)D_3+2\epsilon E_2+\epsilon D_2+\\
+(1+a+\epsilon)D_5+2\epsilon E_3+(2+2a-2\epsilon)D_6+(1+2a-2\epsilon)D_7+\\
+(1-2\epsilon)E_5+(2a-2\epsilon) E_4.\end{gather*}
So, $L$ is an effective divisor for sufficiently small $\epsilon$. Note that there exists a $\PP^1$-fibration $g\colon\tilde{X}\rightarrow\PP^1$ with three singular fibers $$C_1=2E_2+D_1+D_3,\quad C_2=2E_3+D_2+D_5,\quad C_3=E_5+D_6+D_7+E_4.$$ Hence, $\tilde{X}\backslash\Supp(L)\cong\FF_2\backslash(M+F_1+F_2+F_3)$, where $M$ is a unique $(-2)$-curve and $F_1, F_2, F_3$ are fibers. So, $X$ has an $H$-polar cylinder.

Assume that $m=5$ and $n=6$. So, we have two singular points $P$ and $Q$ on $X$. Moreover, $P$ is of type $D_5$ and $Q$ is of type $A_1$. Since $\deg(X)=1$, we see that $\rho(X)=3$. Assume that $H$ is of type $C(r_H)$. Then $H\equiv-K_X+aC+bE$, where $C$ ia a $(0)$-curve, $E$ ia a $(-1)$-curve and $a>0, 2>b\geq 0$. Moreover, $E$ contains a singular point $Q$. So, we have the following configuration on $\tilde{X}$.\\

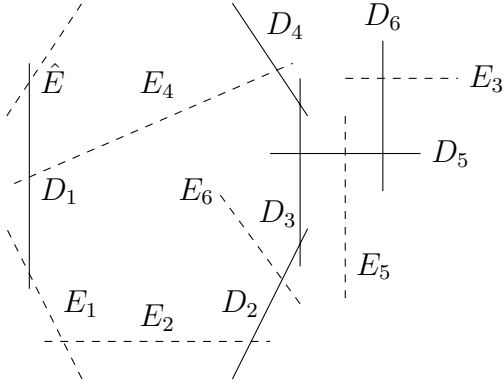
\begin{tikzpicture}
\draw (0.3,1.2) -- (0.3,4.2)  (3,0)--(4,2) (3.9,1.5)--(3.9,4) (4,3.5)--(3,5) (3.5,3)--(5.5,3) (5,2.5)--(5,4.5); \draw [dashed] (0.5,0.5) -- (3.5,0.5) (1,0) -- (0,2) (0,3.5)--(1,5) (4.5,4)--(6,4) (0.1,2.6)--(3.8,4.2) (4.5,3.5)-- (4.5,1) (3.9,1)-- (2.8,2.5);
\draw (2,0.5) node [above] {$E_2$} (0.6,1) node [right] {$E_1$} (0.3,2.5) node [right] {$D_1$} (0.3,4) node [right] {$\hat{E}$} (3.5,1) node [left] {$D_2$} (3.3,4.7) node [right] {$D_4$} (4,2.2) node [left] {$D_3$} (5.5,3) node [right] {$D_5$} (5,4.5) node [above] {$D_6$} (6,4) node [right] {$E_3$} (2,3.6) node [above] {$E_4$} (4.5,1.5) node [right] {$E_5$} (2.9,2.5) node [left] {$E_6$};
\end{tikzpicture}\\
where $\hat{E}$ is the proper transform of $E$ on $\tilde{X}$, $\hat{E}, E_1, \ldots, E_6$ are $(-1)$-curves $D_1,\ldots,D_6$ are $(-2)$-curves. Moreover, $\varphi^{*}(E)\equiv\hat{E}+\frac{1}{2}D_1$ and $$\varphi^{*}(C)\equiv R_1=2E_3+2D_6+2D_5+2D_3+D_4+D_2.$$ Assume that $b\geq 1$. Then \begin{gather*}\varphi^*(H)\equiv L= (2+a)D_4+(1+\epsilon)E_4+(\frac{b}{2}+\epsilon) D_1+(b-1+\epsilon)\hat{E}+\\
+(3+2a-\epsilon)D_3+(2+a-\epsilon)D_2+(1-\epsilon)E_6+(2+2a-\epsilon)D_5+\\
+(1+2a-\epsilon)D_6+(2a-\epsilon)E_3.\end{gather*} So, $L$ is an effective divisor for sufficiently small $\epsilon$. Note that there exists a $\PP^1$-fibration $g\colon\tilde{X}\rightarrow\PP^1$ with two singular fibers $$C_1=E_4+D_1+\hat{E},\quad C_2=E_3+D_6+D_5+D_3+D_2+E_6.$$ Hence, $\tilde{X}\backslash\Supp(L)\cong\FF_2\backslash(M+F_1+F_2)$, where $M$ is a unique $(-2)$-curve and $F_1, F_2$ are fibers. So, $X$ has an $H$-polar cylinder. So, we may assume that $b<1$. Note that $$\hat{E}+D_1+E_4\sim E_2+D_2+D_3+D_5+D_6+E_3.$$ So, $$\varphi^{*}(E)\equiv R_2=-E_4-\frac{1}{2}D_1+E_2+D_2+D_3+D_5+D_6+E_3.$$ We have
\begin{gather*}\varphi^*(H)\equiv L= (2+2a)D_3+(b+\epsilon) E_2+\epsilon E_6+(1+a+b+\epsilon)D_2+\\
+(\epsilon-\frac{b}{2}) D_1+(2\epsilon-b) E_4+(1+a+\epsilon)D_4+(1-2\epsilon)E_5+\\
+(2+2a+b-2\epsilon)D_5+(1+2a+b-2\epsilon)D_6+(2a+b-2\epsilon) E_3.\end{gather*} We see that $L$ is an effective divisor for $a+\frac{b}{2}>\epsilon>\frac{b}{2}$. Note that there exists a $\PP^1$-fibration $g\colon\tilde{X}\rightarrow\PP^1$ with three singular fibers $$C_1=2E_4+D_1+D_4,\quad C_2=E_2+D_2+E_6,\quad C_3=E_5+D_5+D_6+E_3.$$ Hence, $\tilde{X}\backslash\Supp(L)\cong\FF_2\backslash(M+F_1+F_2+F_3)$, where $M$ is a unique $(-2)$-curve and $F_1, F_2, F_3$ are fibers. So, $X$ has an $H$-polar cylinder.

Assume that $H$ is of type $B(r_H)$. Then $$H\equiv-K_X+a_1E_1+a_2E_2,$$ where $E_1,E_2$ are $(-1)$-curves. As above, we may assume that $E_1$ contains a singular point $Q$ and $2>a_1\geq 0$, $1>a_2\geq 0$. So, we have the following configuration on $\tilde{X}$.\\

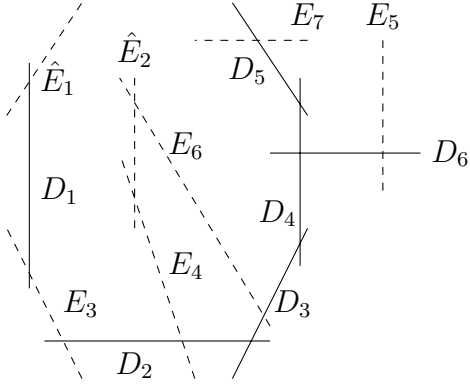
\begin{tikzpicture}
\draw (0.5,0.5) -- (3.5,0.5) (0.3,1.2) -- (0.3,4.2) (3,0)--(4,2) (3.9,1.5)--(3.9,4) (4,3.5)--(3,5) (3.5,3)--(5.5,3); \draw [dashed] (1,0) -- (0,2) (0,3.5)--(1,5) (5,2.5)--(5,4.5) (2.5,0)--(1.5,3) (1.7,2)--(1.7,4) (3.5,0.7)--(1.5,4) (4,4.5)--(2.5,4.5);
\draw (1.7,0.5) node [below] {$D_2$} (0.6,1) node [right] {$E_3$} (0.3,2.5) node [right] {$D_1$} (0.3,4) node [right] {$\hat{E}_1$} (4.2,1) node [left] {$D_3$} (3.6,4.1) node [left] {$D_5$} (4,2.2) node [left] {$D_4$} (5.5,3) node [right] {$D_6$} (5,4.5) node [above] {$E_5$} (1.7,4) node [above] {$\hat{E}_2$} (2,1.5) node [right] {$E_4$} (2,3.1) node [right] {$E_6$} (4,4.5) node [above] {$E_7$};
\end{tikzpicture}\\
where $\hat{E}_1$ and $\hat{E}_2$ are the proper transform of $E_1$ and $E_2$ on $\tilde{X}$ correspondingly, $\hat{E}_1, \hat{E}_2, E_3, \ldots, E_7$ are $(-1)$-curves $D_1,\ldots,D_6$ are $(-2)$-curves. Moreover, $\varphi^{*}(E_1)\equiv\hat{E}_1+\frac{1}{2}D_1$. Since
$$\hat{E}_1+D_1+E_3\sim \hat{E}_2+E_4\sim D_3+D_5+2D_4+2D_6+2E_5,$$ we see that $\varphi^{*}(E_1)\equiv -E_3-\frac{1}{2}D_1+D_3+D_5+2D_4+2D_6+2E_5$ and $\varphi^{*}(E_2)\sim-E_4+D_3+D_5+2D_4+2D_6+2E_5$.  Then \begin{gather*}\varphi^*(H)\equiv L= x D_2+(x-1-a_2)E_4+\\
+(x-1-\frac{a_1}{2})D_1+(2x-2-a_1)E_3+(y+a_1+a_2)D_3+(y-1)E_6+\\
+(y-1+a_1+a_2)D_5+(2y-2+2a_1+2a_2)D_4+\\
+(2y-3+2a_1+2a_2)D_6+(2y-4+2a_1+2a_2)E_5,\end{gather*} where $x+y=3$. Since $a_1>0$ or $a_2>0$, we see that there exist $x$ and $y$ such that $L$ is an effective divisor. Let $h\colon\bar{X}\rightarrow\tilde{X}$ be the blow-ups the point of intersection $D_2$ and $D_3$ and let $\bar{E}_1,\ldots, \bar{E}_6,\bar{D}_1,\ldots,\bar{D}_6$ be the proper transform of $\hat{E}_1,\hat{E}_2,E_3,\ldots,E_6,D_1,\ldots,D_6$ correspondingly. Note that there exists a $\PP^1$-fibration $g\colon\bar{X}\rightarrow\PP^1$ with two singular fibers $$C_1=2\bar{E}_3+\bar{D}_1+\bar{D}_2+\bar{E}_4,\quad C_2=\bar{E}_5+2\bar{D}_6+2\bar{D}_4+\bar{D}_5+\bar{D}_3+\bar{E}_6.$$ Hence, $\tilde{X}\backslash\Supp(L)\cong\FF_1\backslash(M+F_1+F_2)$, where $M$ is a unique $(-1)$-curve and $F_1, F_2$ are fibers. So, $X$ has an $H$-polar cylinder.

Assume that $m=5$ and $n=5$. Then we have only one singular point of type $D_5$. Since $\deg(X)=1$, we see that $\rho(X)=4$. Assume that $H$ is of type $C(r_H)$ and $H\equiv-K_X+aC$, where $C$ is a $(0)$-curve. So, we have the following configuration on $\tilde{X}$.\\

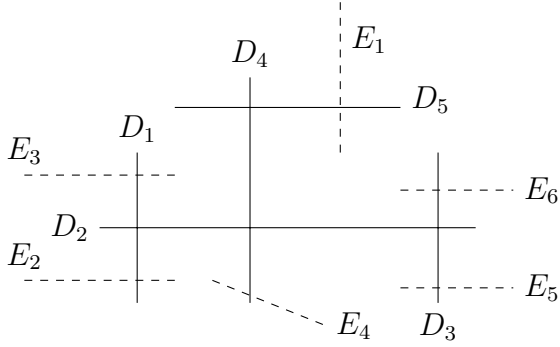
\begin{tikzpicture}
\draw  (1,2) -- (6,2) (1.5,1)--(1.5,3) (3,1)--(3,4) (2,3.6)--(5,3.6) (5.5,1)--(5.5,3); \draw [dashed] (0,1.3) -- (2,1.3) (0,2.7)--(2,2.7) (2.5,1.3)--(4,0.7) (4.2,3)--(4.2,5) (5,2.5)--(6.5,2.5) (5,1.2)--(6.5,1.2);
\draw (1,2) node [left] {$D_2$} (0,1.3) node [above] {$E_2$} (0,2.7) node [above] {$E_3$} (4,0.7) node [right] {$E_4$} (1.5,3) node [above] {$D_1$} (4.2,4.5) node [right] {$E_1$} (5,3.7) node [right] {$D_5$} (3,4) node [above] {$D_4$} (5.5,1) node [below] {$D_3$} (6.5,2.5)  node [right] {$E_6$} (6.5,1.2) node [right] {$E_5$};
\end{tikzpicture}\\
where $\varphi^{*}(C)\equiv R=2E_1+2D_5+2D_4+2D_2+D_1+D_3.$ We have \begin{gather*}\varphi^*(H)\equiv L= (2+2a)D_2+\epsilon E_2+\epsilon E_3+(1+a+\epsilon)D_1+\epsilon E_5+\epsilon E_6+\\
+(1+a+\epsilon)D_3+(1-2\epsilon) E_4+(2+2a-2\epsilon)D_4+\\
+(1+2a-2\epsilon)D_5+(2a-2\epsilon) E_1\end{gather*} So, $L$ is an effective divisor for a sufficiently small $\epsilon$. Note that there exists a $\PP^1$-fibration $g\colon\tilde{X}\rightarrow\PP^1$ with three singular fibers $$C_1=E_2+D_1+E_3,\quad C_2=E_1+D_5+D_4+E_4,\quad C_3=E_5+D_3+E_6.$$ Hence, $\tilde{X}\backslash\Supp(L)\cong\FF_2\backslash(M+F_1+F_2+F_3)$, where $M$ is a unique $(-2)$-curve and $F_1, F_2,F_3$ are fibers. So, $X$ has an $H$-polar cylinder.

Assume that $H\equiv-K_X+aC+bE_1$ where $C$ is a $(0)$-curve, $E_1$ is a $(-1)$-curves, $a\geq 0$ and $1>b>0$. Then we have the following configuration on $\tilde{X}$.\\
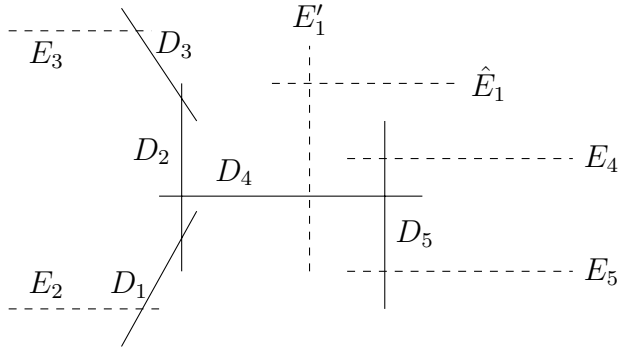
\begin{tikzpicture}
\draw  (1.5,0) -- (2.5,1.8) (2.3,1)--(2.3,3.5) (1.5,4.5)--(2.5,3) (2,2)--(5.5,2) (5,0.5)--(5,3); \draw [dashed] (0,0.5) -- (2,0.5) (4.5,2.5)--(7.5,2.5) (0,4.2) -- (1.9,4.2) (4,1)--(4,4) (3.5,3.5)--(6,3.5)(4.5,1)--(7.5,1);
\draw (0.5,0.5) node [above] {$E_2$} (1.6,0.5) node [above] {$D_1$} (2.3,2.6) node [left] {$D_2$} (2.2,3.7) node [above] {$D_3$} (3,2) node [above] {$D_4$} (5,1.5) node [right] {$D_5$} (7.5,2.5) node [right] {$E_4$} (6,3.5) node [right] {$\hat{E}_1$} (4,4) node [above] {$E'_1$} (0.5,4.2) node [below] {$E_3$} (7.5,1) node [right] {$E_5$};
\end{tikzpicture}\\
where $\hat{E}_1$ is the proper transform of $E_1$ on $\tilde{X}$, $\hat{E}_1, E'_1,E_2, E_3, E_4,E_5$ are $(-1)$-curves, $D_1,\ldots,D_5$ are $(-2)$-curves, $$\varphi^{*}(C)\equiv R=2E_5+2D_5+2D_4+2D_2+D_1+D_3.$$ We have \begin{gather*}\varphi^*(H)\equiv L= (2+2a)D_4+\epsilon E_2+\epsilon E_3+(1+a+\epsilon)D_1+\\+(1+a+\epsilon)D_3+(2+2a+\epsilon)D_2+\epsilon E_4+(2a+\epsilon) E_5+\\
+(1+2a+\epsilon)D_5+(1-2\epsilon)E'_1+(b-2\epsilon)\hat{E}_1\end{gather*} So, $L$ is an effective divisor for a sufficiently small $\epsilon$. Note that there exists a $\PP^1$-fibration $g\colon\tilde{X}\rightarrow\PP^1$ with three singular fibers $$C_1=E_2+D_1+D_2+D_3+E_3,\quad C_2=E_4+D_5+E_5,\quad C_3=E'_1+\hat{E}_1.$$ Hence, $\tilde{X}\backslash\Supp(L)\cong\FF_2\backslash(M+F_1+F_2+F_3)$, where $M$ is a unique $(-2)$-curve and $F_1, F_2,F_3$ are fibers. So, $X$ has an $H$-polar cylinder.

Assume that $H\equiv-K_X+aC+b_1E_1+b_2E_2$ where $C$ is a $(0)$-curve, $E_1, E_2$ are $(-1)$-curves, $a\geq 0$ and $1>b_2\geq b_1>0$. Then we have the following configuration on $\tilde{X}$.\\
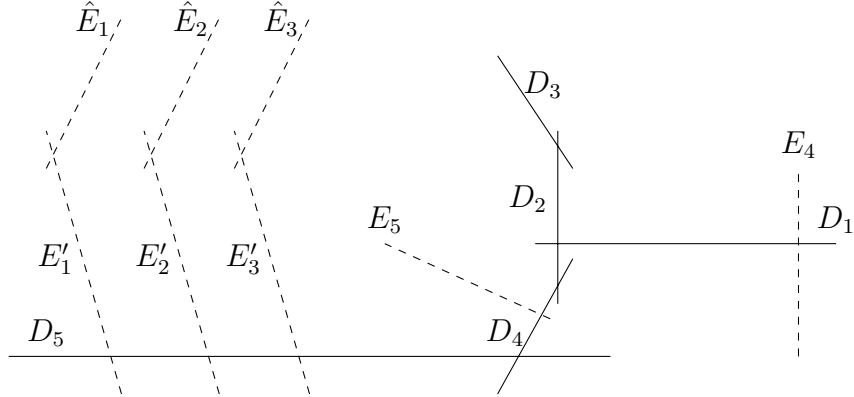
\begin{tikzpicture}
\draw  (0,0.5) -- (8,0.5) (6.5,0) -- (7.5,1.8) (7.3,1.2)--(7.3,3.5) (6.5,4.5)--(7.5,3) (7,2)--(11,2); \draw [dashed] (10.5,0.5)--(10.5,3) (1.5,0)--(0.5,3.5) (0.5,3)--(1.5,5) (2.8,0)--(1.8,3.5) (1.8,3)--(2.8,5) (4,0)--(3,3.5) (3,3)--(4,5) (7.2,1)--(5,2);
\draw (0.5,0.5) node [above] {$D_5$} (6.6,0.5) node [above] {$D_4$} (7.3,2.6) node [left] {$D_2$} (7.1,3.8) node [above] {$D_3$} (11,2) node [above] {$D_1$} (10.5,3) node [above] {$E_4$} (1,1.8) node [left] {$E'_1$} (2.3,1.8) node [left] {$E'_2$} (3.5,1.8) node [left] {$E'_3$} (0.5,0.5) (1.5,5) node [left] {$\hat{E}_1$} (2.8,5) node [left] {$\hat{E}_2$} (4,5) node [left] {$\hat{E}_3$} (5,2) node [above] {$E_5$};
\end{tikzpicture}\\
where $\hat{E}_1, \hat{E}_2$ are the proper transform of $E_1,E_2$ on $\tilde{X}$ correspondingly, $\hat{E}_1, \hat{E}_2, \hat{E}_3, E'_1,E'_2, E'_3, E_4, E_5$ are $(-1)$-curves, $D_1,\ldots,D_5$ are $(-2)$-curves, $\varphi^{*}(C)\equiv R=2E'_3+2D_5+2D_4+2D_2+D_1+D_3.$ We have $$\hat{E}_i\sim -E'_i+D_4+D_3+2D_2+2D_1+2E_4$$ for $i=1,2,3$. We have \begin{gather*}\varphi^*(H)\equiv L= (x+2a)D_5+(x-1-b_1)E'_1+(x-1-b_2)E'_2+\\
+(x-1+a)E'_3+(y+2a+b_1+b_2)D_4+(y-1)E_5+\\
+(y-1+a+b_1+b_2)D_3+(2y-2+2a+2b_1+2b_2)D_2+\\
+(2y-3+a+2b_1+2b_2)D_1+(2y-4+2b_1+2b_2)E_4,\end{gather*} where $x+y=3$. Since $b_1>0$ and $b_2>0$ we see that there exist $x$ and $y$ such that $L$ is an effective divisor. Let $h\colon\bar{X}\rightarrow\tilde{X}$ be the blow-ups the point of intersection $D_4$ and $D_5$ and let $\bar{E}_1,\ldots, \bar{E}_5,\bar{D}_1,\ldots,\bar{D}_5$ be the proper transform of $E'_1,E'_2,E'_3,E_4,E_5,D_1,\ldots,D_5$ correspondingly. Note that there exists a $\PP^1$-fibration $g\colon\bar{X}\rightarrow\PP^1$ with two singular fibers $$C_1=\bar{D}_5+\bar{E}_1+\bar{E}_2+\bar{E}_3,\quad C_2=2\bar{E}_4+2\bar{D}_1+2\bar{D}_2+\bar{D}_4+\bar{D}_3+\bar{E}_5.$$ Hence, $\tilde{X}\backslash\Supp(L)\cong\FF_1\backslash(M+F_1+F_2)$, where $M$ is a unique $(-1)$-curve and $F_1, F_2$ are fibers. So, $X$ has an $H$-polar cylinder.

Assume that $H\equiv-K_X+b_1E_1+b_2E_2+b_3 E_3$ where  $E_1, E_2, E_3$ are $(-1)$-curves, $1>b_1\geq b_2\geq b_3\geq 0$. We may assume that at least $b_1>0$ and $b_2>0$. Put $\hat{E}_1, \hat{E}_2, \hat{E}_3$ are the proper transform of $E_1,E_2,E_3$ on $\tilde{X}$ correspondingly. Put $b=b_1+b_2+b_3$ Then \begin{gather*}\varphi^*(H)\equiv L= xD_5+(x-1-b_1)E'_1+(x-1-b_2)E'_2+\\
+(x-1-b_3)E'_3+(y+b)D_4+(y-1)E_5+(y-1+a+b)D_3+\\
+(2y-2+2a+2b)D_2+(2y-3+a+2b)D_1+(2y-4+2b)E_4,\end{gather*} where $x+y=3$. Since $b_1>0$ and $b_2>0$ we see that there exist $x$ and $y$ such that $L$ is an effective divisor. As above, $$\tilde{X}\backslash\Supp(L)\cong\FF_1\backslash(M+F_1+F_2),$$ where $M$ is a unique $(-1)$-curve and $F_1, F_2$ are fibers. So, $X$ has an $H$-polar cylinder.
\end{proof}

\begin{lemma}
\label{Lemm3}
Let $X$ be a del Pezzo surface with du Val singularities and let $H$ be an ample divisor on $X$. Assume that $X$ has a singular point of type $A_7$. Then $X$ has an $H$-polar cylinder.
\end{lemma}

\begin{proof}
Let $\varphi\colon\tilde{X}\rightarrow X$ be the minimal resolution of singularities of $X$ and let $D=\sum_{i=1}^7 D_i$ be the exceptional divisor of $\varphi$.
Assume that $H=-K_X+a E$, where $E$ a $(-1)$-curve and $E$ does not contain $P$ and $a<1$. So, we have the following configuration on $\tilde{X}$.\\
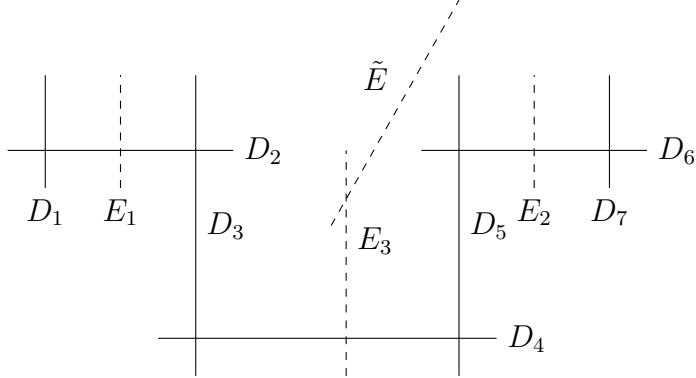
\begin{tikzpicture}
\draw  (3.5,0.5) -- (8,0.5) (4,0) -- (4,4) (1.5,3)--(4.5,3) (2,2.5)--(2,4) (7.5,0)--(7.5,4) (7,3)--(10,3) (9.5,2.5)--(9.5,4); \draw [dashed] (3,2.5) -- (3,4) (8.5,2.5) -- (8.5,4) (6,0) --(6,3) (5.8,2)--(7.5,5);
\draw (8,0.5) node [right] {$D_4$} (2,2.5) node [below] {$D_1$} (4.5,3) node [right] {$D_2$} (4,2) node [right] {$D_3$} (3,2.5) node [below] {$E_1$} (9.5,2.5) node [below] {$D_7$} (10,3) node [right] {$D_6$}
(8.5,2.5) node [below] {$E_2$} (7.5,2) node [right] {$D_5$} (6,1.8) node [right] {$E_3$} (6.7,4) node [left] {$\tilde{E}$};
\end{tikzpicture}\\
where $\tilde{E}$ is the proper transform of $E$ on $\tilde{X}$, $\tilde{E},  E_1, E_2, E_3$ are $(-1)$-curves $D_1,\ldots,D_7$ are $(-2)$-curves. We have \begin{gather*}\varphi^*(H)\equiv L=
2D_4+(\frac{1}{2}+\epsilon)D_1+ (1+2\epsilon)D_2+(\frac{3}{2}+\epsilon)D_3+2\epsilon E_1+\\
+(\frac{1}{2}+\epsilon)D_7+ (1+2\epsilon)D_6+(\frac{3}{2}+\epsilon)D_5+2\epsilon E_2+(1-2\epsilon)E_3+(a-2\epsilon)\tilde{E}.\end{gather*}
So, $L$ is an effective divisor for a sufficiently small $\epsilon$. Note that there exists a $\PP^1$-fibration $g\colon\tilde{X}\rightarrow\PP^1$ with three singular fibers $$C_1=2E_1+D_1+2D_2+D_3,\quad C_2=2E_2+D_5+2D_6+D_7,\quad C_3=E_3+\tilde{E}.$$ Hence, $\tilde{X}\backslash\Supp(L)\cong\FF_2\backslash(M+F_1+F_2+F_3)$, where $M$ is a unique $(-2)$-curve and $F_1, F_2,F_3$ are fibers. So, $X$ has an $H$-polar cylinder.

Assume that $H=-K_X+a E$, where $E$ a $(-1)$-curve and $E$ passes through $P$ and $a<8$. So, we have the following configuration on $\tilde{X}$.\\
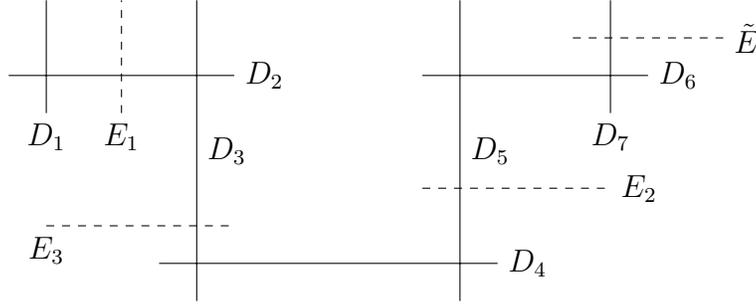
\begin{tikzpicture}
\draw  (3.5,0.5) -- (8,0.5) (4,0) -- (4,4) (1.5,3)--(4.5,3) (2,2.5)--(2,4) (7.5,0)--(7.5,4) (7,3)--(10,3) (9.5,2.5)--(9.5,4); \draw [dashed] (3,2.5) -- (3,4) (7,1.5) -- (9.5,1.5) (2,1) -- (4.5,1) (9,3.5)--(11,3.5);
\draw (8,0.5) node [right] {$D_4$} (2,2.5) node [below] {$D_1$} (4.5,3) node [right] {$D_2$} (4,2) node [right] {$D_3$} (3,2.5) node [below] {$E_1$} (9.5,2.5) node [below] {$D_7$} (10,3) node [right] {$D_6$}
(9.5,1.5) node [right] {$E_2$} (7.5,2) node [right] {$D_5$} (2,1) node [below] {$E_3$} (11,3.5) node [right] {$\tilde{E}$};
\end{tikzpicture}\\
where $\tilde{E}$ is the proper transform of $E$ on $\tilde{X}$, $\tilde{E},  E_1, E_2, E_3$ are $(-1)$-curves $D_1,\ldots,D_7$ are $(-2)$-curves. Assume that $a>\frac{1}{2}$. We have \begin{gather*}\varphi^*(H)\equiv L= (2+\frac{a}{2})D_4+(\frac{1}{2}+\frac{a}{8}+\epsilon)D_1+(1+\frac{a}{4}+2\epsilon)D_2+\\
+(\frac{3}{2}+\frac{3a}{8}+\epsilon)D_3+2\epsilon E_1+(\frac{5}{2}+\frac{5a}{8}-\epsilon)D_5+(\frac{3}{2}-\epsilon)E_2+\\
+(\frac{3}{2}+\frac{3a}{4}-\epsilon)D_6+(\frac{1}{2}+\frac{7a}{8}-\epsilon)D_7+(a-\frac{1}{2}-\epsilon)\tilde{E}.\end{gather*} Note that $L$ is effective for sufficiently small $\epsilon$. On the other hand, there exists a $\PP^1$-fibration $g\colon\tilde{X}\rightarrow\PP^1$ such that $D_4$ is a section and $g$ has two singular fibers $$C_1=2E_1+2D_2+D_1+D_3,\quad C_2=E_2+D_5+D_6+D_7+\tilde{E}.$$ Hence, $\tilde{X}\backslash\Supp(L)\cong\FF_2\backslash(M+F_1+F_2)$, where $M$ is a unique $(-2)$-curve and $F_1, F_2$ are fibers. So, $X$ has an $H$-polar cylinder.

Assume that $a\leq\frac{1}{2}$. Note that $$\tilde{E}+D_7+D_6+D_5+E_2\sim2E_1+2D_2+D_1+D_3.$$ Then \begin{gather*}\varphi^*(H)\equiv L= (\frac{2}{3}+\frac{9}{8}a-\epsilon)D_1+(\frac{4}{3}+\frac{9}{4}a-2\epsilon)D_2+(\frac{5}{3}+\frac{11}{8}a-\epsilon)D_3+\\
+(\frac{2}{3}-\epsilon)E_3+(\frac{1}{3}+2a-2\epsilon)E_1+(\frac{4}{3}+\frac{a}{2}+\epsilon)D_4+(\frac{1}{3}-\frac{a}{8}+\epsilon)D_7+\\
+(\frac{2}{3}-\frac{a}{4}+\epsilon)D_6+(1-\frac{3}{8}a+\epsilon)D_5+3\epsilon E_2.\end{gather*} Note that $L$ is effective for sufficiently small $\epsilon$. Let $h\colon\bar{X}\rightarrow\tilde{X}$ be that blow-up of intersection point of $D_3$ and $D_4$, and $B$ be an exceptional divisor. Put $\bar{E}_1,\bar{E}_2,\bar{E}_3,\bar{D}_1,\ldots,\bar{D}_7$ are the proper transform of $E_1,E_2,E_3,D_1,\ldots,D_7$ correspondingly. We see that there exists a $\PP^1$-fibration $g\colon\bar{X}\rightarrow\PP^1$ such that $B$ is a section and $g$ has two singular fibers $$C_1=2\bar{E}_1+2\bar{D}_2+\bar{D}_1+\bar{D}_3+\bar{E}_3,\quad C_2=3\bar{E}_2+\bar{D}_4+3\bar{D}_5+2\bar{D}_6+\bar{D}_7.$$ Hence, $\tilde{X}\backslash\Supp(L)\cong\FF_1\backslash(M+F_1+F_2)$, where $M$ is a unique $(-2)$-curve and $F_1, F_2$ are fibers. So, $X$ has an $H$-polar cylinder.
\end{proof}

\begin{lemma}
\label{LemA4}
Let $X$ be a del Pezzo surface with du Val singularities and let $H$ be an ample divisor on $X$. Assume that $X$ has a unique singular point $P$. Assume that $P$ is of type $A_4$. Then $X$ has an $H$-polar cylinder.
\end{lemma}

\begin{proof}
Let $\varphi\colon\tilde{X}\rightarrow X$ be the minimal resolution of singularities of $X$ and let $D=\sum_{i=1}^4 D_i$ be the exceptional divisor of $\varphi$.
Assume that $H=-K_X+\lambda E$, where $E$ is a $(-1)$-curve and $\lambda<1$. So, we have the following configuration on $\tilde{X}$.\\
\begin{tikzpicture}
\draw  (0.5,0) -- (0.5,2) (0,0.5) -- (4,0.5) (3.5,0)--(3.5,5) (3,4.5)--(5,4.5); \draw [dashed] (2,0) -- (2,3) (3,1)--(5,1) (3,1.5)--(5,1.5) (3,2)--(5,2) (3,2.5)--(5,2.5);
\draw (1,0.5) node [above] {$D_3$} (0.5,2) node [above] {$D_4$} (3.5,5) node [above] {$D_2$}  (5,4.5) node [right] {$D_1$} (2,3) node [above] {$E_1$} (5,1) node [right] {$E_2$} (5,1.5) node [right] {$E_3$} (5,2) node [right] {$E_4$} (5,2.5) node [right] {$E_5$};
\end{tikzpicture}$$\text{Figure 1}.$$
where $D_1, D_2, D_3, D_4$ are $(-2)$-curves, $E_1, E_2, E_3, E_4, E_5$ are $(-1)$-curves. Moreover, $E_1\cdot\tilde{E}=E_2\cdot\tilde{E}=0$, $E_3\cdot\tilde{E}=E_4\cdot\tilde{E}=E_5\cdot\tilde{E}=1$, $M\cdot\tilde{E}=2$, where $\tilde{E}$ is the proper transform of $E$. We have
\begin{gather*}\varphi^*(H)\equiv L= bD_4+2bD_3+2aD_2+aD_1+(2b-1)E_1+(2a-1)E_2+\\
+(2a-1-\lambda)(E_3+E_4+E_5)+(3b-1-2\lambda)M,\end{gather*} where $a=\frac{1+\lambda}{2}+\epsilon$, $b=\frac{1+\lambda}{2}-\epsilon$, $M$ is a $(0)$-curve such that $M$ pass through the intersection point of $D_2$ and $D_3$, $\epsilon$ is a sufficiently small number. Moreover, $M$ does not meet $D_1, D_4, E_1, E_2,\ldots, E_5$. We can blow-up the intersection point of $D_2$ and $D_3$, after that blow-up the intersection point of exceptional divisor and the proper transform of $D_3$ and so on. We obtain a surface $Y$ with the following configuration\\
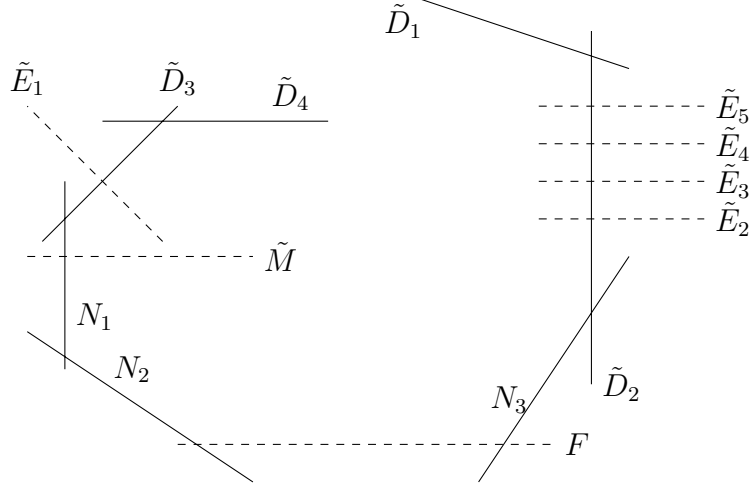
\begin{tikzpicture}
\draw  (3,0) -- (0,2) (0.5,1.5) -- (0.5,4) (0.2,3.2)--(2,5) (1,4.8)--(4,4.8) (6,0)--(8,3) (7.5,1.3)--(7.5,6) (8,5.5)--(5,6.5); \draw [dashed] (2,0.5) -- (7,0.5) (0,3)--(3,3) (1.8,3.2)--(0,5) (6.8,3.5)--(9,3.5) (6.8,4)--(9,4) (6.8,4.5)--(9,4.5) (6.8,5)--(9,5);
\draw (7.5,1.3) node [right] {$\tilde{D}_2$} (2,5) node [above] {$\tilde{D}_3$} (3.5,4.8) node [above] {$\tilde{D}_4$}  (5,6.5) node [below] {$\tilde{D}_1$} (0,5) node [above] {$\tilde{E}_1$} (9,3.5) node [right] {$\tilde{E}_2$} (9,4) node [right] {$\tilde{E}_3$} (9,4.5) node [right] {$\tilde{E}_4$} (9,5) node [right] {$\tilde{E}_5$} (3,3) node [right] {$\tilde{M}$} (7,0.5) node [right] {$F$} (0.5,2.2) node [right] {$N_1$} (1,1.5) node [right] {$N_2$} (6.8,1.1) node [left] {$N_3$};
\end{tikzpicture}\\
where curves $\tilde{D}_1, \tilde{D}_2, \tilde{D}_3, \tilde{D}_4$ are proper transform of $D_1, D_2, D_3, D_4$, curves $\tilde{E}_1, \tilde{E}_2, \tilde{E}_3, \tilde{E}_4, \tilde{E}_5$ are proper transform of $E_1, E_2, E_3, E_4, E_5$, the curve $\tilde{M}$ is the proper transform of $M$, and $N_1, N_2, N_3$ and $F$ are exceptional curves. Moreover, $\tilde{E}_1, \tilde{E}_2, \tilde{E}_3, \tilde{E}_4, \tilde{E}_5, \tilde{M}$ and $F$ are $(-1)$-curves, $\tilde{D}_1, \tilde{D}_4, N_1, N_3$ are $(-2)$-curves, $\tilde{D}_2, N_2$ are $(-3)$-curves, $\tilde{D}_3$ is a $(-5)$-curve. So, there exists $\PP^1$-fibration $h\colon Y\rightarrow\PP^1$ such that $h$ has only two singular fibers $$C_1=2\tilde{E}_1+3\tilde{M}+2\tilde{D}_3+\tilde{D}_4+3N_1+N_2,$$ $$C_2=2(\tilde{E}_2+\tilde{E}_3+\tilde{E}_4+\tilde{E}_5+\tilde{D}_2)+\tilde{D}_1+N_3.$$
Note that $\tilde{X}\backslash\Supp(L)\cong Y\backslash\Supp(C_1+C_2+F)\cong\FF_1\backslash(K+F_1+F_2)$, where $K$ is a unique $(-1)$-curve and $F_1, F_2$ are fibers. So, $X$ has an $H$-polar cylinder.

Assume that $H=-K_X+\lambda_1 \hat{E}_1+\lambda_2 \hat{E}_2$, where $\hat{E}_1$, $\hat{E}_2$ are $(-1)$-curves and $\lambda_1<1$, $\lambda_2<1$. We may assume that $\lambda_1\geq\lambda_2$. Assume that $\lambda_1+\lambda_2\geq 1$. Then we have the following configuration on $\tilde{X}$.\\
\begin{tikzpicture}
\draw  (0.5,0) -- (0.5,5) (0,0.5) -- (6,0.5) (5.5,0)--(5.5,5) (5,4.5)--(10,4.5); \draw [dashed] (4,0) -- (4,3) (3,0)--(3,3) (5,1.5)--(7,1.5) (5,2)--(7,2) (5,2.5)--(7,2.5) (9,5)--(9,2);
\draw (1,0.5) node [above] {$D_3$} (0.5,5) node [above] {$D_4$} (5.5,5) node [above] {$D_2$}  (10,4.5) node [right] {$D_1$} (4,3) node [above] {$E_2$} (3,3) node [above] {$E_1$} (7,1.5) node [right] {$E_3$} (7,2) node [right] {$E_4$} (7,2.5) node [right] {$E_5$} (9,2) node [right] {$E_6$};
\end{tikzpicture}$$\text{Figure 2}.$$
where $D_1, D_2, D_3, D_4$ are $(-2)$-curves, $E_1, E_2, E_3, E_4, E_5, E_6$ are $(-1)$-curves. We have
\begin{gather*}\varphi^*(H)\equiv L=aD_4+2aD_3+(b+1)D_2+bD_1+(2a-1-\lambda_1)E_1+\\
+(2a-1-\lambda_2)E_2+bE_3+(b-\lambda_1-\lambda_2)(E_4+E_5)+(b-1)E_6,\end{gather*}
where $a+b=1+\lambda_1+\lambda_2$. We see that $L$ is an effective divisor if $b>\lambda_1+\lambda_2$ and $a>\frac{1}{2}+\frac{\lambda_1}{2}$. Since $\lambda_1<1$, we see that we can choose $a$ and $b$ such that $L$ is an effective divisor. Let $h_1\colon Y_1\rightarrow\tilde{X}$ be the blow-up of intersection point of $D_2$ and $D_3$, and $F_1$ is the exceptional divisor. Let $h_2\colon Y_2\rightarrow Y_1$ be the blow-up of intersection point of $N_1$ and $D'_2$, where $D'_2$ is the proper transform of $D_2$. Let $N_2$ be the exceptional divisor of $h_2$. Note that there exists a $\PP^1$-fibration $g\colon Y_2\rightarrow\PP^1$ such that $N_2$ is a section of $g$ and $g$ has two singular fibers $$C_1=2\bar{E}_1+2\bar{E}_2+2\bar{D}_3+\bar{D}_4+N_1,\quad C_2=\bar{E}_3+\bar{E}_4+\bar{E}_5+\bar{D}_2+\bar{D}_1+\bar{E}_6,$$ where $\bar{E}_1,\ldots,\bar{E}_6,\bar{D}_1,\ldots,\bar{D}_4$ are the proper transform of $$E_1,\ldots, E_6,D_1,\ldots, D_4$$ correspondingly. So, $\tilde{X}\backslash\Supp(L)\cong\FF_1\backslash(K+F_1+F_2)$, where $K$ is a unique $(-1)$-curve and $F_1, F_2$ are fibers. So, $X$ has an $H$-polar cylinder.

Assume that $\lambda_1+\lambda_2<1$. Then we have the same configuration as in Figure 1. So, \begin{gather*}\varphi^*(H)\equiv L= bD_4+2bD_3+2aD_2+aD_1+(2b-1)E_1+(2a-1-\lambda_1)E_2+\\
+(2a-1-\lambda_2)E_3+(2a-1-\lambda_1-\lambda_2)(E_4+E_5)+\\
+(3b-1-2(\lambda_1+\lambda_2))M,\end{gather*} where $a=\frac{1+\lambda_1+\lambda_2}{2}+\epsilon$, $b=\frac{1+\lambda_1+\lambda_2}{2}-\epsilon$, $M$ is a $(0)$-curve such that $M$ pass through the intersection point of $D_2$ and $D_3$ and $\epsilon$ is a sufficiently small number. As above, we see that $X$ has an $H$-polar cylinder.

Assume that $H=-K_X+\lambda_1 \hat{E}_1+\lambda_2 \hat{E}_2+\lambda_3 \hat{E}_3$, where $\hat{E}_1$, $\hat{E}_2$, $\hat{E}_3$ are $(-1)$-curves and $\lambda_1<1$, $\lambda_2<1$, $\lambda_3<1$. We may assume that $\lambda_1\geq\lambda_2\geq\lambda_3$. We have configuration as in Figure 2. Moreover, $E_1\cdot\hat{E}_i=E_6\cdot\hat{E}_i=0$, $E_2\cdot\hat{E}_i=1$ for every $i$, $E_3\cdot\hat{E}_1=0$, $E_3\cdot\hat{E}_2=E_3\cdot\hat{E}_3=1$, $E_4\cdot\hat{E}_2=0$, $E_4\cdot\hat{E}_1=E_4\cdot\hat{E}_3=1$, $E_5\cdot\hat{E}_3=0$, $E_5\cdot\hat{E}_1=E_5\cdot\hat{E}_2=1$. Then
\begin{gather*}\varphi^*(H)\equiv L=aD_4+2aD_3+(b+1)D_2+bD_1+(2a-1)E_1+\\
+(2a-1-\lambda_1-\lambda_2-\lambda_3)E_2+(b-\lambda_2-\lambda_3)E_3+(b-\lambda_1-\lambda_3)E_4+\\
+(b-\lambda_1-\lambda_2)E_5+(b-1)E_6,\end{gather*} where $a+b=1+\lambda_1+\lambda_2+\lambda_3$. Note that if $2a-1-\lambda_1-\lambda_2-\lambda_3>0$ and $b-\lambda_1-\lambda_2>0$ then $L$ is an effective divisor. So, if $\lambda_1+\lambda_2-\lambda_3<1$ then $X$ has an $H$-polar cylinder. Assume that $\lambda_1>\lambda_2$. We have the following configuration on $\tilde{X}$.\\
\begin{tikzpicture}
\draw  (0,0.5) -- (10,0.5) (8.5,0) -- (10.5,2) (10,1)--(10,5) (10.5,4)--(8.5,6); \draw [dashed] (2,0) -- (0,3) (0.5,2)--(0.5,5) (4,0) -- (2,3) (2.5,2)--(2.5,5) (6,0) -- (4,3) (4.5,2)--(4.5,5) (8,0) -- (6,3) (6.5,2)--(6.5,5) (9,3)--(12,3) (9.8,0.7)--(9,2);
\draw (1.5,1.5) node [above] {$E_1$} (0.5,5) node [above] {$\hat{E}_1$} (3.5,1.5) node [above] {$E_2$} (2.5,5) node [above] {$\hat{E}_2$} (5.5,1.5) node [above] {$E_3$} (4.5,5) node [above] {$\hat{E}_3$} (7.5,1.5) node [above] {$E_4$} (6.5,5) node [above] {$\hat{E}_4$} (12,3) node [right] {$E_5$} (0,0.5) node [above] {$D_1$} (10.5,2) node [right] {$D_2$} (10,5) node [above] {$D_3$} (8.5,6) node [above] {$D_4$} (9,2) node [left] {$E_6$} ;
\end{tikzpicture}$$\text{Figure 3}.$$ where dotted lines are $(-1)$-curves and $D_1, D_2,D_3,D_4$ are $(-2)$-curves. We see that $\hat{E}_i\equiv-E_i+D_2+2D_3+D_4+2E_5$.
\begin{gather*}\varphi^*(H)\equiv L=aD_1+(a-1-\lambda_2)E_2+
(a-1-\lambda_3)E_3+(a-1)E_4+cM+\\+(c+\lambda_1-1)\hat{E}_1+(b+\lambda_2+\lambda_3)D_2+(b-1)E_6+(b+\lambda_2+\lambda_3-1)D_4+\\
+(2b+2\lambda_2+2\lambda_3-2)D_3+(2b+2\lambda_2+2\lambda_3-3)E_5,\end{gather*} where $M$ is a $(0)$-curve that passes through intersection point of $D_1$ and $D_2$, $a+b+c=3$. Put $a=1+\lambda_2+\epsilon$, $c=1-\lambda_1+\epsilon$, $b=1+\lambda_1-\lambda_2-2\epsilon$. Note that $\lambda_1>\frac{1}{2}$. We see that $L$ is an effective divisor for sufficiently small $\epsilon$.
Let $h_1\colon Y_1\rightarrow\tilde{X}$ be the blow-up of the intersection point of $D_1$ and $D_2$ and $F$ is an exceptional divisor of $h_1$. We have the following configuration on $Y_1$.\\
\begin{tikzpicture}
\draw  (0.5,0) -- (0.5,5) (7.5,0) -- (9.5,2) (9,1)--(9,5) (9.5,4)--(7.5,6); \draw [dashed] (0,0.5) -- (10,0.5) (0,2) -- (3,2) (0,3) -- (3,3) (0,4) -- (3,4) (4.5,0)--(4.5,3) (4,2) -- (6,3) (8.8,0.7)--(8,2) (8,3)--(11,3);
\draw (0.5,5) node [above] {$\tilde{D}_1$} (3,2) node [right] {$\tilde{E}_2$} (3,3) node [right] {$\tilde{E}_3$} (3,4) node [right] {$\tilde{E}_4$} (6,3) node [right] {$\hat{E}'_1$}  (4.5,1.5) node [right] {$M'$} (11,3) node [right] {$\tilde{E}_5$}  (9.5,2) node [right] {$\tilde{D}_2$} (9,5) node [above] {$\tilde{D}_3$} (7.5,6) node [above] {$\tilde{D}_4$} (8,2) node [left] {$\tilde{E}_6$} (10,0.5) node [right] {$F$};
\end{tikzpicture}$$\text{Figure 4}.$$
where $\tilde{D}_1, \tilde{D}_2, \tilde{D}_3, \tilde{D}_4$ are the proper transforms of $D_1,D_2,D_3,D_4$, the curves $\tilde{E}_2,\tilde{E}_3,\tilde{E}_4,\tilde{E}_5,\tilde{E}_6$ are the proper transforms of $E_2,E_3,E_4,E_5,E_6$, $\hat{E}'_1, \hat{E}'_2, \hat{E}'_3, \hat{E}'_4$ are the proper transforms of $\hat{E}_1, \hat{E}_2, \hat{E}_3, \hat{E}_4$, $M'$ is the proper transform of $M$. Note that there exists a $\PP^1$-fibration $g\colon Y_1\rightarrow\PP^1$ such that $g$ has three singular fibers $$C_1=\tilde{D}_1+\tilde{E}_2+\tilde{E}_3+\tilde{E}_4,\quad C_2=\tilde{D}_2+\tilde{D}_3+\tilde{D}_4+\tilde{E}_5+\tilde{E}_6,\quad C_3=M'+\hat{E}'_1.$$ Then $\tilde{X}\backslash\Supp(L)\cong\FF_1\backslash(K+F_1+F_2+F_3)$, where $K$ is a unique $(-1)$-curve and $F_1, F_2, F_3$ are fibers. Then $X$ has an $H$-polar cylinder. So, $\lambda_1=\lambda_2>\frac{1}{2}$. We have \begin{gather*}\varphi^*(H)\equiv L=aD_1+(a-1-\lambda_1)E_1+(a-1-\lambda_2)E_2+(a-1-\lambda_3)E_3+\\
+(a-1)E_4+bM+(b+\lambda_1+\lambda_2+\lambda_3)D_2+(b-1+\lambda_1+\lambda_2+\lambda_3)D_4+\\
+(2b-2+2\lambda_1+2\lambda_2+2\lambda_3)D_3+(2b-3+2\lambda_1+2\lambda_2+2\lambda_3)E_5,\end{gather*}
where $a+b=2$. Put $a=1+\lambda_1+\epsilon$, $b=1-\lambda_1-\epsilon$. Since $\lambda_1=\lambda_2>\frac{1}{2}$, we see that $L$ is an effective divisor for sufficiently small $\epsilon$.

Let $h_2\colon Y_2\rightarrow Y_1$ be the blow-up of the intersection point of $\tilde{D}_1$ and $F$ and $F'$ is an exceptional divisor of $h_2$. We have the following configuration on $Y_2$.\\
\begin{tikzpicture}
\draw  (0.5,0) -- (0.5,5) (7.5,0) -- (9.5,2) (9,1)--(9,5) (9.5,4)--(7.5,6) (6,5.5)--(8.5,5.5); \draw [dashed] (0,0.5) -- (10,0.5) (0,1) -- (3,1) (0,2) -- (3,2) (0,3) -- (3,3) (0,4) -- (3,4) (8.8,0.7)--(8,2) (6.5,4)--(8.7,5.3);
\draw (0.5,5) node [above] {$D'_1$} (3,1) node [right] {$E'_1$} (3,2) node [right] {$E'_2$} (3,3) node [right] {$E'_3$} (3,4) node [right] {$E'_4$}   (9.5,2) node [right] {$\tilde{F}$} (9,5) node [above] {$D'_2$} (7.5,6) node [above] {$D'_3$} (8,2) node [left] {$M''$} (6.5,4) node [left] {$E'_5$} (6,5.5) node [left] {$D'_4$} (10,0.5) node [right] {$F'$};
\end{tikzpicture}$$\text{Figure 5}.$$
where $D'_1, D'_2, D'_3, D'_4$ are the proper transforms of $D_1,D_2,D_3,D_4$, the curves $E'_1,E'_2,E'_3,E'_4,E'_5$ are the proper transforms of $E_1,E_2,E_3,E_4,E_5$, $M''$ is the proper transform of $M$, $\tilde{F}$ is the proper transform of $F$. Put $\hat{E}''_1$, $\hat{E}''_2$ $\hat{E}''_3$ are the proper transforms of $\hat{E}_1,\hat{E}_2,\hat{E}_3$. Note that there exists a $\PP^1$-fibration $g\colon Y_2\rightarrow\PP^1$ such that $F'$ is a section of $g$ and $g$ has two singular fibers $$C_1=D'_1+E'_1+E'_2+E'_3+E'_4,\quad C_2=2E'_5+2D'_3+D'_4+D'_2+\tilde{F}+M''.$$ So, $\tilde{X}\backslash\Supp(L)\cong\FF_1\backslash(K+F_1+F_2)$, where $K$ is a unique $(-1)$-curve and $F_1, F_2$ are fibers. Then $X$ has an $H$-polar cylinder.

Assume that $H=-K_X+\lambda_1 \hat{E}_1+\lambda_2 \hat{E}_2+\lambda_3 \hat{E}_3+\lambda_4 \hat{E}_4$, where $\hat{E}_1$, $\hat{E}_2$, $\hat{E}_3$, $\hat{E}_4$ are $(-1)$-curves and $\lambda_1<1$, $\lambda_2<1$, $\lambda_3<1$, $\lambda_4<1$. We may assume that $\lambda_1\geq\lambda_2\geq\lambda_3\geq\lambda_4$.  We have configuration as in Figure 1. Moreover, $E_1\cdot\hat{E}_i=0$ for every $i$, $E_j\cdot\hat{E}_i=1$ for $j=2,3,4,5$ and $j\neq i+1$, $E_j\cdot\hat{E}_{j+1}=0$ for $j=2,3,4,5$. We have
\begin{gather*}\varphi^*(H)\equiv L= bD_1+2bD_2+2aD_3+aD_4+(2b-1)E_1+\\
+(2a-1-\lambda_2-\lambda_3-\lambda_4)E_2+
+(2a-1-\lambda_1-\lambda_3-\lambda_4)E_3+\\
+(2a-1-\lambda_1-\lambda_2-\lambda_4)E_4+(2a-1-\lambda_1-\lambda_2-\lambda_3)E_5+\\
+(3b-1-2\lambda_1-2\lambda_2-2\lambda_3-2\lambda_4)M,\end{gather*} where $a+b=1+\lambda_1+\lambda_2+\lambda_3+\lambda_4$, $M$ is a $(0)$-curve such that $M$ pass through the intersection point of $D_2$ and $D_3$. As above, we see that if $\lambda_1+\lambda_2+\lambda_3-2\lambda_4<1$ then there exist $a$ and $b$ such that $L$ is an effective divisor. Then $X$ has an $H$-polar cylinder. So, we may assume that $\lambda_1+\lambda_2+\lambda_3-2\lambda_4\geq 1$. Assume that $\lambda_1>\lambda_2$. Then \begin{gather*}\varphi^*(H)\equiv L=aD_1+(a-1-\lambda_2)E_2+\\+
(a-1-\lambda_3)E_3+(a-1-\lambda_4)E_4+cM+(c+\lambda_1-1)\hat{E}_1+\\
+(b+\lambda_2+\lambda_3+\lambda_4)D_2+(b-1)E_6+(b+\lambda_2+\lambda_3+\lambda_4-1)D_4+\\
+(2b+2\lambda_2+2\lambda_3+2\lambda_4-2)D_3+(2b+2\lambda_2+2\lambda_3+2\lambda_4-3)E_5,\end{gather*} where $M$ is $(0)$-curve that passes through intersection point of $D_1$ and $D_2$. As above, put $a=1+\lambda_2+\epsilon$, $c=1-\lambda_1+\epsilon$, $b=1+\lambda_1-\lambda_2$. Note that $\lambda_1+\lambda_3+\lambda_4>\frac{1}{2}$. We see that $L$ is an effective divisor for sufficiently small $\epsilon$ and $\tilde{X}\backslash\Supp(L)\cong\FF_1\backslash(K+F_1+F_2+F_3)$, where $K$ is a unique $(-1)$-curve and $F_1, F_2, F_3$ are fibers. Then $X$ has an $H$-polar cylinder. So, $\lambda_1=\lambda_2$. We have
\begin{gather*}\varphi^*(H)\equiv L=aD_1+(a-1-\lambda_1)E_1+(a-1-\lambda_2)E_2+(a-1-\lambda_3)E_3+\\
+(a-1-\lambda_4)E_4+bM+(b+\lambda_1+\lambda_2+\lambda_3+\lambda_4)D_2+\\
+(b-1+\lambda_1+\lambda_2+\lambda_3+\lambda_4)D_4+(2b-2+2\lambda_1+2\lambda_2+2\lambda_3+2\lambda_4)D_3+\\
+(2b-3+2\lambda_1+2\lambda_2+2\lambda_3+2\lambda_4)E_5,\end{gather*} where $a+b=2$. Put $a=1+\lambda_1+\epsilon$, $b=1-\lambda_1-\epsilon$. Since $\lambda_1=\lambda_2$ and $\lambda_1+\lambda_2+\lambda_3-2\lambda_4\geq 1$ we see that $\lambda_2+\lambda_3+\lambda_4>\frac{1}{2}$. So, $L$ is an effective divisor for sufficiently small $\epsilon$ and $\tilde{X}\backslash\Supp(L)\cong\FF_1\backslash(K+F_1+F_2)$, where $K$ is a unique $(-1)$-curve and $F_1, F_2$ are fibers. Then $X$ has an $H$-polar cylinder.

Assume that $H=-K_X+\lambda\varphi(E)$, where $0<\lambda<5$, $E$ is a $(-1)$-curve on $\tilde{X}$ such that $E\cdot D_1=1$ and $E\cdot D_2=E\cdot D_3=E\cdot D_4=0$. We have $$\varphi^*(H)=-K_{\tilde{X}}+\lambda E+\frac{4}{5}\lambda D_1+\frac{3}{5}\lambda D_2+\frac{2}{5}\lambda D_3+\frac{1}{5}\lambda D_4.$$ Assume that $\lambda<1$. We have the same configuration as in Figure 2 with $E=E_6$. We obtain \begin{gather*}\varphi^*(H)\equiv L= bD_1+2bD_2+2aD_3+aD_4+(2b-1-\frac{3}{5}\lambda)E_3+\\
+(2b-1-\frac{3}{5}\lambda)E_4+(2b-1-\frac{3}{5}\lambda)E_5+(2a-1-\frac{2}{5}\lambda)E_1+\\
+(2a-1-\frac{2}{5}\lambda)E_2+(a-\frac{6}{5}\lambda)M,\end{gather*} where $a+b=1+\lambda$, $M$ is a $(0)$-curve such that $M$ pass through the intersection point of $D_2$ and $D_3$. Since $\lambda<1$, we see that there exist $a$ and $b$ such that $L$ is an effective divisor. Then $X$ has an $H$-polar cylinder.  So, $\lambda\geq 1$. We obtain \begin{gather*}\varphi^*(H)\equiv L= bD_1+(b-1+\frac{\lambda}{5})E+(b+1-\frac{\lambda}{5})D_2+2aD_3+aD_4+\\
+(b-\frac{4}{5}\lambda)E_3+(b-\frac{4}{5}\lambda)E_4+(b-\frac{4}{5}\lambda)E_5+\\
+(2a-1-\frac{2}{5}\lambda)E_1+(2a-1-\frac{2}{5}\lambda)E_2,\end{gather*}
where $a+b=1+\lambda$. Since $\lambda>\frac{1}{2}$, we see that there exist $a$ and $b$ such that $L$ is an effective divisor. As above, $X$ has an $H$-polar cylinder.

Assume that $H=-K_X+\lambda\varphi(E)+\lambda_1 \hat{E}_1+\lambda_2 \hat{E}_2+\lambda_3 \hat{E}_3$, where $0<\lambda<5$, $0\leq\lambda_3\leq\lambda_2\leq\lambda_1<1$, $E$ is a $(-1)$-curve on $\tilde{X}$ such that $E\cdot D_1=1$ and $E\cdot D_2=E\cdot D_3=E\cdot D_4=0$, $\hat{E}_1, \hat{E}_2, \hat{E}_3$ are $(-1)$-curves on $X$. Let $\tilde{E}_1$, $\tilde{E}_2$, $\tilde{E}_3$ be the proper transforms of $\hat{E}_1$, $\hat{E}_2$, $\hat{E}_3$.
Assume that $\lambda_1+\lambda_2-\lambda_3<1$. As above, we consider configurations on Figure 2, and we may assume that $E=E_6$. Then $X$ has as an $H$-polar cylinder. So, we may assume that $\lambda_1+\lambda_2-\lambda_3\geq 1$. Assume that $\lambda_1>\frac{1}{2}$ or $\lambda_2+\lambda_3>\frac{1}{2}$. As above, we consider configurations on Figure 4 and Figure 5. We may assume that the proper transform of $E$ is $\tilde{E}_4$ for Figure 4 and $E'_4$ for Figure 4. As above, $X$ has an $H$-polar cylinder. So, we may assume that $\lambda_1\leq\frac{1}{2}$, $\lambda_2+\lambda_3\leq\frac{1}{2}$ and $\lambda_1+\lambda_2-\lambda_3\geq 1$. Hence, $\lambda_1=\lambda_2=\frac{1}{2}$ and $\lambda_3=0$. We have the following configuration on $\tilde{X}$.\\
\begin{tikzpicture}
\draw  (0.5,0) -- (0.5,5) (0,0.5) -- (6,0.5) (5.5,0)--(5.5,5) (5.2,4.5)--(10,4.5); \draw [dashed] (0,4)--(2,4) (0,3) -- (2,3) (4,0) -- (4,3) (3.8,2)--(5,4) (3,0)--(3,3) (2.8,2)--(4,4) (5,1.5)--(7,1.5) (9,5)--(9,2);
\draw (1,0.5) node [above] {$D_3$} (0.5,5) node [above] {$D_4$} (5.5,5) node [above] {$D_2$}  (10,4.5) node [right] {$D_1$} (4,3) node [above] {$E_2$} (3,3) node [above] {$E_1$} (7,1.5) node [right] {$E_3$} (2,3) node [right] {$E_4$} (2,4) node [right] {$E_5$} (9,2) node [right] {$E$} (5,4) node [above] {$\tilde{E}_2$} (4,4) node [above] {$\tilde{E}_1$};
\end{tikzpicture}$$\text{Figure 6}$$
where dotted lines are $(-1)$-curves, $D_1, D_2,D_3,D_4$ are $(-2)$-curves, $\tilde{E}_1$, $\tilde{E}_2$ are the proper transform of $\hat{E}_1$, $\hat{E}_2$. We obtain \begin{gather*}\varphi^*(H)\equiv L=(2+\frac{2}{5}\lambda)D_3+(1+\epsilon+\frac{1}{5}\lambda)D_4+\epsilon E_4+\epsilon E_5+\\
+(\frac{1}{2}+\epsilon)E_1+\epsilon\tilde{E}_1+(\frac{1}{2}+\epsilon)E_2+\epsilon\tilde{E}_2+\\
+(2-3\epsilon+\frac{3}{5}\lambda)D_2+(1-3\epsilon+\frac{4}{5}\lambda)D_1+(1-3\epsilon)E_3+(\lambda-3\epsilon)E.\end{gather*} We see that for sufficiently small $\epsilon$, $L$ is an effective divisor. Note that there exists a $\PP^1$-fibration $g\colon\tilde{X}\rightarrow\PP^1$ with four singular fibers $$C_1=D_4+E_4+E_5,\quad C_2=D_2+D_1+E_3+E,$$ $$C_3=E_1+\tilde{E}_1,\quad C_4=E_4+\tilde{E}_4.$$ Hence, $\tilde{X}\backslash\Supp(L)\cong\FF_2\backslash(M+F_1+F_2+F_3+F_4)$, where $M$ is a unique $(-2)$-curve and $F_1, F_2,F_3,F_4$ are fibers. So, $X$ has an $H$-polar cylinder.

Assume that $H=-K_X+\mu C+\lambda\varphi(E)+\lambda_1\hat{E}_1+\lambda_2\hat{E}_2$, where $C$ is a $(0)$-curve and $\mu>0$, $0\leq\lambda<5$, $0\leq\lambda_2\leq\lambda_1<1$, $E$ is a $(-1)$-curve on $\tilde{X}$ such that $E\cdot D_1=1$ and $E\cdot D_2=E\cdot D_3=E\cdot D_4=0$, $\hat{E}_1, \hat{E}_2$ are $(-1)$-curves on $X$. Let $\tilde{E}_1$, $\tilde{E}_2$ be the proper transforms of $\hat{E}_1$, $\hat{E}_2$. Note that $\tilde{C}\sim E+D_1+D_2+D_3+D_4+E'$, where $E'$ is a $(-1)$-curve such that $E'\cdot D_4=1$ and $E'\cdot D_1=E'\cdot D_2=E'\cdot D_3=0$. Assume that $\lambda_1=\lambda_2=0$. We have the following configuration on $\tilde{X}$.\\
\begin{tikzpicture}
\draw  (0.5,0) -- (0.5,5) (0,0.5) -- (6,0.5) (5.5,0)--(5.5,5) (5.2,4.5)--(10,4.5); \draw [dashed] (0,4)--(2,4) (4,0) -- (4,3)  (3,0)--(3,3)  (5,1.5)--(7,1.5) (5,2)--(7,2) (9,5)--(9,2);
\draw (1,0.5) node [above] {$D_3$} (0.5,5) node [above] {$D_4$} (5.5,5) node [above] {$D_2$}  (10,4.5) node [right] {$D_1$} (4,3) node [above] {$E_3$} (3,3) node [above] {$E_4$} (7,1.5) node [right] {$E_1$} (7,2) node [right] {$E_2$}  (2,4) node [right] {$E'$} (9,2) node [right] {$E$};
\end{tikzpicture}$$\text{Figure 7}$$ where dotted lines are $(-1)$-curves, $D_1,D_2,D_3,D_4$ are $(-2)$-curves. Then \begin{gather*}\varphi^*(H)\equiv L=(a+\frac{2}{5}\lambda+\mu)D_3+(a-1+\frac{1}{5}\lambda+\mu)D_4+\\
+(a-1)E_3+(a-1)E_4+(a-2+\mu)E'+(b+\frac{3}{5}\lambda+\mu)D_2\\
+(b-1)E_1+(b-1)E_2+(b-1+\frac{4}{5}\lambda+\mu)D_1+(b-2+\lambda+\mu)E,\end{gather*}
where $a+b=3$. We see that if $\lambda+2\mu>1$ then there exist $a$ and $b$ such that $L$ is an effective divisor. Let $h\colon\bar{X}\rightarrow\tilde{X}$ be the blow-ups the point of intersection $D_2$ and $D_3$ and let $\bar{E}_1,\ldots, \bar{E}_4,\bar{E},\bar{E}' \bar{D}_1,\ldots,\bar{D}_4$ be the proper transform of $E_1,\ldots,E_4,E, E',D_1,\ldots,D_4$ correspondingly. Note that there exists a $\PP^1$-fibration $g\colon\bar{X}\rightarrow\PP^1$ with two singular fibers $$C_1=\bar{E}'+\bar{D}_4+\bar{D}_3+\bar{E}_3+\bar{E}_4,\quad C_2=\bar{E}+\bar{D}_1+\bar{D}_2+\bar{E}_1+\bar{E}_2.$$ Hence, $\tilde{X}\backslash\Supp(L)\cong\FF_1\backslash(M+F_1+F_2)$, where $M$ is a unique $(-1)$-curve and $F_1, F_2$ are fibers. So, $X$ has an $H$-polar cylinder.
So, we may assume that $\lambda+2\mu\leq 1$. As above, we have the configuration as in Figure 2. We obtain \begin{gather*}\varphi^*(H)\equiv L= bD_1+2bD_2+2aD_3+aD_4+(2b-1-\frac{3}{5}\lambda-2\mu)E_1+\\
+(2b-1-\frac{3}{5}\lambda-\mu)E_2+(2b-1-\frac{3}{5}\lambda-\mu)E_3+(2a-1-\frac{2}{5}\lambda-\mu)E_4+\\
+(2a-1-\frac{2}{5}\lambda-\mu)E_5+(a-\frac{6}{5}\lambda-2\mu)M,\end{gather*} where $a+b=1+2\mu+\lambda$, $M$ is a $(0)$-curve such that $M$ pass through the intersection point of $D_2$ and $D_3$. We see that if $\lambda+2\mu<1$ then there exist $a$ and $b$ such that $L$ is an effective divisor. Then $X$ has an $H$-polar cylinder. So, we may assume that $\lambda+2\mu = 1$. As above, we have the configuration as in Figure 2. We obtain \begin{gather*}\varphi^*(H)\equiv L= bD_1+(b+1-\frac{1}{5}\lambda)D_2+2aD_3+aD_4+(b-1+\frac{1}{5}\lambda)E+\\
+(b-\frac{4}{5}\lambda-2\mu)E_3+(b-\frac{4}{5}\lambda-\mu)E_4+(b-\frac{4}{5}\lambda-\mu)E_5+\\
+(2a-1-\frac{2}{5}\lambda-\mu)E_1+(2a-1-\frac{2}{5}\lambda-\mu)E_2,\end{gather*} where $a+b=1+2\mu+\lambda$. We see that if $\lambda+2\mu=1$ then there exist $a$ and $b$ such that $L$ is an effective divisor. Then $X$ has an $H$-polar cylinder. So, we may assume that $\lambda_1>0$. We have configuration as in Figure 6. So, if $\lambda_1+\lambda_2\geq 1$, then $X$ has an $H$-polar cylinder. Assume that $\lambda_1+\lambda_2<1$. We have configuration as in Figure 7. Note that there exist two cases. The first case is when $\tilde{E}_1\cdot E_1=\tilde{E}_2\cdot E_1=1$, $\tilde{E}_1\cdot E_2=\tilde{E}_2\cdot E_2=0$, $\tilde{E}_1\cdot E_3=\tilde{E}_2\cdot E_4=1$, $\tilde{E}_1\cdot E_4=\hat{E}_2\cdot E_3=0$. Then \begin{gather*}\varphi^*(H)\equiv L=a\tilde{D}_4+(a-1-\frac{1}{5}\lambda)E'+(a+1+\frac{1}{5}\lambda)\tilde{D}_3+\\
+(a-\mu-\lambda_1-\frac{1}{5}\lambda)E_3+(a-\mu-\lambda_2-\frac{1}{5}\lambda)E_4+b\tilde{D}_1+(b-1+\frac{1}{5}\lambda)E+\\
+(b+1-\frac{1}{5}\lambda)\tilde{D}_2+(b-\frac{4}{5}\lambda-\mu-\lambda_1-\lambda_2)E_1+(b-\frac{4}{5}\lambda-\mu)E_2,\end{gather*}
where $a+b=1+\lambda+\lambda_1+\lambda_2+2\mu$. We see that if $\lambda+\lambda_1+\lambda_2+2\mu>1$ then $X$ has an $H$-polar cylinder.
The second case is when $\tilde{E}_1\cdot E_3=\tilde{E}_2\cdot E_3=1$, $\tilde{E}_1\cdot E_4=\tilde{E}_2\cdot E_4=0$, $\tilde{E}_1\cdot E_1=\tilde{E}_2\cdot E_2=1$, $\tilde{E}_1\cdot E_2=\tilde{E}_2\cdot E_1=0$. Then \begin{gather*}\varphi^*(H)\equiv L=a\tilde{D}_4+(a-1-\frac{1}{5}\lambda)E'+(a+1+\frac{1}{5}\lambda)\tilde{D}_3+\\
+(a-\mu-\lambda_1-\lambda_2-\frac{1}{5}\lambda)E_3+(a-\mu-\frac{1}{5}\lambda)E_4+b\tilde{D}_1+(b-1+\frac{1}{5}\lambda)E+\\
+(b+1-\frac{1}{5}\lambda)\tilde{D}_2+(b-\frac{4}{5}\lambda-\mu-\lambda_1)E_1+(b-\frac{4}{5}\lambda-\mu-\lambda_2)E_2,\end{gather*} where $a+b=1+\lambda+\lambda_1+\lambda_2+2\mu$. Also, we see that if $\lambda+\lambda_1+\lambda_2+2\mu>1$ then $X$ has an $H$-polar cylinder. So, we may assume that $\lambda+\lambda_1+\lambda_2+2\mu\leq 1$.
As above, we have the configuration as in Figure 2. Also, there exist two cases. The first case is when $\tilde{E}_1\cdot E_1=\tilde{E}_2\cdot E_1=1$, $\tilde{E}_1\cdot E_2=\tilde{E}_2\cdot E_2=1$, $\tilde{E}_1\cdot E_3=\tilde{E}_2\cdot E_3=0$, $\tilde{E}_1\cdot E_4=\tilde{E}_2\cdot E_5=1$, $\tilde{E}_1\cdot E_5=\tilde{E}_2\cdot E_4=0$. Then \begin{gather*}\varphi^*(H)\equiv L= bD_1+2bD_2+2aD_3+aD_4+\\
+(2b-1-\frac{3}{5}\lambda-\lambda_1-\lambda_2-2\mu)E_3+(2b-1-\frac{3}{5}\lambda-\lambda_1-\lambda_2-\mu)E_4+\\
+(2b-1-\frac{3}{5}\lambda-\mu)E_5+(2a-1-\frac{2}{5}\lambda-\lambda_1-\mu)E_1+\\
+(2a-1-\frac{2}{5}\lambda-\lambda_2-\mu)E_2+(a-\frac{6}{5}\lambda-\lambda_1-\lambda_2-2\mu)M,\end{gather*} where $a+b=1+\lambda+\lambda_1+\lambda_2+2\mu$, $M$ is a $(0)$-curve such that $M$ pass through the intersection point of $D_2$ and $D_3$. The second case is when $\tilde{E}_1\cdot E_1=\tilde{E}_2\cdot E_1=1$, $\tilde{E}_1\cdot E_2=\tilde{E}_2\cdot E_3=1$, $\tilde{E}_1\cdot E_3=\tilde{E}_2\cdot E_2=0$, $\tilde{E}_1\cdot E_4=\tilde{E}_2\cdot E_4=1$, $\tilde{E}_1\cdot E_5=\tilde{E}_2\cdot E_5=0$. Then \begin{gather*}\varphi^*(H)\equiv L= bD_1+2bD_2+2aD_3+aD_4+\\
+(2b-1-\frac{3}{5}\lambda-\lambda_1-\lambda_2-2\mu)E_3+(2b-1-\frac{3}{5}\lambda-\lambda_1-\mu)E_4+\\
+(2b-1-\frac{3}{5}\lambda-\lambda_2-\mu)E_5+(2a-1-\frac{2}{5}\lambda-\lambda_1-\lambda_2-\mu)E_1+\\
+(2a-1-\frac{2}{5}\lambda-\mu)E_2+(a-\frac{6}{5}\lambda-\lambda_1-\lambda_2-2\mu)M,\end{gather*} where $a+b=1+\lambda+\lambda_1+\lambda_2+2\mu$, $M$ is a $(0)$-curve such that $M$ pass through the intersection point of $D_2$ and $D_3$. We see that in both cases if $\lambda+\lambda_1+\lambda_2+2\mu<1$ then $X$ has an $H$-polar cylinder. So, we may assume that $\lambda+\lambda_1+\lambda_2+2\mu=1$. Consider the configuration as in Figure 2. Also, we have two cases. The first case is when $\tilde{E}_1\cdot E_1=\tilde{E}_2\cdot E_1=1$, $\tilde{E}_1\cdot E_2=\tilde{E}_2\cdot E_2=0$, $\tilde{E}_1\cdot E_3=\tilde{E}_2\cdot E_3=1$, $\tilde{E}_1\cdot E_4=\tilde{E}_2\cdot E_5=1$, $\tilde{E}_1\cdot E_5=\tilde{E}_2\cdot E_4=0$. Then \begin{gather*}\varphi^*(H)\equiv L= bD_1+(b+1-\frac{1}{5}\lambda)D_2+2aD_3+aD_4+\\
+(b-1+\frac{1}{5}\lambda)E+(b-\frac{4}{5}\lambda-\lambda_1-\lambda_2-2\mu)E_3+(b-\frac{4}{5}\lambda-\lambda_1-\mu)E_4+\\
+(b-\frac{4}{5}\lambda-\lambda_2-\mu)E_5+(2a-1-\frac{2}{5}\lambda-\lambda_1-\lambda_2-\mu)E_1+\\
+(2a-1-\frac{2}{5}\lambda-\mu)E_2,\end{gather*} where $a+b=1+\lambda+\lambda_1+\lambda_2+2\mu$. We see that $L$ is an effective divisor if and only if $$\begin{cases}a>\frac{1+\lambda_1+\lambda_2+\mu}{2}+\frac{\lambda}{5}\\
b>1-\frac{\lambda}{5}\\
b>\lambda_1+\lambda_2+\frac{4}{5}\lambda+2\mu.\end{cases}$$ We obtain $$\begin{cases}1+\lambda+\lambda_1+\lambda_2+2\mu>1+\frac{1+\lambda_1+\lambda_2+\mu}{2}\\
1+\lambda+\lambda_1+\lambda_2+2\mu>\frac{1}{2}+\frac{3}{2}(\lambda_1+\lambda_2)+\lambda+\frac{5}{2}\mu.\end{cases}$$ Hence, $$\begin{cases} 2\lambda+\lambda_1+\lambda_2+3\mu>1\\
\lambda_1+\lambda_2+\mu<1.\end{cases}$$ Since $\lambda+\lambda_1+\lambda_2+2\mu=1$, we see that both inequality are holds. Therefore, there exists $a$ and $b$ such that $L$ is an effective divisor. Then $X$ has an $H$-polar cylinder. The second case is when $\tilde{E}_1\cdot E_1=\tilde{E}_2\cdot E_2=1$, $\tilde{E}_1\cdot E_2=\tilde{E}_2\cdot E_1=0$, $\tilde{E}_1\cdot E_3=\tilde{E}_2\cdot E_3=1$, $\tilde{E}_1\cdot E_4=\tilde{E}_2\cdot E_4=1$, $\tilde{E}_1\cdot E_5=\tilde{E}_2\cdot E_5=0$. Then \begin{gather*}\varphi^*(H)\equiv L= bD_1+(b+1-\frac{1}{5}\lambda)D_2+2aD_3+aD_4+\\
+(b-1+\frac{1}{5}\lambda)E+(b-\frac{4}{5}\lambda-\lambda_1-\lambda_2-2\mu)E_3+\\
+(b-\frac{4}{5}\lambda-\lambda_1-\lambda_2-\mu)E_4+(b-\frac{4}{5}\lambda-\mu)E_5+\\
+(2a-1-\frac{2}{5}\lambda-\lambda_1-\mu)E_1+(2a-1-\frac{2}{5}\lambda-\lambda_2-\mu)E_2,\end{gather*} where $a+b=1+\lambda+\lambda_1+\lambda_2+2\mu$. We see that $L$ is an effective divisor if and only if $$\begin{cases}a>\frac{1+\lambda_1+\mu}{2}+\frac{\lambda}{5}\\
b>1-\frac{\lambda}{5}\\
b>\lambda_1+\lambda_2+\frac{4}{5}\lambda+2\mu.\end{cases}$$ We obtain $$\begin{cases} 2\lambda+\lambda_1+2\lambda_2+3\mu>1\\
\lambda_1+\mu<1.\end{cases}$$ Since $\lambda+\lambda_1+\lambda_2+2\mu=1$, we see that both inequality are holds. Therefore, there exists $a$ and $b$ such that $L$ is an effective divisor. Then $X$ has an $H$-polar cylinder.
\end{proof}

\begin{lemma}
\label{LemA4A3}
Let $X$ be a del Pezzo surface with du Val singularities and let $H$ be an ample divisor on $X$. Assume that $X$ has the following collection of singularities $A_4+A_3$. Then $X$ has an $H$-polar cylinder.
\end{lemma}
\begin{proof}
Let $\varphi\colon\tilde{X}\rightarrow X$ be the minimal resolution of singularities of $X$ and let $D=\sum_{i=1}^7 D_i$ be the exceptional divisor of $\varphi$. We may assume that $D_1,D_2,D_3,D_4$ correspond to singular point of type $A_4$ and $D_5,D_6,D_7$ correspond to singular point of type $A_3$. Since $\rho(X)=2$, we see that $H=-K_X+\lambda\varphi(E)$, where $E$ meets $D_7$ and does not meet $D_1,\ldots,D_6$. Note that $0<\lambda<4$.
We have the following configuration on $\tilde{X}$.\\
\begin{tikzpicture}
\draw  (0.5,0) -- (0.5,5) (0,0.5) -- (8,0.5) (7.5,0)--(7.5,5) (7.2,4.5)--(10,4.5) (3.5,1.5)--(5,1.5) (4.5,1)--(4.5,4) (4,3.5)--(6,3.5); \draw [dashed]  (7,1.5)--(9,1.5) (4,0)--(4,2);
\draw (1,0.5) node [above] {$D_3$} (0.5,5) node [above] {$D_4$} (7.5,5) node [above] {$D_2$}  (10,4.5) node [right] {$D_1$} (5,1.5) node [right] {$D_5$} (4.5,4) node [above] {$D_6$} (6,3.5) node [right] {$D_7$}  (9,1.5) node [right] {$E_1$} (4,2) node [above] {$E_2$};
\end{tikzpicture}$$\text{Figure 8}$$
where dotted lines are $(-1)$-curves, $D_1,\ldots,D_7$ are $(-2)$-curves. We have \begin{gather*}\varphi^*(H)\equiv L=aD_4+2aD_3+2bD_2+bD_1+(2a-1)E_1+\\
+(2b-1-\frac{3}{4}\lambda)D_7+2(2b-1-\frac{3}{4}\lambda)D_6+3(2b-1-\frac{3}{4}\lambda)D_5+\\
+4(2b-1-\frac{3}{4}\lambda)E_2+(3a-1-2\lambda)M,\end{gather*} where $a+b=1+\lambda$, $M$ is a $(0)$-curve such that $M$ pass through the intersection point of $D_2$ and $D_3$ and $M\cdot D_2=M\cdot D_3=1$. Since $\lambda<4$, we see that there exist $a$ and $b$ such that $L$ is an effective divisor. Let $Y_1\rightarrow\tilde{X}$ be the blow up of the intersection point of $D_2$ and $D_3$, and $N_1$ be the exceptional divisor. Let $Y_2\rightarrow Y_1$ be the blow up of the intersection point of $N_1$ and the proper transform of $D_3$, and $N_2$ be the exceptional divisor. Let $Y_3\rightarrow Y_2$ be the blow up of the intersection point of $N_2$ and the proper transform of $D_3$, and $N_3$ be the exceptional divisor. Let $Y_4\rightarrow Y_3$ be the blow up of the intersection point of $N_3$ and the proper transform of $N_2$, and $F$ be the exceptional divisor.
We obtain a surface $Y_4$ with the following configuration\\
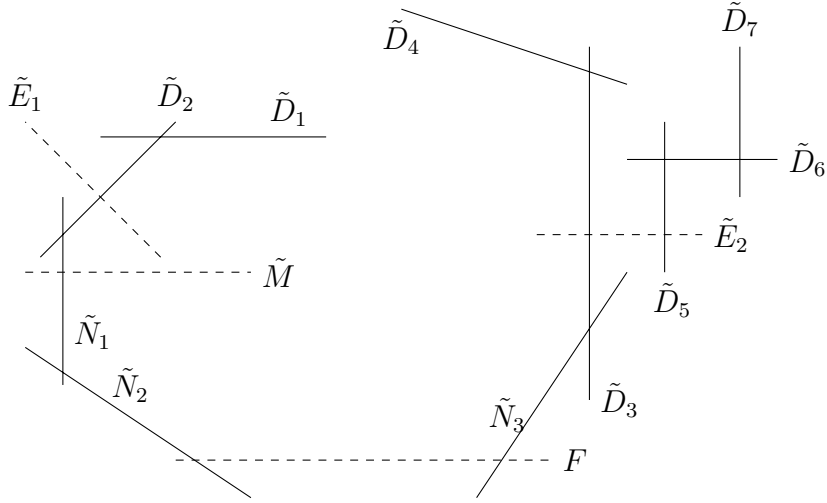
\begin{tikzpicture}
\draw  (3,0) -- (0,2) (0.5,1.5) -- (0.5,4) (0.2,3.2)--(2,5) (1,4.8)--(4,4.8) (6,0)--(8,3) (7.5,1.3)--(7.5,6) (8,5.5)--(5,6.5) (8.5,3)--(8.5,5) (8,4.5)--(10,4.5) (9.5,4)--(9.5,6); \draw [dashed] (2,0.5) -- (7,0.5) (0,3)--(3,3) (1.8,3.2)--(0,5) (6.8,3.5)--(9,3.5);
\draw (7.5,1.3) node [right] {$\tilde{D}_3$} (2,5) node [above] {$\tilde{D}_2$} (3.5,4.8) node [above] {$\tilde{D}_1$}  (5,6.5) node [below] {$\tilde{D}_4$} (0,5) node [above] {$\tilde{E}_1$} (9,3.5) node [right] {$\tilde{E}_2$} (3,3) node [right] {$\tilde{M}$} (7,0.5) node [right] {$F$} (0.5,2.2) node [right] {$\tilde{N}_1$} (1,1.5) node [right] {$\tilde{N}_2$} (6.8,1.1) node [left] {$\tilde{N}_3$} (8.6,3) node [below] {$\tilde{D}_5$} (10,4.5) node [right] {$\tilde{D}_6$} (9.5,6) node [above] {$\tilde{D}_7$};
\end{tikzpicture}\\
where $\tilde{D}_1,\ldots,\tilde{D}_7$ are the proper transforms of $D_1,\ldots, D_7$, $\tilde{E}_1, \tilde{E}_2$ are the proper transforms of $E_1,E_2$, $\tilde{N}_1, \tilde{N}_2, \tilde{N}_3$ are the proper transforms of $N_1, N_2, N_3$, $\tilde{M}$ is the proper transform of $M$. Note that there exists $\PP^1$-fibration $g\colon Y_4\rightarrow\PP^1$ such that $g$ has only two singular fibers $C_1, C_2$ and $$C_1=\tilde{D}_1+\tilde{N}_2+2\tilde{D}_2+2\tilde{E}_1+3\tilde{N}_1+3\tilde{M},$$ $$C_2=\tilde{D}_4+\tilde{N}_3+2\tilde{D}_3+2\tilde{D}_5+4\tilde{D}_6+6\tilde{D}_7+8\tilde{E}_2.$$
So, $$\tilde{X}\backslash\Supp(L)\cong\FF_1\backslash(K+F_1+F_2),$$ where $K$ is a unique $(-1)$-curve and $F_1, F_2$ are fibers. Then $X$ has an $H$-polar cylinder.
\end{proof}

\begin{lemma}
\label{LemA4A2A1}
Let $X$ be a del Pezzo surface with du Val singularities and let $H$ be an ample divisor on $X$. Assume that $X$ has the following collection of singularities $A_4+A_2+A_1$. Then $X$ has an $H$-polar cylinder.
\end{lemma}
\begin{proof}
Let $\varphi\colon\tilde{X}\rightarrow X$ be the minimal resolution of singularities of $X$ and let $D=\sum_{i=1}^7 D_i$ be the exceptional divisor of $\varphi$. We may assume that $D_1,D_2,D_3,D_4$ correspond to the singular point of type $A_4$, $D_5,D_6$ correspond to the singular point of type $A_2$, $D_7$ correspond to the singular point of type $A_1$. Since $\rho(X)=2$, we see that $H=-K_X+\lambda\varphi(E)$, where $E$ meets $D_1$ and does not meet $D_2,\ldots,D_7$. Note that $0<\lambda<5$. Note that $$\varphi^*(H)=-K_{\tilde{X}}+\lambda E+\frac{4}{5}\lambda D_1+\frac{3}{5}\lambda D_2+\frac{2}{5}\lambda D_3+\frac{1}{5}\lambda D_4.$$
We have the following configuration on $\tilde{X}$.\\
\begin{tikzpicture}
\draw  (0.5,0) -- (0.5,5) (0,0.5) -- (8,0.5) (7.5,0)--(7.5,5) (7.2,4.5)--(10,4.5) (3.5,1.5)--(5,1.5) (4.5,1)--(4.5,4) (9.5,1)--(9.5,3); \draw [dashed]  (7,1.5)--(10,1.5) (4,0)--(4,2) (0,4)--(3,4);
\draw (1,0.5) node [above] {$D_2$} (0.5,5) node [above] {$D_1$} (7.5,5) node [above] {$D_3$}  (10,4.5) node [right] {$D_4$} (5,1.5) node [right] {$D_5$} (4.5,4) node [above] {$D_6$} (9.5,3) node [above] {$D_7$}  (10,1.5) node [right] {$E_1$} (4,2) node [above] {$E_2$} (3,4) node [right] {$E$};
\end{tikzpicture}$$\text{Figure 9}$$
where dotted lines are $(-1)$-curves, $D_1,\ldots,D_7$ are $(-2)$-curves. We have \begin{gather*}\varphi^*(H)\equiv L=(a-2+\lambda)E+(a-1+\frac{4}{5}\lambda)D_1+(a+\frac{3}{5}\lambda)D_2+\\
+(a-1)D_6+(2a-2)D_5+(3a-3)E_2+(b-1+\frac{1}{5}\lambda)D_4+\\
+(2b-2+\frac{2}{5}\lambda)D_3+(2b-3)D_7+(4b-6)E_1,\end{gather*} where $a+b=3$. Note that if $\lambda>\frac{1}{2}$ then there exist $a$ and $b$ such that $L$ is an effective divisor. Let $Y_1\rightarrow\tilde{X}$ be the blow up of intersection point of $D_2$ and $D_3$, and let $N$ be the exceptional divisor. Let $Y_2\rightarrow Y_1$ be the blow up of intersection point of $N$ and the proper transform of $D_2$, and let $F$ be the exceptional divisor. Then there exists $\PP^1$-fibration $g\colon Y_2\rightarrow\PP^1$ such that $g$ has only two singular fibers $C_1, C_2$ and $$C_1=\tilde{E}+\tilde{D}_1+\tilde{D}_2+\tilde{D}_6+2\tilde{D}_5+3\tilde{E}_2,$$ $$C_2=\tilde{D}_4+\tilde{N}+2\tilde{D}_3+2\tilde{D}_7+4\tilde{E}_1,$$ where $\tilde{D}_1,\ldots,\tilde{D}_7$ are the proper transforms of $D_1,\ldots, D_7$, $\tilde{E}_1, \tilde{E}_2$ are the proper transforms of $E_1,E_2$, $\tilde{E}$ is the proper transform of $E$, $\tilde{N}$ is the proper transform of $N$. On the other hand, $$\tilde{X}\backslash\Supp(L)\cong\FF_1\backslash(K+F_1+F_2),$$ where $K$ is a unique $(-1)$-curve and $F_1, F_2$ are fibers. Then $X$ has an $H$-polar cylinder. Assume that $\lambda\leq\frac{1}{2}$. We have \begin{gather*}\varphi^*(H)\equiv L=aD_1+2aD_2+(2a-1-\frac{3}{5}\lambda)D_6+2(2a-1-\frac{3}{5}\lambda)D_5+\\
+3(2a-1-\frac{3}{5}\lambda)E_2+bD_4+2bD_3+(2b-1-\frac{2}{5}\lambda)D_7+\\
+2(2b-1-\frac{2}{5}\lambda)E_1+(b-\frac{6}{5}\lambda)M,\end{gather*}
where $a+b=1+\lambda$, $M$ is a $(0)$-curve such that $M$ pass through the intersection point of $D_2$ and $D_3$ and $M\cdot D_2=M\cdot D_3=1$. Note that if $\lambda<1$ then there exist $a$ and $b$ such that $L$ is an effective divisor. Let $Y_1\rightarrow\tilde{X}$ be the blow up of the intersection point of $D_2$ and $D_3$, and $N_1$ be the exceptional divisor. Let $Y_2\rightarrow Y_1$ be the blow up of the intersection point of $N_1$ and the proper transform of $D_2$, and $N_2$ be the exceptional divisor. Let $Y_3\rightarrow Y_2$ be the blow up of the intersection point of $N_2$ and the proper transform of $N_1$, and $F$ be the exceptional divisor.
We obtain a surface $Y_3$ with the following configuration\\
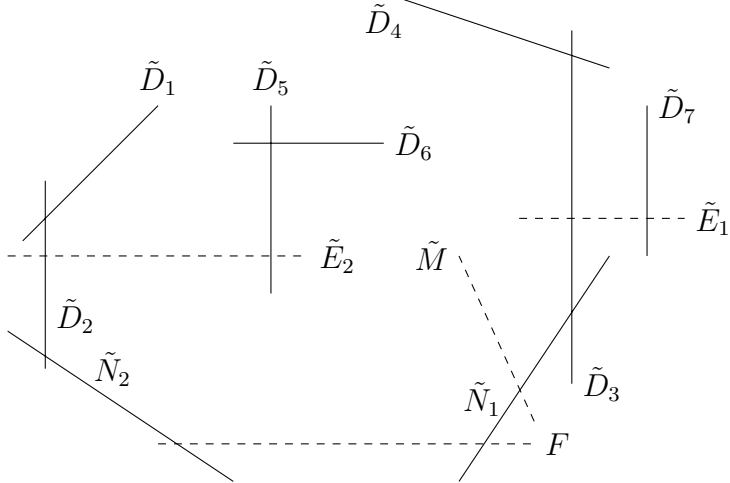
\begin{tikzpicture}
\draw  (3,0) -- (0,2) (0.5,1.5) -- (0.5,4) (0.2,3.2)--(2,5) (6,0)--(8,3) (7.5,1.3)--(7.5,6) (8,5.5)--(5,6.5) (8.5,3)--(8.5,5) (3.5,2.5)--(3.5,5) (3,4.5)--(5,4.5); \draw [dashed] (2,0.5) -- (7,0.5) (0,3)--(4,3) (6.8,3.5)--(9,3.5) (7,0.8)--(6,3);
\draw (7.5,1.3) node [right] {$\tilde{D}_3$} (2,5) node [above] {$\tilde{D}_1$} (3.5,5) node [above] {$\tilde{D}_5$}  (5,6.5) node [below] {$\tilde{D}_4$} (9,3.5) node [right] {$\tilde{E}_1$} (4,3) node [right] {$\tilde{E}_2$} (7,0.5) node [right] {$F$} (0.5,2.2) node [right] {$\tilde{D}_2$} (1,1.5) node [right] {$\tilde{N}_2$} (6.7,1.1) node [left] {$\tilde{N}_1$} (8.5,5) node [right] {$\tilde{D}_7$} (5,4.5) node [right] {$\tilde{D}_6$} (6,3) node [left] {$\tilde{M}$};
\end{tikzpicture}\\
where $\tilde{D}_1,\ldots,\tilde{D}_7$ are the proper transforms of $D_1,\ldots, D_7$, $\tilde{E}_1, \tilde{E}_2$ are the proper transforms of $E_1,E_2$, $\tilde{N}_1, \tilde{N}_2$ are the proper transforms of $N_1, N_2$, $\tilde{M}$ is the proper transform of $M$. Note that there exists $\PP^1$-fibration $g\colon Y_3\rightarrow\PP^1$ such that $g$ has only two singular fibers $C_1, C_2$ and
$$C_1=\tilde{D}_1+\tilde{N}_2+2\tilde{D}_2+2\tilde{D}_6+4\tilde{D}_5+6\tilde{E}_2,$$ $$C_2=\tilde{D}_4+\tilde{N}_1+\tilde{M}+2\tilde{D}_3+2\tilde{D}_7+4\tilde{E}_1.$$
On the other hand, $$\tilde{X}\backslash\Supp(L)\cong\FF_1\backslash(K+F_1+F_2),$$ where $K$ is a unique $(-1)$-curve and $F_1, F_2$ are fibers. Then $X$ has an $H$-polar cylinder.
\end{proof}

\begin{lemma}
\label{LemA4A1}
Let $X$ be a del Pezzo surface with du Val singularities and let $H$ be an ample divisor on $X$. Assume that $X$ has the following collection of singularities $A_4+A_1$. Then $X$ has an $H$-polar cylinder.
\end{lemma}
\begin{proof}
Let $\varphi\colon\tilde{X}\rightarrow X$ be the minimal resolution of singularities of $X$ and let $D=\sum_{i=1}^5 D_i$ be the exceptional divisor of $\varphi$. We may assume that $D_1,D_2,D_3,D_4$ correspond to the singular point of type $A_4$, $D_5$ correspond to the singular point of type $A_1$. Assume that $H=-K_X+\lambda E$, where $E$ ia a $(-1)$-curve and $\lambda<1$. We have the following configuration on $\tilde{X}$.\\
\begin{tikzpicture}
\draw  (0.5,0) -- (0.5,3) (0,0.5) -- (8,0.5) (7.5,0)--(7.5,4) (7.2,3.5)--(10,3.5) (5,1.5)--(6,1.5); \draw [dashed] (3,0)--(3,3) (4,0)--(4,3) (7,1.5)--(10,1.5) (5.5,0)--(5.5,2) ;
\draw (1,0.5) node [above] {$D_2$} (0.5,3) node [above] {$D_1$} (7.5,4) node [above] {$D_3$}  (10,3.5) node [right] {$D_4$} (6,1.5) node [right] {$D_5$}  (10,1.5) node [right] {$E_1$} (5.5,2) node [above] {$E_4$} (3,3) node [above] {$E_2$} (4,3) node [above] {$E_3$};
\end{tikzpicture}$$\text{Figure 10}$$
where dotted lines are $(-1)$-curves, $D_1,\ldots,D_5$ are $(-2)$-curves. Moreover, $E\cdot E_1=E\cdot E_2=0$, $E\cdot E_3=E\cdot E_4=1$. We have \begin{gather*}\varphi^*(H)\equiv L=aD_1+2aD_2+2bD_3+bD_4+(2b-1)E_1+(2a-1)E_2+\\
+(2a-1-\lambda)E_3+(2a-1-\lambda)D_5+2(2a-1-\lambda)E_5+(3b-1-2\lambda)M,\end{gather*} where $a+b=1+\lambda$, $M$ is a $(0)$-curve such that $M$ pass through the intersection point of $D_2$ and $D_3$ and $M\cdot D_2=M\cdot D_3=1$. Since $\lambda<1$, we see that there exist $a$ and $b$ such that $L$ is an effective divisor. Let $Y_1\rightarrow\tilde{X}$ be the blow up of the intersection point of $D_2$ and $D_3$, and $N_1$ be the exceptional divisor. Let $Y_2\rightarrow Y_1$ be the blow up of the intersection point of $N_1$ and the proper transform of $D_2$, and $N_2$ be the exceptional divisor. Let $Y_3\rightarrow Y_2$ be the blow up of the intersection point of $N_2$ and the proper transform of $D_2$, and $N_3$ be the exceptional divisor. Let $Y_4\rightarrow Y_3$ be the blow up of the intersection point of $N_3$ and the proper transform of $N_2$, and $F$ be the exceptional divisor.
We obtain a surface $Y_4$ with the following configuration\\
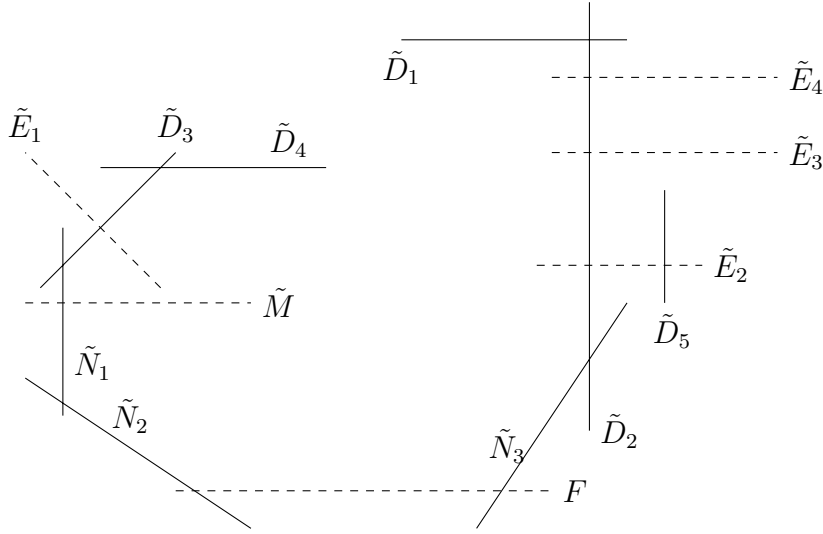
\begin{tikzpicture}
\draw  (3,0) -- (0,2) (0.5,1.5) -- (0.5,4) (0.2,3.2)--(2,5) (1,4.8)--(4,4.8) (6,0)--(8,3) (7.5,1.3)--(7.5,7) (8,6.5)--(5,6.5) (8.5,3)--(8.5,4.5); \draw [dashed] (2,0.5) -- (7,0.5) (0,3)--(3,3) (1.8,3.2)--(0,5) (6.8,3.5)--(9,3.5) (7,5)--(10,5) (7,6)--(10,6);
\draw (7.5,1.3) node [right] {$\tilde{D}_2$} (2,5) node [above] {$\tilde{D}_3$} (3.5,4.8) node [above] {$\tilde{D}_4$}  (5,6.5) node [below] {$\tilde{D}_1$} (0,5) node [above] {$\tilde{E}_1$} (9,3.5) node [right] {$\tilde{E}_2$} (3,3) node [right] {$\tilde{M}$} (7,0.5) node [right] {$F$} (0.5,2.2) node [right] {$\tilde{N}_1$} (1,1.5) node [right] {$\tilde{N}_2$} (6.8,1.1) node [left] {$\tilde{N}_3$} (8.6,3) node [below] {$\tilde{D}_5$} (10,5) node [right] {$\tilde{E}_3$} (10,6) node [right] {$\tilde{E}_4$};
\end{tikzpicture}\\ where $\tilde{D}_1,\ldots,\tilde{D}_5$ are the proper transforms of $D_1,\ldots, D_5$, $\tilde{E}_1, \tilde{E}_2, \tilde{E}_3, \tilde{E}_4$ are the proper transforms of $E_1,E_2,E_3,E_4$, $\tilde{N}_1, \tilde{N}_2, \tilde{N}_3$ are the proper transforms of $N_1, N_2, N_3$, $\tilde{M}$ is the proper transform of $M$. Note that there exists $\PP^1$-fibration $g\colon Y_4\rightarrow\PP^1$ such that $g$ has only two singular fibers $C_1, C_2$ and $$C_1=\tilde{D}_1+\tilde{N}_3+2\tilde{D}_2+2\tilde{D}_5+2\tilde{E}_3+2\tilde{E}_4+4\tilde{E}_2,$$ $$C_2=\tilde{D}_4+\tilde{N}_2+2\tilde{D}_3+2\tilde{E}_1+3\tilde{N}_1+3\tilde{M}.$$
On the other hand, $$\tilde{X}\backslash\Supp(L)\cong\FF_1\backslash(K+F_1+F_2),$$ where $K$ is a unique $(-1)$-curve and $F_1, F_2$ are fibers. Then $X$ has an $H$-polar cylinder.
Assume that $H=-K_X+\lambda \varphi(E)$, where $E$ ia a $(-1)$-curve on $\tilde{X}$ that meets only $D_5$ and does not meet $D_1,D_2,D_3,D_4$. Note that $\lambda<2$.  We have the following configuration on $\tilde{X}$.\\
\begin{tikzpicture}
\draw  (0.5,0) -- (0.5,3) (0,0.5) -- (8,0.5) (7.5,0)--(7.5,4) (7.2,3.5)--(10,3.5) (3,1.5)--(5,1.5); \draw [dashed] (7,1.5)--(10,1.5) (7,2)--(10,2) (7,2.5)--(10,2.5) (4,0)--(4,2) ;
\draw (1,0.5) node [above] {$D_2$} (0.5,3) node [above] {$D_1$} (7.5,4) node [above] {$D_3$}  (10,3.5) node [right] {$D_4$} (5,1.5) node [right] {$D_5$}  (10,1.5) node [right] {$E_1$} (4,2) node [above] {$E_4$} (10,2) node [right] {$E_2$} (10,2.5) node [right] {$E_3$};
\end{tikzpicture}$$\text{Figure 11}$$
where dotted lines are $(-1)$-curves, $D_1,\ldots,D_5$ are $(-2)$-curves. Moreover, $E\cdot E_1=E\cdot E_4=0$, $E\cdot E_2=E\cdot E_3=1$. We have \begin{gather*}\varphi^*(H)\equiv L=aD_1+2aD_2+2bD_3+bD_4+(2a-1-\frac{1}{2}\lambda)D_5+\\
+2(2a-1-\frac{1}{2}\lambda)E_4+(2b-1-\lambda)E_2+(2b-1-\lambda)E_3+\\
+2(2b-1)E_1+(a-\lambda)M,\end{gather*} where $a+b=1+\lambda$, $M$ is a $(0)$-curve such that $M$ pass through the intersection point of $D_2$ and $D_3$ and $M\cdot D_2=M\cdot D_3=1$. If $\lambda<1$, then there exist $a$ and $b$ such that $L$ is an effective divisor. As above, $$\tilde{X}\backslash\Supp(L)\cong\FF_1\backslash(K+F_1+F_2),$$ where $K$ is a unique $(-1)$-curve and $F_1, F_2$ are fibers. Then $X$ has an $H$-polar cylinder. So, we may assume that $\lambda\geq 1$. We have the following configuration on $\tilde{X}$.\\
\begin{tikzpicture}
\draw  (0.5,0) -- (0.5,3) (0,0.5) -- (8,0.5) (7.5,0)--(7.5,4) (7.2,3.5)--(11,3.5) (3,1.5)--(5,1.5); \draw [dashed] (7,1.5)--(9,1.5) (7,2)--(9,2) (7,2.5)--(9,2.5) (4,0)--(4,2) (10.5,4)--(10.5,2);
\draw (1,0.5) node [above] {$D_2$} (0.5,3) node [above] {$D_1$} (7.5,4) node [above] {$D_3$}  (11,3.5) node [right] {$D_4$} (5,1.5) node [right] {$D_5$}  (9,1.5) node [right] {$E_1$} (4,2) node [above] {$E_4$} (9,2) node [right] {$E_2$} (9,2.5) node [right] {$E_3$} (10.5,2) node [right] {$E_5$};
\end{tikzpicture}$$\text{Figure 12}$$
where dotted lines are $(-1)$-curves, $D_1,\ldots,D_5$ are $(-2)$-curves. Moreover, $E\cdot E_3=E\cdot E_4=E\cdot E_5=0$, $E\cdot E_1=E\cdot E_2=1$. We have \begin{gather*}\varphi^*(H)\equiv L=aD_1+2aD_2+(b+1)D_3+bD_4+(b-1)E_5+bE_3+\\
+(b-\lambda)E_2+(b-\lambda)E_1+2(2a-1-\frac{1}{2}\lambda)E_4+(2a-1-\frac{1}{2}\lambda)D_5.\end{gather*} Note that if $\lambda>\frac{2}{3}$ then there exist $a$ an $b$ such that $L$ is an effective divisor. Let $Y_1\rightarrow\tilde{X}$ be the blow up of intersection point of $D_2$ and $D_3$, and let $N$ be the exceptional divisor. Let $Y_2\rightarrow Y_1$ be the blow up of intersection point of $N$ and the proper transform of $D_3$, and let $F$ be the exceptional divisor. Note that there exists $\PP^1$-fibration $g\colon Y_2\rightarrow\PP^1$ such that $g$ has only two singular fibers $C_1, C_2$ and $$C_1=\tilde{D}_1+\tilde{N}+2\tilde{D}_2+2\tilde{D}_5+4\tilde{E}_4,$$ $$C_2=\tilde{D}_3+\tilde{D}_4+\tilde{E}_1+\tilde{E}_2+\tilde{E}_3+\tilde{E}_5,$$ where $\tilde{D}_1,\ldots,\tilde{D}_5$ are the proper transforms of $D_1,\ldots, D_5$, $\tilde{E}_1,\ldots, \tilde{E}_5$ are the proper transforms of $E_1,\ldots,E_4$, $\tilde{N}$ is the proper transform of $N$.
On the other hand, $$\tilde{X}\backslash\Supp(L)\cong\FF_1\backslash(K+F_1+F_2),$$ where $K$ is a unique $(-1)$-curve and $F_1, F_2$ are fibers. Then $X$ has an $H$-polar cylinder.
Assume that $H=-K_X+\lambda \varphi(E)$, where $E$ ia a $(-1)$-curve on $\tilde{X}$ that meets only $D_4$ and does not meet $D_1,D_2,D_3,D_5$. Note that $\lambda<5$. We have the configuration as in Figure 12. We may assume that $E=E_5$. We have \begin{gather*}\varphi^*(H)\equiv L=(a-1+\frac{1}{5}\lambda)D_1+(2a-2+\frac{2}{5}\lambda)D_2+ (2a-3)D_5+\\
+(4a-6)E_4+(b+\frac{3}{5}\lambda)D_3+(b-1+\frac{4}{5}\lambda)D_4+(b-1)E_1+\\
+(b-1)E_2+(b-1)E_3+(b-2+\lambda)E_5,\end{gather*} where $a+b=3$. If $\lambda>\frac{1}{2}$ then there exist $a$ and $b$ such that $L$ is an effective divisor. As above, $X$ has an $H$-polar cylinder. So, we may assume that $\lambda\leq\frac{1}{2}$. We have the configuration as in Figure 11. We obtain \begin{gather*}\varphi^*(H)\equiv L=aD_1+2aD_2+2bD_3+bD_4+(2a-1-\frac{2}{5}\lambda)D_5+\\
+2(2a-1-\frac{2}{5}\lambda)E_4+(2b-1-\frac{3}{5}\lambda)E_1+(2b-1-\frac{3}{5}\lambda)E_2+\\
+(2b-1-\frac{3}{5}\lambda)E_3+(a-\frac{6}{5}\lambda)M,\end{gather*} where $a+b=1+\lambda$, $M$ is a $(0)$-curve such that $M$ pass through the intersection point of $D_2$ and $D_3$ and $M\cdot D_2=M\cdot D_3=1$. If $\lambda<1$, then there exist $a$ and $b$ such that $L$ is an effective divisor. As above, $$\tilde{X}\backslash\Supp(L)\cong\FF_1\backslash(K+F_1+F_2),$$ where $K$ is a unique $(-1)$-curve and $F_1, F_2$ are fibers. Then $X$ has an $H$-polar cylinder. In the same way, we can construct cylinders when $E$ meets $D_1$ and does not meet $D_2,\ldots,D_5$. We omit this case.

Assume that $H=-K_X+\lambda_1 \hat{E}_1+\lambda_2 \hat{E}_2$, where $\hat{E}_1,\hat{E}_2$ are $(-1)$-curves and $0<\lambda_2\leq\lambda_1<1$. Put $\tilde{E}_1, \tilde{E}_2$ are the proper transform of $\hat{E}_1,\hat{E}_2$. We have the following configuration on $\tilde{X}$.\\
\begin{tikzpicture}
\draw  (0.5,0) -- (0.5,5) (0,0.5) -- (6,0.5) (5.5,0)--(5.5,5) (5.2,4.5)--(10,4.5) (1.5,3.8)--(2.7,5); \draw [dashed] (0,4)--(2,4) (4,0) -- (4,3) (3.8,2)--(5,4) (3,0)--(3,3) (2.8,2)--(4,4) (5,1.5)--(7,1.5) (9,5)--(9,2);
\draw (1,0.5) node [above] {$D_3$} (0.5,5) node [above] {$D_4$} (5.5,5) node [above] {$D_2$}  (10,4.5) node [right] {$D_1$} (4,3) node [above] {$E_2$} (3,3) node [above] {$E_1$} (7,1.5) node [right] {$E_3$} (2.7,5) node [right] {$D_5$} (2,4) node [right] {$E_5$} (9,2) node [right] {$E_4$} (5,4) node [above] {$\tilde{E}_2$} (4,4) node [above] {$\tilde{E}_1$};
\end{tikzpicture}$$\text{Figure 13}$$ where dotted lines are $(-1)$-curves, $D_1,\ldots,D_5$ are $(-2)$-curves. Moreover, $\tilde{E}_1, \tilde{E}_2$ are the proper transform of $\hat{E}_1,\hat{E}_2$. We have \begin{gather*}\varphi^*(H)\equiv L=2D_3+aD_2+(a-1)D_1+(a-2)E_4+(a-1)E_3+\\
+bD_4+(b-1)D_5+(2b-2)E_5+c_1E_1+(c_1-1+\lambda_1)\tilde{E}_1+\\
+c_2E_2+(c_2-1+\lambda_2)\tilde{E}_2,\end{gather*} where $a+b+c_1+c_2=4$. If $\lambda_1+\lambda_2>1$ then there exist $a$ and $b$ such that $L$ is an effective divisor. As above, $X$ has an $H$-polar cylinder. So, we may assume that $\lambda_1+\lambda_2\leq 1$. We have the following configuration on $\tilde{X}$.\\
\begin{tikzpicture}
\draw  (0.5,0) -- (0.5,3) (0,0.5) -- (8,0.5) (7.5,0)--(7.5,4) (7.2,3.5)--(10,3.5) (5,1.5)--(6,1.5); \draw [dashed] (4,0)--(4,3) (7,1.5)--(10,1.5) (7,2)--(10,2) (5.5,0)--(5.5,2) ;
\draw (1,0.5) node [above] {$D_2$} (0.5,3) node [above] {$D_1$} (7.5,4) node [above] {$D_3$}  (10,3.5) node [right] {$D_4$} (6,1.5) node [right] {$D_5$}  (10,1.5) node [right] {$E_1$} (5.5,2) node [above] {$E_4$} (10,2) node [right] {$E_2$} (4,3) node [above] {$E_3$};
\end{tikzpicture}$$\text{Figure 14}$$ where dotted lines are $(-1)$-curves, $D_1,\ldots,D_5$ are $(-2)$-curves. Moreover, $\tilde{E}_1\cdot E_1=\tilde{E}_2\cdot E_2=\tilde{E}_1\cdot E_4=\tilde{E}_2\cdot E_4=1$, $\tilde{E}_1\cdot E_2=\tilde{E}_1\cdot E_2=\tilde{E}_1\cdot E_3=\tilde{E}_2\cdot E_3=0$. We have \begin{gather*}\varphi^*(H)\equiv L=aD_1+2aD_2+(2a-1)E_3+(2a-1-\lambda_1-\lambda_2)D_5+\\
+2(2a-1-\lambda_1-\lambda_2)E_4+bD_4+2bD_3+(2b-1-\lambda_1)E_1+\\
+(2b-1-\lambda_2)E_2+(b-\lambda_1-\lambda_2)M,\end{gather*} where $a+b=1+\lambda_1+\lambda_2$, $M$ is a $(0)$-curve such that $M$ pass through the intersection point of $D_2$ and $D_3$ and $M\cdot D_2=M\cdot D_3=1$. Assume that $\lambda_1+\lambda_2<1$. We see that there exist $a$ and $b$ such that $L$ is an effective divisor. As above, $X$ has an $H$-polar cylinder. So, we may assume that $\lambda_1+\lambda_2=1$. We have the following configuration on $\tilde{X}$.\\ \begin{tikzpicture}
\draw  (0.5,0) -- (0.5,3) (0,0.5) -- (8,0.5) (7.5,0)--(7.5,4) (7.2,3.5)--(10,3.5) (5,1.5)--(6,1.5); \draw [dashed] (0,2)--(3,2)(4,0)--(4,3) (7,1.5)--(10,1.5) (7,2)--(10,2) (5.5,0)--(5.5,2) ;
\draw (1,0.5) node [above] {$D_2$} (0.5,3) node [above] {$D_1$} (7.5,4) node [above] {$D_3$}  (10,3.5) node [right] {$D_4$} (6,1.5) node [right] {$D_5$}  (10,1.5) node [right] {$E_1$} (5.5,2) node [above] {$E_4$} (10,2) node [right] {$E_2$} (4,3) node [above] {$E_3$} (3,2) node [right] {$E_5$};
\end{tikzpicture}$$\text{Figure 15}$$ where dotted lines are $(-1)$-curves, $D_1,\ldots,D_5$ are $(-2)$-curves. Moreover, $\tilde{E}_1\cdot E_5=\tilde{E}_2\cdot E_5=0$. We have \begin{gather*}\varphi^*(H)\equiv L=aD_1+(a-1)E_5+(a+1)D_2+aE_3+(a-\lambda_1-\lambda_2)D_5+\\
+2(a-\lambda_1-\lambda_2)E_4+bD_4+2bD_3+(2b-1-\lambda_1)E_1+(2b-1-\lambda_2)E_2,\end{gather*} where $a+b=1+\lambda_1+\lambda_2$. Since $\lambda_1+\lambda_2=1$, we see that there exist $a$ and $b$ such that $L$ is an effective divisor. As above, $X$ has an $H$-polar cylinder.

Assume that $H=-K_X+\lambda_1 \hat{E}_1+\lambda \varphi(E)$, where $E$ ia a $(-1)$-curve on $\tilde{X}$ that meets only $D_5$ and does not meet $D_1,D_2,D_3,D_4$. We have the configuration as in Figure 14. Moreover, $E\cdot E_1=E\cdot E_3=1$, $E\cdot E_2=E\cdot E_4=0$, $\tilde{E}_1\cdot E_4=\tilde{E}_1\cdot E_1=1$, $\tilde{E}_1\cdot E_3=\tilde{E}_1\cdot E_2=0$. We have \begin{gather*}\varphi^*(H)\equiv L=aD_1+2aD_2+(2a-1-\lambda)E_3+(2a-1-\lambda_1-\frac{1}{2}\lambda)D_5+\\
+2(2a-1-\lambda_1-\frac{1}{2}\lambda)E_4+bD_4+2bD_3+(2b-1-\lambda_1-\lambda)E_1+\\
+(2b-1)E_2+(b-\lambda_1-\lambda)M,\end{gather*} where $a+b=1+\lambda_1+\lambda_2$, $M$ is a $(0)$-curve such that $M$ pass through the intersection point of $D_2$ and $D_3$ and $M\cdot D_2=M\cdot D_3=1$. Assume that $\lambda>2\lambda_1$ and $\lambda<1$. We see that there exist $a$ and $b$ such that $L$ is an effective divisor. As above, $X$ has an $H$-polar cylinder. Assume that $\lambda\leq2\lambda_1$ and $2\lambda_1+\lambda<2$. We see that there exist $a$ and $b$ such that $L$ is an effective divisor. As above, $X$ has an $H$-polar cylinder. Hence, if $\lambda+\lambda_1\leq 1$, then $X$ has an $H$-polar cylinder. So, we may assume that $\lambda+\lambda_1>1$. We have the configuration as in Figure 15. Note that $E\cdot E_5=\tilde{E}_1\cdot E_5=0$. We have \begin{gather*}\varphi^*(H)\equiv L=aD_1+(a-1)E_5+(a+1)D_2+(a-\lambda)E_3+\\
+(a-\lambda_1-\frac{1}{2}\lambda)D_5+2(a-\lambda_1-\frac{1}{2}\lambda)E_4+bD_4+2bD_3+\\
+(2b-1-\lambda_1-\lambda)E_1+(2b-1)E_2,\end{gather*} where $a+b=1+\lambda_1+\lambda$. Note that $L$ is an effective divisor if $$\begin{cases} a>\lambda\\
a>\lambda_1+\frac{1}{2}\lambda\\
b>\frac{1}{2}(1+\lambda+\lambda_1).\end{cases}$$ Assume that $\lambda\leq 2\lambda_1$. Since $\lambda_1<1$ we see that there exist $a$ and $b$ such that $L$ is an effective divisor. As above, $X$ has an $H$-polar cylinder. So, we may assume that $\lambda> 2\lambda_1$. We obtain if $\lambda<1+\lambda_1$ then there exist $a$ and $b$ such that $L$ is an effective divisor. Hence, we may assume that $\lambda\geq 1+\lambda_1$. We have the following configuration on $\tilde{X}$.
\begin{tikzpicture}
\draw  (0,0.5) -- (10,0.5) (8.5,0) -- (10.5,2) (10,1)--(10,5) (10.5,4)--(8.5,6) (0.5,2)--(0.5,5); \draw [dashed] (2,0) -- (0,3) (0,4.5)--(3,4.5) (6,0) -- (4,3) (4.5,2)--(4.5,5) (8,0) -- (6,3) (6.5,2)--(6.5,5) (9,3)--(12,3) (9.8,0.7)--(9,2);
\draw (1.5,1.5) node [above] {$E_3$} (0.5,5) node [above] {$D_5$} (3,4.5) node [right] {$E$} (5.5,1.5) node [above] {$E_1$} (4.5,5) node [above] {$\tilde{E}_1$} (7.5,1.5) node [above] {$E_2$} (6.5,5) node [above] {$\tilde{E}_2$} (12,3) node [right] {$E_4$} (0,0.5) node [above] {$D_1$} (10.5,2) node [right] {$D_2$} (10,5) node [above] {$D_3$} (8.5,6) node [above] {$D_4$} (9,2) node [left] {$E_5$} ;
\end{tikzpicture}$$\text{Figure 16}.$$ where dotted lines are $(-1)$-curves, $D_1,\ldots,D_5$ are $(-2)$-curves. We have $$\lambda E+\frac{1}{2}\lambda D_5\equiv -\lambda E_3-\frac{1}{2}\lambda D_5+\lambda D_2+\lambda D_4+2\lambda D_3+2\lambda E_4,$$ $$\lambda_1\tilde{E}_1\equiv-\lambda_1 E_1+\lambda_1 D_2+\lambda_1 D_4+2\lambda_1 D_3+2\lambda_1 E_4.$$ We obtain \begin{gather*}\varphi^*(H)\equiv L=aD_1+(a-1-\lambda_1)E_1+(a-1)E_2+(a-1-\frac{1}{2}\lambda)D_5+\\
+(2a-2-\lambda)E_3+(b+\lambda+\lambda_1)D_2+(b-1+\lambda+\lambda_1)D_4+\\
+(2b-2+2\lambda+2\lambda_1)D_3+(2b-3+2\lambda+2\lambda_1)E_4+bM,\end{gather*} where $a+b=2$, $M$ is a $(0)$-curve such that $M$ pass through the intersection point of $D_1$ and $D_2$ and $M\cdot D_1=M\cdot D_2=1$. Since $\lambda\geq 1+\lambda_1$, we see that there exist $a$ and $b$ such that $L$ is an effective divisor. Let $Y_1\rightarrow\tilde{X}$ be the blow up of intersection point of $D_1$ and $D_2$, and let $N$ be the exceptional divisor. Let $Y_2\rightarrow Y_1$ be the blow up of intersection point of $N$ and the proper transform of $D_1$, and let $F$ be the exceptional divisor. Note that there exists $\PP^1$-fibration $g\colon Y_2\rightarrow\PP^1$ such that $g$ has only two singular fibers $C_1, C_2$ and $$C_1=\bar{D}_1+\bar{E}_1+\bar{E}_2+\bar{D}_5+2\bar{E}_3,$$ $$C_2=\bar{N}+\bar{M}+\bar{D}_2+\bar{D}_4+2\bar{D}_3+2\bar{E}_4,$$ where $\bar{D}_1,\ldots,\bar{D}_5$ are the proper transforms of $D_1,\ldots, D_5$, $\bar{E}_1,\ldots, \bar{E}_4$ are the proper transforms of $E_1,\ldots,E_4$, $\bar{N}$ is the proper transform of $N$. On the other hand, $$\tilde{X}\backslash\Supp(L)\cong\FF_1\backslash(K+F_1+F_2),$$ where $K$ is a unique $(-1)$-curve and $F_1, F_2$ are fibers. Then $X$ has an $H$-polar cylinder.

Assume that $H=-K_X+\lambda_1 \hat{E}_1 +\lambda_2 \hat{E}_2+\lambda \varphi(E)$, where $E$ ia a $(-1)$-curve on $\tilde{X}$ that meets only $D_1$ and does not meet $D_2,D_3,D_4,D_5$, $\hat{E}_1, \hat{E}_2$ are $(-1)$-curves and $0\leq\lambda_2\leq\lambda_1<1$, $0<\lambda<5$. Moreover, $\lambda_1>0$. We have the configuration as in Figure 14. We have \begin{gather*}\varphi^*(H)\equiv L=aD_1+2aD_2+(2a-1-\frac{3}{5}\lambda)E_3+\\
+(2a-1-\lambda_1-\lambda_2-\frac{3}{5}\lambda)D_5+2(2a-1-\lambda_1-\lambda_2-\frac{3}{5}\lambda)E_4+bD_4+2bD_3+\\
+(2b-1-\lambda_1-\frac{2}{5}\lambda)E_1+(2b-1-\lambda_2-\frac{2}{5}\lambda)E_2+(b-\lambda_1-\lambda_2-\frac{6}{5}\lambda)M,\end{gather*} where $a+b=1+\lambda_1+\lambda_2+\lambda$, $M$ is a $(0)$-curve such that $M$ pass through the intersection point of $D_2$ and $D_3$ and $M\cdot D_2=M\cdot D_3=1$. Assume that $\lambda_1+\lambda_2+\lambda<1$. Then there exist $a$ and $b$ such that $L$ is an effective divisor. As above, $X$ has an $H$-polar cylinder. So, we may assume that $\lambda_1+\lambda_2+\lambda\geq 1$. We have the configuration as in Figure 15. We may assume that $E=E_5$. We have \begin{gather*}\varphi^*(H)\equiv L=(a+\frac{4}{5}\lambda)D_1+(a-1+\lambda)E_5+(a+1+\frac{3}{5}\lambda)D_2+aE_3+\\
+(a-\lambda_1-\lambda_2)D_5+2(a-\lambda_1-\lambda_2)E_4+(b+\frac{1}{5}\lambda)D_4+(2b+\frac{2}{5}\lambda)D_3+\\
+(2b-1-\lambda_1)E_1+(2b-1-\lambda_2)E_2,\end{gather*} where $a+b=1+\lambda_1+\lambda_2$. Since $\lambda_1+\lambda_2+\lambda\geq 1$ and $\lambda_1<1$, we see that there exist $a$ and $b$ such that $L$ is an effective divisor. As above, $X$ has an $H$-polar cylinder.

Assume that $H=-K_X+\lambda_1 \hat{E}_1+\lambda_2 \varphi(\hat{E}_2) +\lambda \varphi(\hat{E}_3)$, where $\hat{E}_3$ ia a $(-1)$-curve on $\tilde{X}$ that meets only $D_1$ and does not meet $D_2,D_3,D_4,D_5$, $\hat{E}_2$ ia a $(-1)$-curve on $\tilde{X}$ that meets only $D_5$ and does not meet $D_1,D_2,D_3,D_4$, $\hat{E}_1$ ia a $(-1)$-curve on $X$. Put $\tilde{E}_1$ is the proper transform of $\hat{E}_1$. Note that $0\leq\lambda_1<1$, $0\leq\lambda_2<2$, $0<\lambda_3<5$. We may assume that $\lambda_2>0$. We have the configuration as in Figure 14. We have \begin{gather*}\varphi^*(H)\equiv L=aD_1+2aD_2+(2a-1-\lambda_2-\frac{3}{5}\lambda)E_3+\\
+(2a-1-\lambda_1-\frac{1}{2}\lambda_2-\frac{3}{5}\lambda)D_5+2(2a-1-\lambda_1-\frac{1}{2}\lambda_2-\frac{3}{5}\lambda)E_4+\\
+bD_4+2bD_3+(2b-1-\lambda_1-\lambda_2-\frac{2}{5}\lambda)E_1+(2b-1-\frac{2}{5}\lambda)E_2+\\
+(b-\lambda_1-\lambda_2-\frac{6}{5}\lambda)M,\end{gather*} where $a+b=1+\lambda_1+\lambda_2+\lambda$, $M$ is a $(0)$-curve such that $M$ pass through the intersection point of $D_2$ and $D_3$ and $M\cdot D_2=M\cdot D_3=1$. Assume that $\lambda_2<2\lambda_1$ and $\lambda_1+\frac{1}{2}\lambda_2+\lambda<1$. Then there exist $a$ and $b$ such that $L$ is an effective divisor. As above, $X$ has an $H$-polar cylinder. Assume that $\lambda_2\geq2\lambda_1$ and $\lambda_2+\lambda<1$. Then there exist $a$ and $b$ such that $L$ is an effective divisor. As above, $X$ has an $H$-polar cylinder. We have the configuration as in Figure 15. We may assume that $\hat{E}_3=E_5$. We have \begin{gather*}\varphi^*(H)\equiv L=(a+\frac{4}{5}\lambda)D_1+(a-1+\lambda)E_5+(a+1+\frac{3}{5}\lambda)D_2+\\
+(a-\lambda_2)E_3+(a-\lambda_1-\frac{1}{2}\lambda_2)D_5+2(a-\lambda_1-\frac{1}{2}\lambda_2)E_4+(b+\frac{1}{5}\lambda)D_4+\\
+(2b+\frac{2}{5}\lambda)D_3+(2b-1-\lambda_1-\lambda_2)E_1+(2b-1)E_2,\end{gather*} where $a+b=1+\lambda_1+\lambda_2$. Assume that $\lambda_2<2\lambda_1$. We may assume $\lambda_1+\frac{1}{2}\lambda_2+\lambda\geq 1$. Then $L$ is an effective divisor if and only if $a>\lambda_1+\frac{1}{2}\lambda_2$ and $b>\frac{1}{2}+\frac{1}{2}\lambda_1+\frac{1}{2}\lambda_2$. Since $\lambda_1<1$, we see that there exist $a$ and $b$ such that $L$ is an effective divisor. As above, $X$ has an $H$-polar cylinder. Assume that $\lambda_2\geq 2\lambda_1$. We may assume $\lambda_2+\lambda\geq 1$. Then $L$ is an effective divisor if and only if $a>\lambda_2$ and $b>\frac{1}{2}+\frac{1}{2}\lambda_1+\frac{1}{2}\lambda_2$. Therefore, if $\lambda_2<1+\lambda_1$, then there exist $a$ and $b$ such that $L$ is an effective divisor. As above, $X$ has an $H$-polar cylinder. So, we may assume that $\lambda_2\geq 1+\lambda_1$. Assume that $\lambda_1>0$. We have the configuration as in Figure 16. We may assume that $\hat{E}_2=E$, $\hat{E}_3=E_2$. Note that $$\lambda_2 E_2+\frac{1}{2}\lambda_2 D_5\equiv -\lambda_2 E_3-\frac{1}{2}\lambda_2 D_5+\lambda_2 D_2+\lambda_2 D_4+2\lambda_2 D_3+2\lambda_2 E_4,$$ $$\lambda_1\tilde{E}_1\equiv-\lambda_1 E_1+\lambda_1 D_2+\lambda_1 D_4+2\lambda_1 D_3+2\lambda_1 E_4.$$ We obtain \begin{gather*}\varphi^*(H)\equiv L=(a+\frac{4}{5}\lambda)D_1+(a-1-\lambda_1)E_1+(a-1+\lambda)E_2+\\
+(a-1-\frac{1}{2}\lambda_2)D_5+(2a-2-\lambda_2)E_3+(b+\lambda_1+\lambda_2+\frac{3}{5}\lambda)D_2+\\
+(b-1+\lambda_1+\lambda_2+\frac{1}{5}\lambda)D_4+(2b-2+2\lambda_1+2\lambda_2+\frac{2}{5}\lambda)D_3+\\
+(2b-3+2\lambda_1+2\lambda_2)E_4+bM,\end{gather*} where $a+b=2$, $M$ is a $(0)$-curve such that $M$ pass through the intersection point of $D_1$ and $D_2$ and $M\cdot D_1=M\cdot D_2=1$. Since $\lambda_1>0$ and $\lambda_2>1$, we see that there exist $a$ and $b$ such that $L$ is an effective divisor. As above, $X$ has an $H$-polar cylinder. So, we may assume that $\lambda_1=0$. We have the following configuration on $\tilde{X}$.\\
\begin{tikzpicture}
\draw  (0.5,0) -- (0.5,3) (0,0.5) -- (8,0.5) (7.5,0)--(7.5,4) (7.2,3.5)--(11,3.5) (3,1.5)--(5,1.5); \draw [dashed] (7,1.5)--(9,1.5) (7,2)--(9,2) (7,2.5)--(9,2.5) (4,0)--(4,2) (10.5,4)--(10.5,2);
\draw (1,0.5) node [above] {$D_3$} (0.5,3) node [above] {$D_4$} (7.5,4) node [above] {$D_2$}  (11,3.5) node [right] {$D_1$} (5,1.5) node [right] {$D_5$}  (9,1.5) node [right] {$E_1$} (4,2) node [above] {$E_4$} (9,2) node [right] {$E_2$} (9,2.5) node [right] {$E_3$} (10.5,2) node [right] {$E_5$};
\end{tikzpicture}$$\text{Figure 17}$$
where dotted lines are $(-1)$-curves, $D_1,\ldots,D_5$ are $(-2)$-curves. We may assume that $\hat{E}_3=E_5$. We have \begin{gather*}\varphi^*(H)\equiv L=(a+\frac{4}{5}\lambda)D_1+(a-1+\lambda)E_5+(a+1+\frac{3}{5}\lambda)D_2+\\
+(a-\lambda_2)E_1+(a-\lambda_2)E_2+aE_3+\\
+(b+\frac{1}{5}\lambda)D_4+(2b+\frac{2}{5}\lambda)D_3+(2b-1-\frac{1}{2}\lambda_2)D_5+(4b-2-\lambda_2)E_4,\end{gather*} where $a+b=1+\lambda_2$. Since $1\leq \lambda_2<2$, we see that there exist $a$ and $b$ such that $L$ is an effective divisor. As above, $X$ has an $H$-polar cylinder.

Assume that $H=-K_X+\lambda_1 \hat{E}_1 +\lambda_2 \hat{E}_2+\lambda \varphi(E)$, where $E$ ia a $(-1)$-curve on $\tilde{X}$ that meets only $D_5$ and does not meet $D_1,D_2,D_3,D_4$, $\hat{E}_1, \hat{E}_2$ are $(-1)$-curves and $0<\lambda_2\leq\lambda_1<1$, $0<\lambda<2$. Put $\tilde{E}_1, \tilde{E}_2$ are the proper transforms of $\hat{E}_1, \hat{E}_2$. We have the configuration as in Figure 10. Moreover, $E_1\cdot\tilde{E}_1=E_1\cdot\tilde{E}_2=E_1\cdot E=E_4\cdot E=0$, $E_2\cdot\tilde{E}_1=E_4\cdot\tilde{E}_1=1$, $E_3\cdot\tilde{E}_1=E_4\cdot\tilde{E}_1=1$, $E_2\cdot\tilde{E}_2=E_3\cdot\tilde{E}_1=0$, $E\cdot E_2=E\cdot E_3=1$. We have \begin{gather*}\varphi^*(H)\equiv L=aD_1+2aD_2+(2a-1-\lambda_1-\lambda)E_2+\\
+(2a-1-\lambda_2-\lambda)E_3+(2a-1-\lambda_1-\lambda_2-\frac{1}{2}\lambda)D_5+\\
+2(2a-1-\lambda_1-\lambda_2-\frac{1}{2}\lambda)E_4+bD_4+2bD_3+(2b-1)E_1+\\
+(3b-1-2\lambda_1-2\lambda_1-2\lambda)M,\end{gather*} where $a+b=1+\lambda_1+\lambda_2+\lambda$, $M$ is a $(0)$-curve such that $M$ pass through the intersection point of $D_2$ and $D_3$ and $M\cdot D_2=M\cdot D_3=1$. Assume that $\lambda>2\lambda_2$. We see that $L$ is an effective divisor if and only if $a>\frac{1}{2}+\frac{1}{2}\lambda_1+\frac{1}{2}\lambda$ and $b>\frac{1}{3}+\frac{2}{3}\lambda_1+\frac{2}{3}\lambda_2+\frac{2}{3}\lambda$. So, $L$ is an effective divisor if $\lambda_1+\lambda<1+2\lambda_2$. Assume that $\lambda\leq 2\lambda_2$. We see that $L$ is an effective divisor if and only if $a>\frac{1}{2}+\frac{1}{2}\lambda_1+\frac{1}{2}\lambda_2+\frac{1}{4}\lambda$ and $b>\frac{1}{3}+\frac{2}{3}\lambda_1+\frac{2}{3}\lambda_2+\frac{2}{3}\lambda$. So, $L$ is an effective divisor if $\lambda_1+\lambda_2<1+\frac{1}{2}\lambda$. Hence, if $\lambda_1+\lambda_2+\lambda\leq 1$, then there exist $a$ and $b$ such that $L$ is an effective divisor. As above, $X$ has an $H$-polar cylinder. So, we may assume that $\lambda_1+\lambda_2+\lambda>1$. Consider the configuration as in Figure 16. We have $$\lambda E+\frac{1}{2}\lambda D_5\equiv -\lambda E_3-\frac{1}{2}\lambda D_5+\lambda D_2+\lambda D_4+2\lambda D_3+2\lambda E_4,$$ $$\lambda_1\tilde{E}_1\equiv-\lambda_1 E_1+\lambda_1 D_2+\lambda_1 D_4+2\lambda_1 D_3+2\lambda_1 E_4,$$ $$\lambda_2\tilde{E}_2\equiv-\lambda_2 E_2+\lambda_2 D_2+\lambda_2 D_4+2\lambda_2 D_3+2\lambda_2 E_4.$$ We obtain \begin{gather*}\varphi^*(H)\equiv L=aD_1+(a-1-\lambda_1)E_1+(a-1-\lambda_2)E_2+(a-1-\frac{1}{2}\lambda)D_5+\\
+(2a-2-\lambda)E_3+(b+\lambda+\lambda_1+\lambda_2)D_2+(b-1+\lambda+\lambda_1+\lambda_2)D_4+\\
+(2b-2+2\lambda+2\lambda_1+2\lambda_2)D_3+(2b-3+2\lambda+2\lambda_1+2\lambda_2)E_4+bM,\end{gather*} where $a+b=2$, $M$ is a $(0)$-curve such that $M$ pass through the intersection point of $D_1$ and $D_2$ and $M\cdot D_1=M\cdot D_2=1$. Assume that $\lambda>2\lambda_1\geq2\lambda_2$. We see that $L$ is an effective divisor if and only if $a>1+\frac{1}{2}\lambda$ and $b>\frac{3}{2}-\lambda_1-\lambda_2-\lambda$. Since $\lambda_1+\lambda_2+\lambda>1$, we see that there exist $a$ and $b$ such that $L$ is an effective divisor. As above, $X$ has an $H$-polar cylinder. Assume that $\lambda\leq 2\lambda_1$. We see that $L$ is an effective divisor if and only if $a>1+\lambda_1$ and $b>\frac{3}{2}-\lambda_1-\lambda_2-\lambda$. Assume that $\lambda_2+\lambda>\frac{1}{2}$. Then there exist $a$ and $b$ such that $L$ is an effective divisor. As above, $X$ has an $H$-polar cylinder. So, we may assume that $\lambda_2+\lambda\leq\frac{1}{2}$. Hence, $\lambda_1>\frac{1}{2}$. Also, consider the configuration as in Figure 16. We obtain \begin{gather*}\varphi^*(H)\equiv L=aD_1+(a-1-\lambda_2)E_2+(a-1-\frac{1}{2}\lambda)D_5+(2a-2-\lambda)E_3+\\
+(b+\lambda+\lambda_2)D_2+(b-1+\lambda+\lambda_2)D_4+(2b-2+2\lambda+2\lambda_2)D_3+\\
+(2b-3+2\lambda+2\lambda_2)E_4+(b-1)E_5+cM+(c-1+\lambda_1)\tilde{E}_1,\end{gather*} where $a+b+c=3$, $M$ is a $(0)$-curve such that $M$ pass through the intersection point of $D_1$ and $D_2$ and $M\cdot D_1=M\cdot D_2=M\cdot \tilde{E}_1=1$. Since $\lambda_1>\frac{1}{2}\geq \lambda+\lambda_2$, we see that there exist $a$ and $b$ such that $L$ is an effective divisor. Let $Y\rightarrow\tilde{X}$ be the blow up of intersection point of $D_1$ and $D_2$, and let $F$ be the exceptional divisor. Note that there exists $\PP^1$-fibration $g\colon Y_2\rightarrow\PP^1$ such that $g$ has only three singular fibers $C_1, C_2, C_3$ and $$C_1=\bar{D}_1+\bar{E}_2+\bar{D}_5+2\bar{E}_3,$$ $$C_2=\bar{D}_2+\bar{D}_4+2\bar{D}_3+2\bar{E}_4+\bar{E}_5,$$ $$C_3=\bar{M}_2+\tilde{E'}_1,$$ where $\bar{D}_1,\ldots,\bar{D}_5$ are the proper transforms of $D_1,\ldots, D_5$, $\bar{E}_1,\ldots, \bar{E}_5$ are the proper transforms of $E_1,\ldots,E_5$, $\bar{M}$ is the proper transform of $M$, $\tilde{E'}_1$ is the proper transform of $\tilde{E}_1$. On the other hand, $$\tilde{X}\backslash\Supp(L)\cong\FF_1\backslash(K+F_1+F_2+F_3),$$ where $K$ is a unique $(-1)$-curve and $F_1, F_2, F_3$ are fibers. Then $X$ has an $H$-polar cylinder.

Assume that $H=-K_X+\mu C+\lambda \varphi(E)+\lambda_1 \varphi(\hat{E})$, where $C$ is a $(0)$-curve and $\mu>0$, $0\leq\lambda<5$, $0\leq \lambda_1<2$, $E$ is a $(-1)$-curve on $\tilde{X}$ that meets only $D_1$ and does not meet $D_2,D_3,D_4,D_5$, $\hat{E}$ is a $(-1)$-curve on $\tilde{X}$ that meets only $D_5$ and does not meet $D_1,D_2,D_3,D_4$. Note that $$\tilde{C}\sim E+D_1+D_2+D_3+D_4+E',$$ where $E'$ is a $(-1)$-curve that meets only $D_4$ and does not meet $D_1,D_2,D_3,D_5$. We have the following configuration on $\tilde{X}$.\\
\begin{tikzpicture}
\draw  (0.5,0) -- (0.5,3) (0,0.5) -- (8,0.5) (7.5,0)--(7.5,4) (7.2,3.5)--(11,3.5) (3,1.5)--(5,1.5); \draw [dashed] (7,1.5)--(9,1.5) (7,2)--(9,2) (0,2.5)--(2,2.5) (4,0)--(4,2) (10.5,4)--(10.5,2);
\draw (1,0.5) node [above] {$D_3$} (0.5,3) node [above] {$D_4$} (7.5,4) node [above] {$D_2$}  (11,3.5) node [right] {$D_1$} (5,1.5) node [right] {$D_5$}  (9,1.5) node [right] {$E_1$} (4,2) node [above] {$E_3$} (9,2) node [right] {$E_2$} (2,2.5) node [right] {$E'$} (10.5,2) node [right] {$E$};
\end{tikzpicture}$$\text{Figure 18}$$
where dotted lines are $(-1)$-curves, $D_1,\ldots,D_5$ are $(-2)$-curves. Moreover, $\hat{E}\cdot E_3=\hat{E}\cdot E_1=1$, $\hat{E}\cdot E_2=0$. We have \begin{gather*}\varphi^*(H)\equiv L=(a+\frac{4}{5}\lambda)D_1+(a-1+\lambda)E+(a+1+\frac{3}{5}\lambda)D_2+\\
+(a-\lambda_1-\mu)E_1+(a-\mu)E_2+(b+\frac{1}{5}\lambda)D_4+(b-1)E'+\\
+(b+1+\frac{2}{5}\lambda)D_3+(b-\mu-\frac{1}{2}\lambda_1)D_5+2(b-\mu-\frac{1}{2}\lambda_1)E_3,\end{gather*} where $a+b=1+\lambda_1+2\mu$. Assume that $\lambda_1+\lambda+2\mu>1$. Then there exist $a$ and $b$ such that $L$ is an effective divisor. Let $Y\rightarrow\tilde{X}$ be the blow up of intersection point of $D_2$ and $D_3$, and let $F$ be the exceptional divisor. Note that there exists $\PP^1$-fibration $g\colon Y_2\rightarrow\PP^1$ such that $g$ has only two singular fibers $C_1, C_2$ and $$C_1=\bar{D}_1+\bar{E}+\bar{E}_1+\bar{E}_2+\bar{D}_2,$$ $$C_2=\bar{D}_3+\bar{D}_4+\bar{D}_5+2\bar{E}_3+\bar{E'},$$ where $\bar{D}_1,\ldots,\bar{D}_5$ are the proper transforms of $D_1,\ldots, D_5$, $\bar{E}_1,\bar{E}_2, \bar{E}_3, \bar{E}, \bar{E'}$ are the proper transforms of $E_1,E_2, E_3, E, E'$. On the other hand, $$\tilde{X}\backslash\Supp(L)\cong\FF_1\backslash(K+F_1+F_2),$$ where $K$ is a unique $(-1)$-curve and $F_1, F_2$ are fibers. Then $X$ has an $H$-polar cylinder. So, we may assume that $\lambda_1+\lambda+2\mu\leq 1$. Consider the configuration as in Figure 14. We have \begin{gather*}\varphi^*(H)\equiv L=aD_1+2aD_2+(2a-1-\lambda_1-2\mu-\frac{3}{5}\lambda)E_3+\\
+(2a-1-\frac{1}{2}\lambda_1-\mu-\frac{3}{5}\lambda)D_5+2(2a-1-\frac{1}{2}\lambda_1-\mu-\frac{3}{5}\lambda)E_4+\\
+bD_4+2bD_3+(2b-1-\lambda_1-\mu-\frac{2}{5}\lambda)E_1+\\
+(2b-1-\mu-\frac{2}{5}\lambda)E_2+(b-2\mu-\lambda_1-\frac{6}{5}\lambda)M,\end{gather*} where $a+b=1+\lambda+\lambda_1+2\mu$, $M$ is a $(0)$-curve such that $M$ pass through the intersection point of $D_2$ and $D_3$ and $M\cdot D_2=M\cdot D_3=1$. Assume that $\lambda_1+\lambda+2\mu<1$. We see that there exist $a$ and $b$ such that $L$ is an effective divisor. As above, $X$ has an $H$-polar cylinder. So, we may assume that $\lambda_1+\lambda+2\mu=1$. Consider the configuration as in Figure 17. We may assume that $E=E_5$. We have \begin{gather*}\varphi^*(H)\equiv L=(a+\frac{4}{5}\lambda)D_1+(a-1+\lambda)E_5+(a+1+\frac{3}{5}\lambda)D_2+\\
+(a-2\mu-\lambda_1)E_1+(a-\mu-\lambda_1)E_2+(a-\mu)E_3+(b+\frac{1}{5}\lambda)D_4+\\
+(2b+\frac{2}{5}\lambda)D_3+(2b-1-\frac{1}{2}\lambda_1-\mu)D_5+2(2b-1-\frac{1}{2}\lambda_1-\mu)E_4,\end{gather*} where $a+b=1+\lambda_1+2\mu$. Since $\lambda_1+\lambda+2\mu=1$, we see that there exist $a$ and $b$ such that $L$ is an effective divisor. As above, $X$ has an $H$-polar cylinder.
\end{proof}

\begin{lemma}
\label{LemA4A2}
Let $X$ be a del Pezzo surface with du Val singularities and let $H$ be an ample divisor on $X$. Assume that $X$ has the following collection of singularities $A_4+A_2$. Then $X$ has an $H$-polar cylinder.
\end{lemma}

\begin{proof}
Let $\varphi\colon\tilde{X}\rightarrow X$ be the minimal resolution of singularities of $X$ and let $D=\sum_{i=1}^6 D_i$ be the exceptional divisor of $\varphi$. We may assume that $D_1,D_2,D_3,D_4$ correspond to the singular point of type $A_4$, $D_5, D_6$ corresponds to the singular point of type $A_2$. Assume that $H=-K_X+\lambda_1 E+\lambda\varphi(E')$, where $E$ ia a $(-1)$-curve and $\lambda_1<1$, $E'$ is a $(-1)$-curve on $\tilde{X}$ that meets only $D_6$ and does not meet $D_1,D_2,D_3,D_4, D_5$, $\lambda<3$. Note that at least one of $\lambda, \lambda_1$ is not equal to zero. We have the following configuration on $\tilde{X}$.\\
\begin{tikzpicture}
\draw  (0.5,0) -- (0.5,3) (0,0.5) -- (8,0.5) (7.5,0)--(7.5,4) (7.2,3.5)--(10,3.5) (3.5,1.5)--(6,1.5) (4,1)--(4,3.5); \draw [dashed] (3,0)--(3,3) (7,1.5)--(10,1.5) (5.5,0)--(5.5,2) ;
\draw (1,0.5) node [above] {$D_2$} (0.5,3) node [above] {$D_1$} (7.5,4) node [above] {$D_3$}  (10,3.5) node [right] {$D_4$} (6,1.5) node [right] {$D_6$}  (10,1.5) node [right] {$E_1$} (5.5,2) node [above] {$E_3$} (3,3) node [above] {$E_2$} (4,3.5) node [above] {$D_5$};
\end{tikzpicture}$$\text{Figure 19}$$
where dotted lines are $(-1)$-curves, $D_1,\ldots,D_6$ are $(-2)$-curve. Moreover, $E_1\cdot\tilde{E}=E_1\cdot E'=E_2\cdot \tilde{E}=E_3\cdot E'=0$, $E_3\cdot\tilde{E}=E_2\cdot E'=1$. We have \begin{gather*}\varphi^*(H)\equiv L=aD_1+2aD_2+(2a-1-\lambda)E_2+(2a-1-\lambda_1-\frac{2}{3}\lambda)D_5+\\
+2(2a-1-\lambda_1-\frac{2}{3}\lambda)D_6+3(2a-1-\lambda_1-\frac{2}{3}\lambda)E_3+bD_4+2bD_3+\\
+(2b-1)E_1+(3b-1-2\lambda-2\lambda_1)M,\end{gather*} where $a+b=1+\lambda+\lambda_1$, $M$ is a $(0)$-curve such that $M$ pass through the intersection point of $D_2$ and $D_3$ and $M\cdot D_2=M\cdot D_3=1$. Assume that $L$ is an effective divisor. Let $Y_1\rightarrow\tilde{X}$ be the blow up of the intersection point of $D_2$ and $D_3$, and $N_1$ be the exceptional divisor. Let $Y_2\rightarrow Y_1$ be the blow up of the intersection point of $N_1$ and the proper transform of $D_2$, and $N_2$ be the exceptional divisor. Let $Y_3\rightarrow Y_2$ be the blow up of the intersection point of $N_2$ and the proper transform of $D_2$, and $N_3$ be the exceptional divisor. Let $Y_4\rightarrow Y_3$ be the blow up of the intersection point of $N_3$ and the proper transform of $N_2$, and $F$ be the exceptional divisor.
We obtain a surface $Y_4$ with the following configuration\\
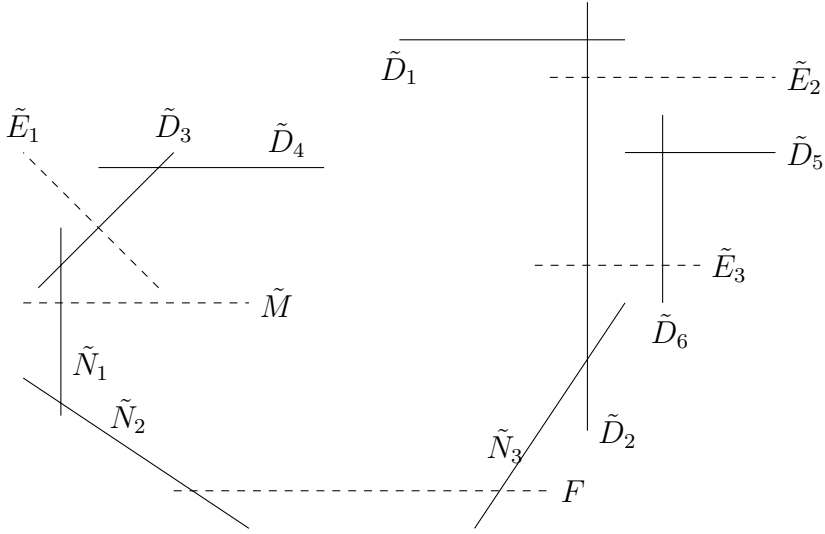
\begin{tikzpicture}
\draw  (3,0) -- (0,2) (0.5,1.5) -- (0.5,4) (0.2,3.2)--(2,5) (1,4.8)--(4,4.8) (6,0)--(8,3) (7.5,1.3)--(7.5,7) (8,6.5)--(5,6.5) (8.5,3)--(8.5,5.5) (8,5)--(10,5); \draw [dashed] (2,0.5) -- (7,0.5) (0,3)--(3,3) (1.8,3.2)--(0,5) (6.8,3.5)--(9,3.5) (7,6)--(10,6);
\draw (7.5,1.3) node [right] {$\tilde{D}_2$} (2,5) node [above] {$\tilde{D}_3$} (3.5,4.8) node [above] {$\tilde{D}_4$}  (5,6.5) node [below] {$\tilde{D}_1$} (0,5) node [above] {$\tilde{E}_1$} (9,3.5) node [right] {$\tilde{E}_3$} (3,3) node [right] {$\tilde{M}$} (7,0.5) node [right] {$F$} (0.5,2.2) node [right] {$\tilde{N}_1$} (1,1.5) node [right] {$\tilde{N}_2$} (6.8,1.1) node [left] {$\tilde{N}_3$} (8.6,3) node [below] {$\tilde{D}_6$}  (10,6) node [right] {$\tilde{E}_2$} (10,5) node [right] {$\tilde{D}_5$};
\end{tikzpicture}\\ where $\tilde{D}_1,\ldots,\tilde{D}_6$ are the proper transforms of $D_1,\ldots, D_6$, $\tilde{E}_1, \tilde{E}_2, \tilde{E}_3$ are the proper transforms of $E_1,E_2,E_3$, $\tilde{N}_1, \tilde{N}_2, \tilde{N}_3$ are the proper transforms of $N_1, N_2, N_3$, $\tilde{M}$ is the proper transform of $M$. Note that there exists $\PP^1$-fibration $g\colon Y_4\rightarrow\PP^1$ such that $g$ has only two singular fibers $C_1, C_2$ and $$C_1=\tilde{D}_1+\tilde{N}_3+2\tilde{D}_2+2\tilde{E}_2+2\tilde{D}_5+4\tilde{D}_6+6\tilde{E}_3,$$ $$C_2=\tilde{D}_4+\tilde{N}_2+2\tilde{D}_3+2\tilde{E}_1+3\tilde{N}_1+3\tilde{M}.$$
Note that $$Y_4\backslash\Supp(F+C_1+C_2)\cong X\backslash\Supp(\varphi(L))\cong\FF_1\backslash(K+F_1+F_2),$$ where $K$ is a unique $(-1)$-curve and $F_1, F_2$ are fibers. Then $X$ has an $H$-polar cylinder. Assume that $\lambda\leq3\lambda_1$. Then $L$ is an effective divisor if and only if $a>\frac{1}{2}+\frac{1}{2}\lambda_1+\frac{1}{3}\lambda$ and $b>\frac{1}{3}+\frac{2}{3}\lambda+\frac{2}{3}\lambda_1$. Since $\lambda_1<1$, we see that there exist $a$ and $b$ such that $L$ is an effective divisor. So, we may assume that $\lambda>3\lambda_1$. Then $L$ is an effective divisor if and only if $a>\frac{1}{2}+\frac{1}{2}\lambda$ and $b>\frac{1}{3}+\frac{2}{3}\lambda+\frac{2}{3}\lambda_1$. Hence, $L$ is an effective divisor if $\lambda<1+2\lambda_1$. So, we may assume that $\lambda\geq 1+2\lambda_1$.
We have the following configuration on $\tilde{X}$.
\begin{tikzpicture}
\draw  (0,0.5) -- (10,0.5) (8.5,0) -- (10.5,2) (10,1)--(10,5) (10.5,4)--(8.5,6) (0.5,2)--(0.5,5) (0,4.5)--(3,4.5); \draw [dashed] (2,0) -- (0,3) (2,5)--(4,3) (8,0) -- (6,3) (6.5,2)--(6.5,5) (9,3)--(12,3);
\draw (1.5,1.5) node [above] {$E_2$} (0.5,5) node [above] {$D_5$} (3,4.5) node [right] {$D_6$} (7.5,1.5) node [above] {$E_1$} (6.5,5) node [above] {$\tilde{E}$} (12,3) node [right] {$E_3$} (0,0.5) node [above] {$D_1$} (10.5,2) node [right] {$D_2$} (10,5) node [above] {$D_3$} (8.5,6) node [above] {$D_4$} (4,3) node [right] {$E'$};
\end{tikzpicture}$$\text{Figure 20}.$$ where dotted lines are $(-1)$-curves, $D_1,\ldots,D_6$ are $(-2)$-curves.  We have $$\lambda E'+\frac{2}{3}\lambda D_6+\frac{1}{3}\lambda D_5\equiv -\lambda E_2-\frac{1}{3}\lambda D_6-\frac{2}{3}\lambda D_5+\lambda D_2+\lambda D_4+2\lambda D_3+2\lambda E_3,$$ $$\lambda_1\tilde{E}\equiv-\lambda_1 E_1+\lambda_1 D_2+\lambda_1 D_4+2\lambda_1 D_3+2\lambda_1 E_3.$$ We obtain \begin{gather*}\varphi^*(H)\equiv L=aD_1+(a-1-\lambda_1)E_1+(a-1-\frac{1}{3}\lambda)D_6+\\
+2(a-1-\frac{1}{3}\lambda)D_5+3(a-1-\frac{1}{3}\lambda)E_2+\\
+(b+\lambda+\lambda_1)D_2+(b-1+\lambda+\lambda_1)D_4+\\
+(2b-2+2\lambda+2\lambda_1)D_3+(2b-3+2\lambda+2\lambda_1)E_3+bM,\end{gather*} where $a+b=2$, $M$ is a $(0)$-curve such that $M$ pass through the intersection point of $D_1$ and $D_2$ and $M\cdot D_1=M\cdot D_2=1$. Since $\lambda\geq 1+2\lambda_1$, we see that there exist $a$ and $b$ such that $L$ is an effective divisor. Let $Y_1\rightarrow\tilde{X}$ be the blow up of the intersection point of $D_1$ and $D_2$, and $N$ be the exceptional divisor. Let $Y_2\rightarrow Y_1$ be the blow up of the intersection point of $N$ and the proper transform of $D_1$, and $F$ be the exceptional divisor. Note that there exists $\PP^1$-fibration $g\colon Y_2\rightarrow\PP^1$ such that $g$ has only two singular fibers $C_1, C_2$ and $$C_1=\tilde{D}_1+\tilde{E}_1+\tilde{D}_6+2\tilde{D}_5+3\tilde{E}_2,$$ $$C_2=\tilde{D}_4+\tilde{N}+\tilde{M}+\tilde{D}_2+2\tilde{D}_3+2\tilde{E}_3,$$ where $\tilde{D}_1,\ldots,\tilde{D}_6$ are the proper transforms of $D_1,\ldots, D_6$, $\tilde{E}_1, \tilde{E}_2, \tilde{E}_3$ are the proper transforms of $E_1,E_2,E_3$, $\tilde{N}$ is the proper transform of $N$, $\tilde{M}$ is the proper transform of $M$. Then $X$ has an $H$-polar cylinder.

Assume that $H=-K_X+\lambda \varphi(E)+\lambda_1\varphi(E')$, where $E$ ia a $(-1)$-curve on $\tilde{X}$ that meets only $D_1$ and does not meet $D_2,D_3,D_4,D_5, D_6$ and $\lambda<5$, $E'$ is a $(-1)$-curve on $\tilde{X}$ that meets only $D_6$ and does not meet $D_1,D_2,D_3,D_4, D_5$, $\lambda<3$. Note that we may assume that $\lambda>0$. We have the following configuration on $\tilde{X}$.\\ \begin{tikzpicture}
\draw  (0.5,0) -- (0.5,3) (0,0.5) -- (8,0.5) (7.5,0)--(7.5,4) (7.2,3.5)--(10,3.5) (4,1.5)--(6,1.5) (4.5,1)--(4.5,3); \draw [dashed] (0,2)--(2,2)(7,1.5)--(10,1.5) (7,2)--(10,2) (5.5,0)--(5.5,2) ;
\draw (1,0.5) node [above] {$D_2$} (0.5,3) node [above] {$D_1$} (7.5,4) node [above] {$D_3$}  (10,3.5) node [right] {$D_4$} (6,1.5) node [right] {$D_6$}  (10,1.5) node [right] {$E_1$} (5.5,2) node [above] {$E_3$} (10,2) node [right] {$E_2$} (4.5,3) node [above] {$D_5$} (2,2) node [right] {$E$};
\end{tikzpicture}$$\text{Figure 21}$$ where dotted lines are $(-1)$-curves, $D_1,\ldots,D_6$ are $(-2)$-curves. Moreover, $E'\cdot E_1=1$, $E'\cdot E_2=0$. We have \begin{gather*}\varphi^*(H)\equiv L=aD_1+2aD_2+(2a-1-\frac{2}{3}\lambda_1-\frac{3}{5}\lambda)D_5+\\
+2(2a-1-\frac{2}{3}\lambda_1-\frac{3}{5}\lambda)D_6+3(2a-1-\frac{2}{3}\lambda_1-\frac{3}{5}\lambda)E_3+bD_4+2bD_3+\\
+(2b-1-\lambda_1-\frac{2}{5}\lambda)E_1+(2b-1-\frac{2}{5}\lambda)E_2+(b-\lambda_1-\frac{6}{5}\lambda)M,\end{gather*} where $a+b=1+\lambda_1+\lambda$, $M$ is a $(0)$-curve such that $M$ pass through the intersection point of $D_2$ and $D_3$ and $M\cdot D_2=M\cdot D_3=1$. Assume that $L$ is an effective divisor. Let $Y_1\rightarrow\tilde{X}$ be the blow up of the intersection point of $D_2$ and $D_3$, and $N_1$ be the exceptional divisor. Let $Y_2\rightarrow Y_1$ be the blow up of the intersection point of $N_1$ and the proper transform of $D_2$, and $N_2$ be the exceptional divisor. Let $Y_3\rightarrow Y_2$ be the blow up of the intersection point of $N_2$ and the proper transform of $N_1$, and $F$ be the exceptional divisor. Note that there exists $\PP^1$-fibration $g\colon Y_3\rightarrow\PP^1$ such that $g$ has only two singular fibers $C_1, C_2$ and $$C_1=\tilde{D}_1+\tilde{N}_2+2\tilde{D}_2+2\tilde{D}_5+4\tilde{D}_6+6\tilde{E}_3,$$ $$C_2=\tilde{D}_4+\tilde{N}_1+\tilde{M}+2\tilde{D}_3+2\tilde{E}_1+2\tilde{E}_2,$$ where $\tilde{D}_1,\ldots,\tilde{D}_6$ are the proper transforms of $D_1,\ldots, D_6$, $\tilde{E}_1, \tilde{E}_2, \tilde{E}_3$ are the proper transforms of $E_1,E_2,E_3$, $\tilde{N}_1, \tilde{N}_2$ are the proper transforms of $N_1, N_2$, $\tilde{M}$ is the proper transform of $M$. Then $X$ has an $H$-polar cylinder. Note that $L$ is an effective divisor if and only if $a>\frac{1}{2}+\frac{1}{3}\lambda_1+\frac{3}{10}\lambda$ and $b>\lambda_1+\frac{6}{5}\lambda$. Note that if $\lambda+\frac{2}{3}\lambda_1<1$, then there exist $a$ and $b$ such that $L$ is an effective divisor. So, we may assume that $\lambda+\frac{2}{3}\lambda_1\geq 1$. Also, consider the configuration as in Figure 21. We have \begin{gather*}\varphi^*(H)\equiv L=(a+\frac{4}{5}\lambda)D_1+(a-1+\lambda)E+(a+1+\frac{3}{5}\lambda)D_2+\\
+(a-\frac{2}{3}\lambda_1)D_5+2(a-\frac{2}{3}\lambda_1)D_6+3(a-\frac{2}{3}\lambda_1)E_3+(b+\frac{1}{5}\lambda)D_4+\\
+(2b+\frac{2}{5}\lambda)D_3+(2b-1-\lambda_1)E_1+(2b-1)E_2,\end{gather*} where $a+b=1+\lambda_1$. Since $\lambda+\frac{2}{3}\lambda_1\geq 1$ and $\lambda_1<3$, we see that there exist $a$ and $b$ such that $L$ is an effective divisor. Let $Y_1\rightarrow\tilde{X}$ be the blow up of the intersection point of $D_2$ and $D_3$, and $N$ be the exceptional divisor. Let $Y_2\rightarrow Y_1$ be the blow up of the intersection point of $N$ and the proper transform of $D_2$, and $F$ be the exceptional divisor. Note that there exists $\PP^1$-fibration $g\colon Y_2\rightarrow\PP^1$ such that $g$ has only two singular fibers $C_1, C_2$ and $$C_1=\tilde{E}+\tilde{D}_1+\tilde{D}_2+\tilde{D}_5+2\tilde{D}_6+3\tilde{E}_3,$$ $$C_2=\tilde{D}_4+\tilde{N}+2\tilde{D}_3+2\tilde{E}_1+2\tilde{E}_2,$$ where $\tilde{D}_1,\ldots,\tilde{D}_6$ are the proper transforms of $D_1,\ldots, D_6$, $\tilde{E}_1, \tilde{E}_2, \tilde{E}_3$ are the proper transforms of $E_1,E_2,E_3$, $\tilde{N}$ is the proper transform of $N$. Then $X$ has an $H$-polar cylinder.
\end{proof}

\begin{lemma}
\label{LemA4A1A1}
Let $X$ be a del Pezzo surface with du Val singularities and let $H$ be an ample divisor on $X$. Assume that $X$ has the following collection of singularities $A_4+A_1+A_1$. Then $X$ has an $H$-polar cylinder.
\end{lemma}

\begin{proof}
Let $\varphi\colon\tilde{X}\rightarrow X$ be the minimal resolution of singularities of $X$ and let $D=\sum_{i=1}^6 D_i$ be the exceptional divisor of $\varphi$. We may assume that $D_1,D_2,D_3,D_4$ correspond to the singular point of type $A_4$, $D_5, D_6$ corresponds to the singular points of type $A_1$. Assume that $H=-K_X+\lambda \varphi(E)$, where $E$ ia a $(-1)$-curve on $\tilde{X}$ that meets only $D_5$ and $\lambda<2$. We have the following configuration on $\tilde{X}$.\\
\begin{tikzpicture}
\draw  (0.5,0) -- (0.5,3) (0,0.5) -- (8,0.5) (7.5,0)--(7.5,4) (7.2,3.5)--(10,3.5) (5,1.5)--(6,1.5) (9,1)--(9,3); \draw [dashed] (4,0)--(4,3) (7,1.5)--(10,1.5) (5.5,0)--(5.5,2) (0,2.5)--(2,2.5);
\draw (1,0.5) node [above] {$D_2$} (0.5,3) node [above] {$D_1$} (7.5,4) node [above] {$D_3$}  (10,3.5) node [right] {$D_4$} (6,1.5) node [right] {$D_5$}  (10,1.5) node [right] {$E_1$} (5.5,2) node [above] {$E_2$} (9,3) node [right] {$D_6$} (4,3) node [above] {$E_3$} (2,2.5) node [right] {$E$};
\end{tikzpicture}$$\text{Figure 22}$$ where dotted lines are $(-1)$-curves, $D_1,\ldots,D_6$ are $(-2)$-curves. Moreover, $E\cdot E_1=1$, $E\cdot E_2=E\cdot E_3=0$. We have \begin{gather*}\varphi^*(H)\equiv L=aD_1+2aD_2+(2a-1)E_3+(2a-1-\frac{1}{2}\lambda)D_5+\\
+2(2a-1-\frac{1}{2}\lambda)E_2+bD_4+2bD_3+(2b-1-\lambda)D_6+\\
+2(2b-1-\lambda)E_1+(b-\lambda)M,\end{gather*} where $a+b=1+\lambda$, $M$ is a $(0)$-curve such that $M$ pass through the intersection point of $D_2$ and $D_3$ and $M\cdot D_2=M\cdot D_2=1$. Since $\lambda<2$, we see that there exist $a$ and $b$ such that $L$ is an effective divisor. Let $Y_1\rightarrow\tilde{X}$ be the blow up of the intersection point of $D_2$ and $D_3$, and $N_1$ be the exceptional divisor. Let $Y_2\rightarrow Y_1$ be the blow up of the intersection point of $N_1$ and the proper transform of $D_2$, and $N_2$ be the exceptional divisor. Let $Y_3\rightarrow Y_2$ be the blow up of the intersection point of $N_2$ and the proper transform of $N_1$, and $F$ be the exceptional divisor. Note that there exists $\PP^1$-fibration $g\colon Y_3\rightarrow\PP^1$ such that $g$ has only two singular fibers $C_1, C_2$ and $$C_1=\tilde{D}_1+\tilde{N}_2+2\tilde{D}_2+2\tilde{D}_5+2\tilde{E}_3+4\tilde{E}_2,$$ $$C_2=\tilde{D}_4+\tilde{N}_1+\tilde{M}+2\tilde{D}_3+2\tilde{D}_6+4\tilde{E}_1,$$ where $\tilde{D}_1,\ldots,\tilde{D}_6$ are the proper transforms of $D_1,\ldots, D_6$, $\tilde{E}_1, \tilde{E}_2, \tilde{E}_3$ are the proper transforms of $E_1,E_2,E_3$, $\tilde{N}_1, \tilde{N}_2$ are the proper transforms of $N_1, N_2$, $\tilde{M}$ is the proper transform of $M$. Then $X$ has an $H$-polar cylinder.

Assume that $H=-K_X+\lambda_1 \varphi(\hat{E}_1)+\lambda_2\varphi(\hat{E}_2)$, where $\hat{E}_1$ ia a $(-1)$-curve on $\tilde{X}$ that meets only $D_5$ and $\lambda_1<2$, $\hat{E}_2$ is a $(-1)$-curve on $\tilde{X}$ that meets only $D_6$ and does not meet $D_1,D_2,D_3,D_4, D_5$, $\lambda_2<2$. We may assume that $2>\lambda_1\geq \lambda_2>0$. We have the following configuration on $\tilde{X}$.\\
\begin{tikzpicture}
\draw  (0.5,0) -- (0.5,3) (0,0.5) -- (8,0.5) (7.5,0)--(7.5,4) (7.2,3.5)--(10,3.5) (4,1.5)--(6,1.5) (2,2.5)--(4,2.5); \draw [dashed] (3,0)--(3,3) (7,1.5)--(10,1.5) (5.5,0)--(5.5,2) ;
\draw (1,0.5) node [above] {$D_2$} (0.5,3) node [above] {$D_1$} (7.5,4) node [above] {$D_3$}  (10,3.5) node [right] {$D_4$} (6,1.5) node [right] {$D_6$}  (10,1.5) node [right] {$E_1$} (5.5,2) node [above] {$E_3$} (3,3) node [above] {$E_2$} (4,2.5) node [right] {$D_5$} ;
\end{tikzpicture}$$\text{Figure 23}$$
where dotted lines are $(-1)$-curves, $D_1,\ldots,D_6$ are $(-2)$-curve. Moreover, $\hat{E}_1\cdot E_3=\hat{E}_2\cdot E_2=1$, $\hat{E}_1\cdot E_2=\hat{E}_2\cdot E_3=0$, $E_1\cdot\hat{E}_1=E_1\cdot\hat{E}_2=0$. We have \begin{gather*}\varphi^*(H)\equiv L=aD_1+2aD_2+(2a-1-\frac{1}{2}\lambda_1-\lambda_2)D_5+\\
+2(2a-1-\frac{1}{2}\lambda_1-\lambda_2)E_2+(2a-1-\lambda_1-\frac{1}{2}\lambda_2)D_6+\\
+2(2a-1-\lambda_1-\frac{1}{2}\lambda_2)E_3+bD_4+2bD_3+(2b-1)E_1+\\
+(3b-1-2\lambda_1-2\lambda_2)M,\end{gather*} where $a+b=1+\lambda_1+\lambda_2$, $M$ is a $(0)$-curve such that $M$ pass through the intersection point of $D_2$ and $D_3$ and $M\cdot D_2=M\cdot D_2=1$. Assume that $L$ is an effective divisor. Let $Y_1\rightarrow\tilde{X}$ be the blow up of the intersection point of $D_2$ and $D_3$, and $N_1$ be the exceptional divisor. Let $Y_2\rightarrow Y_1$ be the blow up of the intersection point of $N_1$ and the proper transform of $D_2$, and $N_2$ be the exceptional divisor. Let $Y_3\rightarrow Y_2$ be the blow up of the intersection point of $N_2$ and the proper transform of $D_2$, and $N_3$ be the exceptional divisor. Let $Y_4\rightarrow Y_3$ be the blow up of the intersection point of $N_3$ and the proper transform of $N_2$, and $F$ be the exceptional divisor. We obtain a surface $Y_4$ with the following configuration\\
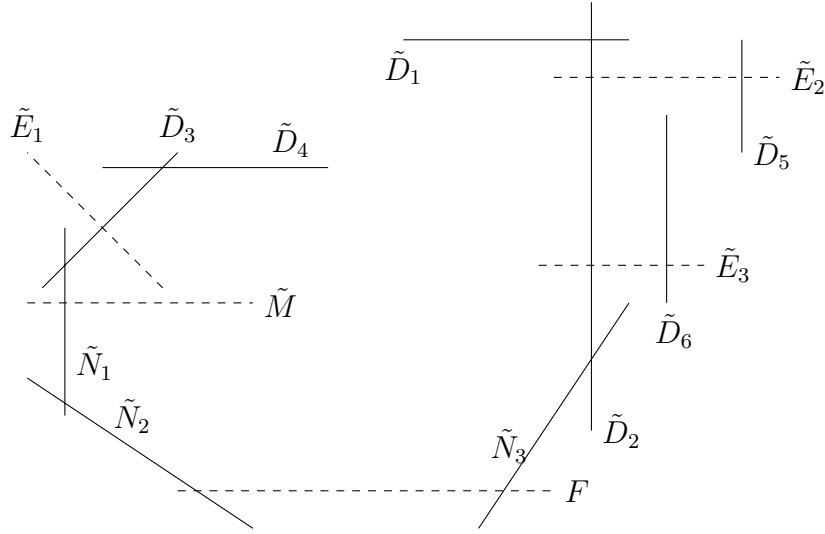
\begin{tikzpicture}
\draw  (3,0) -- (0,2) (0.5,1.5) -- (0.5,4) (0.2,3.2)--(2,5) (1,4.8)--(4,4.8) (6,0)--(8,3) (7.5,1.3)--(7.5,7) (8,6.5)--(5,6.5) (8.5,3)--(8.5,5.5) (9.5,6.5)--(9.5,5); \draw [dashed] (2,0.5) -- (7,0.5) (0,3)--(3,3) (1.8,3.2)--(0,5) (6.8,3.5)--(9,3.5) (7,6)--(10,6);
\draw (7.5,1.3) node [right] {$\tilde{D}_2$} (2,5) node [above] {$\tilde{D}_3$} (3.5,4.8) node [above] {$\tilde{D}_4$}  (5,6.5) node [below] {$\tilde{D}_1$} (0,5) node [above] {$\tilde{E}_1$} (9,3.5) node [right] {$\tilde{E}_3$} (3,3) node [right] {$\tilde{M}$} (7,0.5) node [right] {$F$} (0.5,2.2) node [right] {$\tilde{N}_1$} (1,1.5) node [right] {$\tilde{N}_2$} (6.8,1.1) node [left] {$\tilde{N}_3$} (8.6,3) node [below] {$\tilde{D}_6$}  (10,6) node [right] {$\tilde{E}_2$} (9.5,5) node [right] {$\tilde{D}_5$};
\end{tikzpicture}\\ where $\tilde{D}_1,\ldots,\tilde{D}_6$ are the proper transforms of $D_1,\ldots, D_6$, $\tilde{E}_1, \tilde{E}_2, \tilde{E}_3$ are the proper transforms of $E_1,E_2,E_3$, $\tilde{N}_1, \tilde{N}_2, \tilde{N}_3$ are the proper transforms of $N_1, N_2, N_3$, $\tilde{M}$ is the proper transform of $M$. Note that there exists $\PP^1$-fibration $g\colon Y_4\rightarrow\PP^1$ such that $g$ has only two singular fibers $C_1, C_2$ and $$C_1=\tilde{D}_1+\tilde{N}_3+2\tilde{D}_2+2\tilde{D}_5+2\tilde{D}_6+4\tilde{E}_2+4\tilde{E}_3,$$ $$C_2=\tilde{D}_4+\tilde{N}_2+3\tilde{N}_1+3\tilde{M}+2\tilde{D}_3+2\tilde{E}_1.$$ Then $X$ has an $H$-polar cylinder. On the other hand, $L$ is an effective divisor if and only if $a>\frac{1}{2}+\frac{1}{2}\lambda_1+\frac{1}{4}\lambda_2$ and $b>\frac{1}{3}+\frac{2}{3}\lambda_1+\frac{2}{3}\lambda_2$. We see that if $1+\frac{1}{2}\lambda_2>\lambda_1$, then there exist $a$ and $b$ such that $L$ is an effective divisor. So, we may assume that $1+\frac{1}{2}\lambda_2\leq\lambda_1$. In particular $\lambda_1>1$. We have the following configuration on $\tilde{X}$.
\begin{tikzpicture}
\draw  (0,0.5) -- (10,0.5) (8.5,0) -- (10.5,2) (10,1)--(10,5) (10.5,4)--(8.5,6) (0.5,2)--(0.5,5) (6.5,2)--(6.5,5); \draw [dashed] (0,4.5)--(2,4.5) (2,0) -- (0,3) (8,0) -- (6,3) (9,3)--(12,3) (4.5,4.5)--(7,4.5);
\draw (1.5,1.5) node [above] {$E_1$} (0.5,5) node [above] {$D_5$} (2,4.5) node [right] {$\hat{E}_1$} (7.5,1.5) node [above] {$E_2$} (6.5,5) node [above] {$D_6$} (12,3) node [right] {$E_3$} (0,0.5) node [above] {$D_1$} (10.5,2) node [right] {$D_2$} (10,5) node [above] {$D_3$} (8.5,6) node [above] {$D_4$}  (4.5,4.5) node [left] {$\hat{E}_2$};
\end{tikzpicture}$$\text{Figure 24}.$$ where dotted lines are $(-1)$-curves, $D_1,\ldots,D_6$ are $(-2)$-curves. We have $$\lambda_1 \hat{E}_1+\frac{1}{2}\lambda_1 D_5\equiv -\lambda_1 E_1-\frac{1}{2}\lambda_1 D_5+\lambda_1 D_2+\lambda_1 D_4+2\lambda_1 D_3+2\lambda_1 E_3,$$ $$\lambda_2 \hat{E}_2+\frac{1}{2}\lambda_2 D_6\equiv -\lambda_2 E_2-\frac{1}{2}\lambda_2 D_6+\lambda_2 D_2+\lambda_2 D_4+2\lambda_2 D_3+2\lambda_2 E_3.$$ We obtain \begin{gather*}\varphi^*(H)\equiv L=aD_1+(a-1-\frac{1}{2}\lambda_1)D_5+2(a-1-\frac{1}{2}\lambda_1)E_1+\\
+(a-1-\frac{1}{2}\lambda_2)D_6+2(a-1-\frac{1}{2}\lambda_2)E_2+(b+\lambda_1+\lambda_2)D_2+\\
+(b-1+\lambda_1+\lambda_2)D_4+(2b-2+2\lambda_1+2\lambda_2)D_3+\\
+(2b-3+2\lambda_1+2\lambda_2)E_3+bM,\end{gather*} where $a+b=2$, $M$ is a $(0)$-curve such that $M$ pass through the intersection point of $D_1$ and $D_2$ and $M\cdot D_1=M\cdot D_2=1$. Since $\lambda_1>1$, we see that there exist $a$ and $b$ such that $L$ is an effective divisor. Let $Y_1\rightarrow\tilde{X}$ be the blow up of the intersection point of $D_1$ and $D_2$, and $N$ be the exceptional divisor. Let $Y_2\rightarrow Y_1$ be the blow up of the intersection point of $N$ and the proper transform of $D_1$, and $F$ be the exceptional divisor. Note that there exists $\PP^1$-fibration $g\colon Y_2\rightarrow\PP^1$ such that $g$ has only two singular fibers $C_1, C_2$ and $$C_1=\tilde{D}_1+\tilde{D}_6+\tilde{D}_5+2\tilde{E}_1+2\tilde{E}_2,$$ $$C_2=\tilde{D}_4+\tilde{N}+\tilde{M}+\tilde{D}_2+2\tilde{D}_3+2\tilde{E}_3,$$ where $\tilde{D}_1,\ldots,\tilde{D}_6$ are the proper transforms of $D_1,\ldots, D_6$, $\tilde{E}_1, \tilde{E}_2, \tilde{E}_3$ are the proper transforms of $E_1,E_2,E_3$, $\tilde{N}$ is the proper transform of $N$, $\tilde{M}$ is the proper transform of $M$. Then $X$ has an $H$-polar cylinder.

Assume that $H=-K_X+\lambda_1 \varphi(\hat{E}_1)+\lambda\varphi(E)$, where $\hat{E}_1$ ia a $(-1)$-curve on $\tilde{X}$ that meets only $D_6$ and does not meet $D_1,D_2,D_3,D_4, D_5$, $\lambda_1<2$, $E$ is a $(-1)$-curve on $\tilde{X}$ that meets only $D_1$ and does not meet $D_2,D_3,D_4,D_5, D_6$, $\lambda<5$. We may assume that $\lambda>0$. Consider the configuration as in Figure 22. We have \begin{gather*}\varphi^*(H)\equiv L=aD_1+2aD_2+(2a-1-\frac{3}{5}\lambda)E_3+(2a-1-\lambda_1-\frac{3}{5}\lambda)D_5+\\
+2(2a-1-\lambda_1-\frac{3}{5}\lambda)E_2+bD_4+2bD_3+(2b-1-\frac{1}{2}\lambda_1-\frac{2}{5}\lambda)D_6+\\
+2(2b-1-\frac{1}{2}\lambda_1-\frac{2}{5}\lambda)E_1+(b-\lambda_1-\frac{6}{5}\lambda)M,\end{gather*} where $a+b=1+\lambda_1+\lambda$, $M$ is a $(0)$-curve such that $M$ pass through the intersection point of $D_2$ and $D_3$ and $M\cdot D_2=M\cdot D_2=1$. Assume that $\lambda+\lambda_1<1$. Then there exist $a$ and $b$ such that $L$ is an effective divisor. As above, $X$ has an $H$-polar cylinder. So, we may assume $\lambda+\lambda_1\geq 1$. Also, consider the configuration as in Figure 22. We have \begin{gather*}\varphi^*(H)\equiv L=(a+\frac{4}{5}\lambda)D_1+(a-1+\lambda)E+(a+1+\frac{3}{5}\lambda)D_2+aE_3+\\
+(a-\lambda_1)D_5+2(a-\lambda_1)E_2+(b+\frac{1}{5}\lambda)D_4+(2b+\frac{2}{5}\lambda)D_3+\\
+(2b-1-\frac{1}{2}\lambda_1)D_6+2(2b-1-\frac{1}{2}\lambda_1)E_1,\end{gather*} where $a+b=1+\lambda_1$. Since $\lambda+\lambda_1\geq 1$ and $\lambda_1<2$, we see that there exist $a$ and $b$ such that $L$ is an effective divisor. Let $Y_1\rightarrow\tilde{X}$ be the blow up of the intersection point of $D_2$ and $D_3$, and $N$ be the exceptional divisor. Let $Y_2\rightarrow Y_1$ be the blow up of the intersection point of $N$ and the proper transform of $D_2$, and $F$ be the exceptional divisor. Note that there exists $\PP^1$-fibration $g\colon Y_2\rightarrow\PP^1$ such that $g$ has only two singular fibers $C_1, C_2$ and $$C_1=\tilde{D}_1+\tilde{D}_2+\tilde{D}_5+\tilde{E}+2\tilde{E}_2+\tilde{E}_3,$$ $$C_2=\tilde{D}_4+\tilde{N}+2\tilde{D}_3+2\tilde{D}_6+4\tilde{E}_1,$$ where $\tilde{D}_1,\ldots,\tilde{D}_6$ are the proper transforms of $D_1,\ldots, D_6$, $\tilde{E}_1, \tilde{E}_2, \tilde{E}_3$, $\tilde{E}$ are the proper transforms of $E_1,E_2,E_3, E$, $\tilde{N}$ is the proper transform of $N$, $\tilde{M}$ is the proper transform of $M$. Then $X$ has an $H$-polar cylinder.

Assume that $H=-K_X+\mu C+\lambda \varphi(E)$, where $C$ is a $(0)$-curve and $\mu>0$, $0\leq\lambda<5$, $E$ is a $(-1)$-curve on $\tilde{X}$ that meets only $D_1$ and does not meet $D_2,D_3,D_4,D_5, D_6$. Put $\tilde{C}$ is the proper transform of $C$. Note that $$\tilde{C}\sim E+D_1+D_2+D_3+D_4+E',$$ where $E'$ is a $(-1)$-curve that meets only $D_4$ and does not meet $D_1,D_2,$ $D_3,D_5, D_6$. Consider the configuration as in Figure 22. Note that $E'\cdot E_1=E'\cdot E_2=0$ and $E'\cdot E_3=1$. We have \begin{gather*}\varphi^*(H)\equiv L=aD_1+2aD_2+(2a-1-2\mu-\frac{3}{5}\lambda)E_3+\\
+(2a-1-\mu-\frac{3}{5}\lambda)D_5+2(2a-1-\mu-\frac{3}{5}\lambda)E_2+bD_4+2bD_3+\\
+(2b-1-\mu-\frac{2}{5}\lambda)D_6+2(2b-1-\mu-\frac{2}{5}\lambda)E_1+(b-2\mu-\frac{6}{5}\lambda)M,\end{gather*} where $a+b=1+2\mu+\lambda$, $M$ is a $(0)$-curve such that $M$ pass through the intersection point of $D_2$ and $D_3$ and $M\cdot D_2=M\cdot D_2=3$. Assume that $2\mu+\lambda<1$. Then there exist $a$ and $b$ such that $L$ is an effective divisor. As above, $X$ has an $H$-polar cylinder. So, we may assume that $2\mu+\lambda\geq 1$. Also, consider the configuration as in Figure 22. We have \begin{gather*}\varphi^*(H)\equiv L=(a+\frac{4}{5}\lambda)D_1+(a-1+\lambda)E+(a+1+\frac{3}{5}\lambda)D_2+\\
+(a-2\mu)E_3+(a-\mu)D_5+2(a-\mu)E_2+(b+\frac{1}{5}\lambda)D_4+\\
+(2b+\frac{2}{5}\lambda)D_3+(2b-1-\mu)D_6+2(2b-1-\mu)E_1,\end{gather*} where $a+b=1+2\mu$. Assume that $\mu<1$. Since $2\mu+\lambda\geq 1$, we see that there exist $a$ and $b$ such that $L$ is an effective divisor. As above, $X$ has an $H$-polar cylinder. So, we may assume that $\mu\geq 1$. We have the following configuration on $\tilde{X}$.\\
\begin{tikzpicture}
\draw  (0.5,0) -- (0.5,3) (0,0.5) -- (8,0.5) (7.5,0)--(7.5,4) (7.2,3.5)--(10,3.5) (3,1.5)--(4,1.5) (6.5,1)--(6.5,3); \draw [dashed](6,1.5)--(8,1.5) (3.5,0)--(3.5,2) (0,2.5)--(2,2.5) (9.5,4)--(9.5,1);
\draw (1,0.5) node [above] {$D_2$} (0.5,3) node [above] {$D_1$} (7.5,4) node [above] {$D_3$}  (10,3.5) node [right] {$D_4$} (4,1.5) node [right] {$D_5$}  (6,1.5) node [left] {$E_1$} (3.5,2) node [above] {$E_2$} (6.5,3) node [right] {$D_6$} (2,2.5) node [right] {$E$} (9.5,1) node [right] {$E'$};
\end{tikzpicture}$$\text{Figure 25}$$ where dotted lines are $(-1)$-curves, $D_1,\ldots,D_6$ are $(-2)$-curves. We have \begin{gather*}\varphi^*(H)\equiv L=(a-2+\mu+\lambda)E+(a-1+\mu+\frac{4}{5}\lambda)D_1+\\
+(a+\mu+\frac{3}{5}\lambda)D_2+(a-1)D_5+2(a-1)E_2+(b+\mu+\frac{2}{5}\lambda)D_3+\\
+(b-1+\mu+\frac{1}{5}\lambda)D_4+(b-1)D_6+2(b-1)E_1+(b-2+\mu)E,\end{gather*} where $a+b=3$. Since $\mu\geq 1$, we see that there exist $a$ and $b$ such that $L$ is an effective divisor. Let $Y\rightarrow\tilde{X}$ be the blow up of the intersection point of $D_2$ and $D_3$, and $F$ be the exceptional divisor. Note that there exists $\PP^1$-fibration $g\colon Y_2\rightarrow\PP^1$ such that $g$ has only two singular fibers $C_1, C_2$ and $$C_1=\tilde{D}_1+\tilde{D}_2+\tilde{D}_5+\tilde{E}+2\tilde{E}_2,$$ $$C_2=\tilde{D}_4+\tilde{E'}+\tilde{D}_3+\tilde{D}_6+2\tilde{E}_1,$$ where $\tilde{D}_1,\ldots,\tilde{D}_6$ are the proper transforms of $D_1,\ldots, D_6$, $\tilde{E}_1, \tilde{E}_2, \tilde{E}_3$, $\tilde{E}, \tilde{E'}$ are the proper transforms of $E_1,E_2,E_3, E, E'$. Then $X$ has an $H$-polar cylinder.
\end{proof}

\begin{lemma}
\label{LemA5}
Let $X$ be a del Pezzo surface with du Val singularities and let $H$ be an ample divisor on $X$. Assume that $X$ has a unique singular point $P$. Assume that $P$ is of type $A_5$. Then $X$ has an $H$-polar cylinder.
\end{lemma}

\begin{proof}
Let $\varphi\colon\tilde{X}\rightarrow X$ be the minimal resolution of singularities of $X$ and let $D=\sum_{i=1}^5 D_i$ be the exceptional divisor of $\varphi$. Assume that $H=-K_X+\lambda_1 \hat{E}_1+\lambda_2 \hat{E}_2+\lambda\varphi(E)$, where $\hat{E}_1, \hat{E}_2$ are $(-1)$-curves and $\lambda_2\leq\lambda_1<1$, $\lambda<6$. Put $\tilde{E}_1$, $\tilde{E}_2$ are the proper transform of $\hat{E}_1, \hat{E}_2$. We have the following configuration on $\tilde{X}$.\\
\begin{tikzpicture}
\draw  (0.5,0) -- (0.5,3) (0,0.5) -- (8,0.5) (7.5,0)--(7.5,4) (7.2,3.5)--(11,3.5)  (10.5,4)--(10.5,1); \draw [dashed](6,1.5)--(8,1.5) (6,2.5)--(8,2.5) (3.5,0)--(3.5,3) (3.3,2)--(4,4.3)(4.5,0)--(4.5,3) (4.3,2)--(5,4.3) (0,2.5)--(2,2.5) (9.5,4)--(9.5,1);
\draw (1,0.5) node [above] {$D_2$} (0.5,3) node [above] {$D_1$} (7.5,4) node [above] {$D_3$}  (11,3.5) node [right] {$D_4$}  (6,1.5) node [left] {$E_3$} (6,2.5) node [left] {$E_4$}  (3.5,1.5) node [left] {$E_1$} (4.5,1.5) node [left] {$E_2$}   (2,2.5) node [right] {$E$} (10.5,1) node [right] {$D_5$} (9.5,1) node [right] {$E_5$} (4,4.3) node [right] {$\tilde{E}_1$} (5,4.3) node [right] {$\tilde{E}_2$};
\end{tikzpicture}$$\text{Figure 26}$$ where dotted lines are $(-1)$-curves, $D_1,\ldots,D_5$ are $(-2)$-curves. Moreover, $E\cdot E_3=1$ $E\cdot E_1=E\cdot E_2=E\cdot E_4=E\cdot E_5=0$. We have $$\lambda_1\tilde{E}_1\equiv-\lambda_1 E_1+\lambda_1 D_3+\lambda_1 D_5+2\lambda_1 D_4+2\lambda_1 E_5,$$ $$\lambda_2\tilde{E}_2\equiv-\lambda_2 E_2+\lambda_2 D_3+\lambda_2 D_5+2\lambda_2 D_4+2\lambda_2 E_5.$$ Then \begin{gather*}\varphi^*(H)\equiv L=aD_1+2aD_2+(2a-1-\lambda_1-\frac{4}{6}\lambda)E_1+\\
+(2a-1-\lambda_2-\frac{4}{6}\lambda)E_2+(b+\lambda_1+\lambda_2)D_5+(2b+2\lambda_1+2\lambda_2)D_4+\\
+(2b-1+2\lambda_1+2\lambda_2-\frac{2}{6}\lambda)E_5+(b+1+\lambda_1+\lambda_2+\frac{2}{6}\lambda)D_3+\\
+(b-\frac{7}{6}\lambda)E_3+(b-\frac{1}{6}\lambda)E_4,\end{gather*} where $a+b=1+\lambda$. Assume that $\lambda_1+\lambda<1$. Then there exist $a$ and $b$ such that $L$ is an effective divisor. Let $Y_1\rightarrow\tilde{X}$ be the blow up of the intersection point of $D_2$ and $D_3$, and $N$ be the exceptional divisor. Let $Y_2\rightarrow Y_1$ be the blow up of the intersection point of $N$ and the proper transform of $D_3$, and $F$ be the exceptional divisor. Note that there exists $\PP^1$-fibration $g\colon Y_2\rightarrow\PP^1$ such that $g$ has only two singular fibers $C_1, C_2$ and $$C_1=\bar{D}_1+2\bar{D}_2+\bar{N}+2\bar{E}_1+2\bar{E}_2,$$ $$C_2=\bar{D}_5+\bar{D}_3+2\bar{D}_4+2\bar{E}_5+\bar{E}_3+\bar{E}_4,$$ where $\bar{D}_1,\ldots,\bar{D}_5$ are the proper transforms of $D_1,\ldots, D_5$, $\bar{E}_1, \ldots, \bar{E}_5$ are the proper transforms of $E_1,\ldots,E_5$, $\tilde{N}$ is the proper transform of $N$. Then $X$ has an $H$-polar cylinder. In particular, if $\lambda=0$ then $X$ has an $H$-polar cylinder. Assume that $\lambda+\lambda_1\geq 1$. We have \begin{gather*}\varphi^*(H)\equiv L=(a+\frac{5}{6}\lambda)D_1+(a-1+\lambda)E+(a+1+\frac{4}{6}\lambda)D_2+\\
+(a-\lambda_1)E_1+(a-\lambda_2)E_2+(b+\lambda_1+\lambda_2+\frac{1}{6}\lambda)D_5+\\
+(2b+2\lambda_1+2\lambda_2+\frac{2}{6}\lambda)D_4+(2b-1+2\lambda_1+2\lambda_2)E_5+\\
+(b+1+\lambda_1+\lambda_2+\frac{3}{6}\lambda)D_3+bE_4,\end{gather*} where $a+b=1$. Since $\lambda+\lambda_1\geq 1$, we see that there exist $a$ and $b$ such that $L$ is an effective divisor. Let $Y\rightarrow\tilde{X}$ be the blow up of the intersection point of $D_2$ and $D_3$, and $F$ be the exceptional divisor. Note that there exists $\PP^1$-fibration $g\colon Y\rightarrow\PP^1$ such that $g$ has only two singular fibers $C_1, C_2$ and $$C_1=\bar{D}_1+2\bar{D}_2+\bar{E}_1+\bar{E}_2+\bar{E},$$ $$C_2=\bar{D}_5+\bar{D}_3+2\bar{D}_4+2\bar{E}_5+\bar{E}_4,$$ where $\bar{D}_1,\ldots,\bar{D}_5$ are the proper transforms of $D_1,\ldots, D_5$, $\bar{E}_1, \ldots, \bar{E}_5, \bar{E}$ are the proper transforms of $E_1,\ldots,E_5, E$. Then $X$ has an $H$-polar cylinder.

Assume that $H=-K_X+\mu C+\lambda \varphi(E)+\hat{E}_1$, where $C$ is a $(0)$-curve and $\mu>0$, $0\leq\lambda<6$, $0\leq \lambda_1<1$, $E$ is a $(-1)$-curve on $\tilde{X}$ that meets only $D_1$ and does not meet $D_2,D_3,D_4,D_5$, $\hat{E}_1$ is a $(-1)$-curve on $X$. Put $\tilde{E}_1$ is the proper transform of $\hat{E}_1$. Put $\tilde{C}$ is the proper transform of $C$. Note that $$\tilde{C}\sim E+D_1+D_2+D_3+D_4+D_5+E',$$ where $E'$ is a $(-1)$-curve that meets only $D_5$ and does not meet $D_1,D_2,$ $D_3,D_4$. Consider the configuration as in Figure 26. Note that \\$E'\cdot E_1=E'\cdot E_3=1$, $E'\cdot E_2=E'\cdot E_4=E'\cdot E_5=0$. We have \begin{gather*}\varphi^*(H)\equiv L=aD_1+2aD_2+(2a-1-\lambda_1-\frac{4}{6}\lambda-2\mu)E_1+\\
+(2a-1-\frac{4}{6}\lambda-\mu)E_2+bD_5+2bD_4+(2b-1-\frac{2}{6}\lambda-\mu)E_5+\\
+(b+1+\frac{2}{6}\lambda+\mu)D_3+(b-\lambda_1-\frac{7}{6}\lambda-2\mu)E_3+(b-\lambda_1-\frac{1}{6}\lambda)E_4,\end{gather*} where $a+b=1+\lambda+\lambda_1+2\mu$. Note that $L$ is an effective divisor if and only if $a>\frac{1}{2}+\frac{1}{2}\lambda_1+\frac{2}{6}\lambda+\mu$ and $b>\lambda_1+\frac{7}{6}\lambda+2\mu$. We see that if $\lambda_1+\lambda+2\mu<1$ then there exist $a$ and $b$ such that $L$ is an effective divisor. As above, $X$ has an $H$-polar cylinder. So, we may assume that $\lambda_1+\lambda+2\mu\geq 1$. Also, consider the configuration as in Figure 26. We have \begin{gather*}\varphi^*(H)\equiv L=(a+\frac{5}{6}\lambda)D_1+(a-1+\lambda)E+(a+1)D_2+\\
+(a-\lambda_1-2\mu)E_1+(a-\mu)E_2+(b+\frac{1}{6}\lambda)D_5+(2b+\frac{2}{6}\lambda)D_4+\\
+(2b-1-\mu)E_5+(b+1+\frac{3}{6}\lambda+\mu)D_3+(b-\lambda_1)E_4,\end{gather*} where $a+b=1+\lambda_1+2\mu$. Since $\lambda_1+\lambda+2\mu\geq 1$, we see that $L$ is an effective divisor if and only if $$\begin{cases} a>\lambda_1+2\mu\\
b>\frac{1}{2}+\frac{1}{2}\mu\\
b>\lambda_1.\end{cases}$$ Assume that $\lambda_1\geq\frac{1}{2}+\frac{1}{2}\mu$. Since $\lambda_1<1$, we see that there exist $a$ and $b$ such that $L$ is an effective divisor. As above, $X$ has an $H$-polar cylinder. Assume that $\lambda_1<\frac{1}{2}+\frac{1}{2}\mu$. We see that if $\mu<1$, then there exist $a$ and $b$ such that $L$ is an effective divisor. As above, $X$ has an $H$-polar cylinder. So, we may assume that $\mu\geq 1$. We have the following configuration on $\tilde{X}$.\\
\begin{tikzpicture}
\draw  (0.5,0) -- (0.5,4) (0,0.5) -- (8,0.5) (0,3)--(2,3) (7.5,0)--(7.5,4) (7.2,3.5)--(11,3.5); \draw [dashed](0,2)--(2,2) (1.5,2.5)--(1.5,4) (6,1.5)--(8,1.5) (4.5,0)--(4.5,3) (4.3,2)--(5,4.3) (9.5,4)--(9.5,1);
\draw (1,0.5) node [above] {$D_3$} (0.5,4) node [above] {$D_2$} (7.5,4) node [above] {$D_4$}  (11,3.5) node [right] {$D_5$}  (6,1.5) node [left] {$E_3$}  (4.5,1.5) node [left] {$E_1$}   (2,2) node [right] {$E_2$} (2,3) node [right] {$D_1$} (9.5,1) node [right] {$E'$} (5,4.3) node [right] {$\tilde{E}_1$} (1.5,4)  node [above] {$E$};
\end{tikzpicture}$$\text{Figure 27}$$ where dotted lines are $(-1)$-curves, $D_1,\ldots,D_5$ are $(-2)$-curves. We have \begin{gather*}\varphi^*(H)\equiv L=(2+\frac{3}{6}\lambda+\mu)D_3+(a+\frac{4}{6}\lambda+\mu)D_2+\\
+(a-1+\frac{5}{6}\lambda+\mu)D_1+(a-1)E_2+(a-2+\lambda+\mu)E\\
+(b+\frac{2}{6}\lambda+\mu)D_4+(b-1+\frac{1}{6}\lambda+\mu)D_5+(b-1)E_3+(b-2+\mu)E'+\\
+cE_1+(c-1+\lambda_1)\tilde{E}_1,\end{gather*} where $a+b+c=4$. Since $\mu\geq 1$, we see that there exist $a,b,c$ such that $L$ is an effective divisor. On the other hand, $$\tilde{X}\backslash\Supp(L)\cong X\backslash\Supp(\varphi(L))\cong\FF_2\backslash(K+F_1+F_2+F_3),$$ where $K$ is a unique $(-2)$-curve and $F_1, F_2, F_3$ are fibers. Then $X$ has an $H$-polar cylinder.
\end{proof}

\begin{lemma}
\label{LemA5A1}
Let $X$ be a del Pezzo surface with du Val singularities and let $H$ be an ample divisor on $X$. Assume that $X$ has the following collection of singularities $A_5+A_1$. Then $X$ has an $H$-polar cylinder.
\end{lemma}
\begin{proof}
Let $\varphi\colon\tilde{X}\rightarrow X$ be the minimal resolution of singularities of $X$ and let $D=\sum_{i=1}^6 D_i$ be the exceptional divisor of $\varphi$. We may assume that $D_1,\ldots,D_5$ correspond to the singular point of type $A_5$, $D_6$ corresponds to the singular point of type $A_1$. Assume that $H=-K_X+\lambda_1 \hat{E}_1+\lambda_2 \hat{E}_2$, where $\hat{E}_1, \hat{E}_2$ are $(-1)$-curves and $0\leq\lambda_2\leq\lambda<1$. We may assume that $\lambda_1>0$. Put $\tilde{E}_1, \tilde{E}_2$ are the proper transform of $\hat{E}_1, \hat{E}_2$. We have the following configuration on $\tilde{X}$.\\
\begin{tikzpicture}
\draw  (0.5,0) -- (0.5,3) (0,0.5) -- (8,0.5) (7.5,0)--(7.5,4) (7.2,3.5)--(11,3.5)  (10.5,4)--(10.5,1) (6.5,1)--(6.5,3); \draw [dashed](6,1.5)--(8,1.5) (3.5,0)--(3.5,3) (3.3,2)--(4,4.3)(4.5,0)--(4.5,3) (4.3,2)--(5,4.3) (9.5,4)--(9.5,1);
\draw (1,0.5) node [above] {$D_2$} (0.5,3) node [above] {$D_1$} (7.5,4) node [above] {$D_3$}  (11,3.5) node [right] {$D_4$}  (6,1.5) node [left] {$E_3$} (3.5,1.5) node [left] {$E_1$} (4.5,1.5) node [left] {$E_2$}   (10.5,1) node [right] {$D_5$} (9.5,1) node [right] {$E_4$} (4,4.3) node [right] {$\tilde{E}_1$} (5,4.3) node [right] {$\tilde{E}_2$} (6.5,3) node [above] {$D_6$};
\end{tikzpicture}$$\text{Figure 28}$$ where dotted lines are $(-1)$-curves, $D_1,\ldots,D_6$ are $(-2)$-curves. Note that $$\lambda_1\tilde{E}_1\equiv-\lambda_1 E_1+\lambda_1 D_3+\lambda_1 D_5+2\lambda_1 D_4+2\lambda_1 E_4,$$ $$\lambda_2\tilde{E}_2\equiv-\lambda_2 E_2+\lambda_2 D_3+\lambda_2 D_5+2\lambda_2 D_4+2\lambda_2 E_4.$$ We have \begin{gather*}\varphi^*(H)\equiv L=aD_1+2aD_2+(2a-1-\lambda_1)E_1+(2a-1-\lambda_2)E_2+\\
+(b+\lambda_1+\lambda_2)D_3+(b-1)D_6+2(b-1)E_3+(b-1+\lambda_1+\lambda_2)D_5+\\
+(2b-2+2\lambda_1+2\lambda_2)D_4+(2b-3+2\lambda_1+2\lambda_2)E_4,\end{gather*} where $a+b=2$. Since $\lambda_1>0$, we see that there exist $a$ and $b$ such that $L$ is an effective divisor. Let $Y_1\rightarrow\tilde{X}$ be the blow up of the intersection point of $D_2$ and $D_3$, and $N$ be the exceptional divisor. Let $Y_2\rightarrow Y_1$ be the blow up of the intersection point of $N$ and the proper transform of $D_3$, and $F$ be the exceptional divisor. Note that there exists $\PP^1$-fibration $g\colon Y_2\rightarrow\PP^1$ such that $g$ has only two singular fibers $C_1, C_2$ and $$C_1=\bar{D}_1+2\bar{D}_2+\bar{N}+2\bar{E}_1+2\bar{E}_2,$$ $$C_2=\bar{D}_5+\bar{D}_3+2\bar{D}_4+2\bar{E}_4+2\bar{E}_3+\bar{D}_6,$$ where $\bar{D}_1,\ldots,\bar{D}_6$ are the proper transforms of $D_1,\ldots, D_6$, $\bar{E}_1, \ldots, \bar{E}_4$ are the proper transforms of $E_1,\ldots,E_4$, $\bar{N}$ is the proper transform of $N$. Then $X$ has an $H$-polar cylinder.

Assume that $H=-K_X+\lambda_1 \hat{E}_1+\lambda\varphi(E)$, where $\hat{E}_1$ is a $(-1)$-curve, $E$ is a $(-1)$-curve on $\tilde{X}$ that meets only $D_1$ and does not meet $D_2,\ldots, D_6$ and $\lambda_1<1$, $\lambda<6$. Put $\tilde{E}_1$ is the proper transform of $\hat{E}_1$. We may assume that $\lambda>0$. We have the following configuration on $\tilde{X}$.\\
\begin{tikzpicture}
\draw  (0.5,0) -- (0.5,3) (0,0.5) -- (8,0.5) (7.5,0)--(7.5,4) (7.2,3.5)--(11,3.5)  (10.5,4)--(10.5,1) (3,2.5)--(5,2.5); \draw [dashed](6,1.5)--(8,1.5) (6,2.5)--(8,2.5) (3.5,0)--(3.5,3) (9.5,4)--(9.5,1) (0,2)--(2,2);
\draw (1,0.5) node [above] {$D_2$} (0.5,3) node [above] {$D_1$} (7.5,4) node [above] {$D_3$}  (11,3.5) node [right] {$D_4$}  (6,1.5) node [left] {$E_3$} (3.5,1.5) node [left] {$E_1$}   (10.5,1) node [right] {$D_5$} (9.5,1) node [right] {$E_4$} (4.5,2.5) node [above] {$D_6$} (6.5,2.5) node [above] {$E_2$} (2,2) node [right] {$E$};
\end{tikzpicture}$$\text{Figure 29}$$ where dotted lines are $(-1)$-curves, $D_1,\ldots,D_6$ are $(-2)$-curves. Morover, $\tilde{E}_1\cdot E_1=\tilde{E}_1\cdot E_4=1$, $\tilde{E}_1\cdot E_2=0$, $\tilde{E}_1\cdot E_3=2$, $E\cdot E_3=1$, $E\cdot E_1=E\cdot E_2=E\cdot E_4=0$. We have \begin{gather*}\varphi^*(H)\equiv L=aD_1+2aD_2+(2a-1-\lambda_1-\frac{4}{6}\lambda)D_6+\\
+2(2a-1-\lambda_1-\frac{4}{6}\lambda)E_1+bD_5+2bD_4+(2b-1-\lambda_1-\frac{2}{6}\lambda)E_4+\\
+(b+1+\lambda_1+\frac{2}{6}\lambda)D_3+(b+\lambda_1-\frac{1}{6}\lambda)E_2+(b-\lambda_1-\frac{7}{6}\lambda)E_3,\end{gather*} where $a+b=1+\lambda_1+\lambda$. Note that $L$ is an effective divisor if and only if $a>\frac{1}{2}+\frac{1}{2}\lambda_1+\frac{2}{6}\lambda$ and $b>\lambda_1+\frac{7}{6}\lambda$. Assume that $\lambda_1+\lambda<1$. We see that there exist $a$ and $b$ such that $L$ is an effective divisor. Let $Y_1\rightarrow\tilde{X}$ be the blow up of the intersection point of $D_2$ and $D_3$, and $N$ be the exceptional divisor. Let $Y_2\rightarrow Y_1$ be the blow up of the intersection point of $N$ and the proper transform of $D_3$, and $F$ be the exceptional divisor. Note that there exists $\PP^1$-fibration $g\colon Y_2\rightarrow\PP^1$ such that $g$ has only two singular fibers $C_1, C_2$ and $$C_1=\bar{D}_1+2\bar{D}_2+\bar{N}+2\bar{D}_6+4\bar{E}_1,$$ $$C_2=\bar{D}_5+\bar{D}_3+2\bar{D}_4+2\bar{E}_4+\bar{E}_2+\bar{E}_3,$$ where $\bar{D}_1,\ldots,\bar{D}_6$ are the proper transforms of $D_1,\ldots, D_6$, $\bar{E}_1, \ldots, \bar{E}_4$ are the proper transforms of $E_1,\ldots,E_4$, $\bar{N}$ is the proper transform of $N$. Then $X$ has an $H$-polar cylinder. So, we may assume that $\lambda_1+\lambda\geq 1$. We have \begin{gather*}\varphi^*(H)\equiv L=(a+\frac{5}{6}\lambda)D_1+(a-1+\lambda)E+(a+1+\frac{4}{6}\lambda)D_2+\\
+(a-\lambda_1)D_6+2(a-\lambda_1)E_1+(b+\frac{1}{6}\lambda)D_5+(2b+\frac{2}{6}\lambda)D_4+\\
+(2b-1-\lambda_1)E_4+(b+1+\lambda_1+\frac{3}{6}\lambda)D_3+(b+\lambda_1)E_2,\end{gather*} where $a+b=1+\lambda_1$. Since $\lambda_1+\lambda\geq 1$, we see that $L$ is an effective divisor if and only if $a>\lambda_1$ and $b>\frac{1}{2}+\frac{1}{2}\lambda_1$. Since $\lambda_1<1$, we see that there exist $a$ and $b$ such that $L$ is an effective divisor. Let $Y\rightarrow\tilde{X}$ be the blow up of the intersection point of $D_2$ and $D_3$, and $F$ be the exceptional divisor. Note that there exists $\PP^1$-fibration $g\colon Y\rightarrow\PP^1$ such that $g$ has only two singular fibers $C_1, C_2$ and $$C_1=\bar{E}+\bar{D}_1+\bar{D}_2+\bar{D}_6+2\bar{E}_1,$$ $$C_2=\bar{D}_5+\bar{D}_3+2\bar{D}_4+2\bar{E}_4+\bar{E}_2,$$ where $\bar{D}_1,\ldots,\bar{D}_6$ are the proper transforms of $D_1,\ldots, D_6$, $\bar{E}_1, \bar{E}_2, \bar{E}_4, \bar{E}$ are the proper transforms of $E_1,E_2,E_4, E$. Then $X$ has an $H$-polar cylinder.

Assume that $H=-K_X+\lambda_1 \varphi(E')+\lambda\varphi(E)$, where $E'$ is a $(-1)$-curve on $\tilde{X}$ that meets only $D_6$ and does not meet $D_1,\ldots, D_5$, $E$ is a $(-1)$-curve on $\tilde{X}$ that meets only $D_1$ and does not meet $D_2,\ldots, D_6$ and $\lambda_1<2$, $\lambda<6$. Consider the configuration as in Figure 29. Note that $E'\cdot E_1=E'\cdot E_4=0$, $E'\cdot E_2=E'\cdot E_3=1$. We have \begin{gather*}\varphi^*(H)\equiv L=aD_1+2aD_2+(2a-1-\frac{1}{2}\lambda_1-\frac{4}{6}\lambda)D_6+\\
+2(2a-1-\frac{1}{2}\lambda_1-\frac{4}{6}\lambda)E_1+bD_5+2bD_4+(2b-1-\frac{2}{6}\lambda)E_4+\\
+(b+1+\frac{2}{6}\lambda)D_3+(b-\lambda_1-\frac{1}{6}\lambda)E_2+(b-\lambda_1-\frac{7}{6}\lambda)E_3,\end{gather*} where $a+b=1+\lambda_1+\lambda$. Note that $L$ is an effective divisor if and only if $a>\frac{1}{2}+\frac{1}{4}\lambda_1+\frac{2}{6}\lambda$ and $b>\lambda_1+\frac{7}{6}\lambda$. Assume that $\frac{1}{2}\lambda_1+\lambda<1$. We see that there exist $a$ and $b$ such that $L$ is an effective divisor. As above, $X$ has an $H$-polar cylinder. So, we may assume that $\frac{1}{2}\lambda_1+\lambda\geq 1$.  Also, consider the configuration as in Figure 29. We have \begin{gather*}\varphi^*(H)\equiv L=(a+\frac{5}{6}\lambda)D_1+(a-1+\lambda)E+(a+1+\frac{4}{6}\lambda)D_2+\\
+(a-\frac{1}{2}\lambda_1)D_6+2(a-\frac{1}{2}\lambda_1)E_1+(b+\frac{1}{6}\lambda)D_5+(2b+\frac{2}{6}\lambda)D_4+\\
+(2b-1)E_4+(b+1+\frac{3}{6}\lambda)D_3+(b-\lambda_1)E_2,\end{gather*} where $a+b=1+\lambda_1$. Since $\frac{1}{2}\lambda_1+\lambda\geq 1$ and $\lambda_1<2$, we see that there exist $a$ and $b$ such that $L$ is an effective divisor. As above, $X$ has an $H$-polar cylinder.
\end{proof}

\begin{lemma}
\label{LemA5A1A1}
Let $X$ be a del Pezzo surface with du Val singularities and let $H$ be an ample divisor on $X$. Assume that $X$ has the following collection of singularities $A_5+A_1+A_1$. Then $X$ has an $H$-polar cylinder.
\end{lemma}

\begin{proof}
Let $\varphi\colon\tilde{X}\rightarrow X$ be the minimal resolution of singularities of $X$ and let $D=\sum_{i=1}^7 D_i$ be the exceptional divisor of $\varphi$. We may assume that $D_1,\ldots,D_5$ correspond to the singular point of type $A_5$, $D_6, D_7$ correspond to the singular points of type $A_1$. Note that $H=-K_X+\lambda\varphi(E)$, where $E$ is a $(-1)$-curve on $\tilde{X}$ that meets only $D_7$ and does not meet $D_1,\ldots, D_6$ and  $0<\lambda<2$. We have the following configuration on $\tilde{X}$.\\
\begin{tikzpicture}
\draw  (0.5,0) -- (0.5,3) (0,0.5) -- (8,0.5) (7.5,0)--(7.5,4) (7.2,3.5)--(11,3.5)  (10.5,4)--(10.5,1) (6.5,1)--(6.5,3) (3,2.5)--(5,2.5); \draw [dashed](6,1.5)--(8,1.5) (3.5,0)--(3.5,3)  (9.5,4)--(9.5,1);
\draw (1,0.5) node [above] {$D_2$} (0.5,3) node [above] {$D_1$} (7.5,4) node [above] {$D_3$}  (11,3.5) node [right] {$D_4$}  (6,1.5) node [left] {$E_2$} (3.5,1.5) node [left] {$E_1$} (10.5,1) node [right] {$D_5$} (9.5,1) node [right] {$E_3$} (6.5,3) node [above] {$D_6$} (5,2.5) node [above] {$D_7$} ;
\end{tikzpicture}$$\text{Figure 30}$$ where dotted lines are $(-1)$-curves, $D_1,\ldots,D_7$ are $(-2)$-curves. Moreover, $E\cdot E_1=E\cdot E_3=0$, $E\cdot E_2=1$. We have \begin{gather*}\varphi^*(H)\equiv L=aD_1+2aD_2+(2a-1-\frac{1}{2}\lambda)D_7+2(2a-1-\frac{1}{2}\lambda)E_1+\\
+bD_5+2bD_4+(2b-1)E_3+(b+1)D_3+(b-\lambda)D_6+2(b-\lambda)E_2,\end{gather*} where $a+b=1+\lambda$. Since $\lambda<2$, we see that there exist $a$ and $b$ such that $L$ is an effective divisor. Let $Y_1\rightarrow\tilde{X}$ be the blow up of the intersection point of $D_2$ and $D_3$, and $N$ be the exceptional divisor. Let $Y_2\rightarrow Y_1$ be the blow up of the intersection point of $N$ and the proper transform of $D_3$, and $F$ be the exceptional divisor. Note that there exists $\PP^1$-fibration $g\colon Y_2\rightarrow\PP^1$ such that $g$ has only two singular fibers $C_1, C_2$ and $$C_1=\bar{D}_1+2\bar{D}_2+\bar{N}+2\bar{D}_7+4\bar{E}_1,$$ $$C_2=\bar{D}_5+\bar{D}_3+2\bar{D}_4+2\bar{E}_3+2\bar{E}_2+\bar{D}_6,$$ where $\bar{D}_1,\ldots,\bar{D}_7$ are the proper transforms of $D_1,\ldots, D_7$, $\bar{E}_1, \bar{E}_2, \bar{E}_3$ are the proper transforms of $E_1,E_2,E_3$, $\bar{N}$ is the proper transform of $N$. Then $X$ has an $H$-polar cylinder.
\end{proof}

\begin{lemma}
\label{LemA5A2}
Let $X$ be a del Pezzo surface with du Val singularities and let $H$ be an ample divisor on $X$. Assume that $X$ has the following collection of singularities $A_5+A_2$. Then $X$ has an $H$-polar cylinder.
\end{lemma}

\begin{proof}
Let $\varphi\colon\tilde{X}\rightarrow X$ be the minimal resolution of singularities of $X$ and let $D=\sum_{i=1}^7 D_i$ be the exceptional divisor of $\varphi$. We may assume that $D_1,\ldots,D_5$ correspond to the singular point of type $A_5$, $D_6, D_7$ correspond to the singular point of type $A_2$. Note that $H=-K_X+\lambda E$, where $E$ is a $(-1)$-curve and $0<\lambda<1$. Put $\tilde{E}$ is the proper transform of $E$. We have the following configuration on $\tilde{X}$.\\
\begin{tikzpicture}
\draw  (0.5,0) -- (0.5,3) (0,0.5) -- (8,0.5) (7.5,0)--(7.5,4) (7.2,3.5)--(11,3.5)  (10.5,4)--(10.5,1) (3.5,1)--(3.5,3) (1.5,2.5)--(4,2.5); \draw [dashed](6,1.5)--(8,1.5) (6,2.5)--(8,2.5)  (2,0)--(2,3)  ;
\draw (1,0.5) node [above] {$D_2$} (0.5,3) node [above] {$D_1$} (7.5,4) node [above] {$D_3$}  (11,3.5) node [right] {$D_4$}  (6,1.5) node [left] {$E_2$} (2,1.5) node [left] {$E_1$} (10.5,1) node [right] {$D_5$} (6,2.5) node [above] {$E_3$} (3.5,3) node [above] {$D_6$} (4,2.5) node [right] {$D_7$} ;
\end{tikzpicture}$$\text{Figure 31}$$ where dotted lines are $(-1)$-curves, $D_1,\ldots,D_7$ are $(-2)$-curves. Moreover, $E\cdot E_1=1$ $E\cdot E_2=0$, $E\cdot E_3=2$. We have \begin{gather*}\varphi^*(H)\equiv L=aD_1+2aD_2+(2a-1-\lambda)D_6+2(2a-1-\lambda)D_7+\\
+3(2a-1-\lambda)E_1+bD_5+2bD_4+3bD_3+(3b-1)E_2+(3b-1-2\lambda)E_3,\end{gather*} where $a+b=1+\lambda$. Since $\lambda<1$, we see that there exist $a$ and $b$ such that $L$ is an effective divisor. Let $Y_1\rightarrow\tilde{X}$ be the blow up of the intersection point of $D_2$ and $D_3$, and $N_1$ be the exceptional divisor. Let $Y_2\rightarrow Y_1$ be the blow up of the intersection point of $N_1$ and the proper transform of $D_2$, and $N_2$ be the exceptional divisor. Let $Y_3\rightarrow Y_2$ be the blow up of the intersection point of $N_2$ and the proper transform of $N_1$, and $F$ be the exceptional divisor. Note that there exists $\PP^1$-fibration $g\colon Y_3\rightarrow\PP^1$ such that $g$ has only two singular fibers $C_1, C_2$ and $$C_1=\bar{D}_1+2\bar{D}_2+\bar{N}_2+2\bar{D}_6+4\bar{D}_7+6\bar{E}_1,$$ $$C_2=\bar{D}_5+\bar{N}_1+2\bar{D}_4+3\bar{D}_3+3\bar{E}_2+3\bar{E}_3,$$ where $\bar{D}_1,\ldots,\bar{D}_7$ are the proper transforms of $D_1,\ldots, D_7$, $\bar{E}_1, \bar{E}_2, \bar{E}_3$ are the proper transforms of $E_1,E_2,E_3$, $\bar{N}_1,\bar{N}_2$ are the proper transforms of $N_1, N_2$. Then $X$ has an $H$-polar cylinder.
\end{proof}

\begin{lemma}
\label{LemA6}
Let $X$ be a del Pezzo surface with du Val singularities and let $H$ be an ample divisor on $X$. Assume that $X$ has a unique singular point $P$. Assume that $P$ is of type $A_6$. Then $X$ has an $H$-polar cylinder.
\end{lemma}

\begin{proof}
Let $\varphi\colon\tilde{X}\rightarrow X$ be the minimal resolution of singularities of $X$ and let $D=\sum_{i=1}^6 D_i$ be the exceptional divisor of $\varphi$. Note that $H=-K_X+\lambda_1 \hat{E}_1+\lambda\varphi(E)$, where $\hat{E}_1$ is $(-1)$-curves $E$ is a $(-1)$-curve on $\tilde{X}$ that meets only $D_1$ and does not meet $D_2,\ldots,D_6$. Note that $\lambda_1<1$, $\lambda<7$. Put $\tilde{E}_1$ is the proper transform of $\hat{E}_1$. We have the following configuration on $\tilde{X}$.\\
\begin{tikzpicture}
\draw  (0.5,0) -- (0.5,3) (0,0.5) -- (4,0.5) (3.5,0)--(3.5,4) (3.2,3.5)--(9,3.5)  (8.5,4)--(8.5,1) (8,1.5)--(11,1.5); \draw [dashed] (2,0)--(2,3) (3,1.5)--(5,1.5) (6,4)--(6,2) (8,2)--(10,2);
\draw (1,0.5) node [above] {$D_2$} (0.5,3) node [above] {$D_1$} (3.5,4) node [above] {$D_3$}  (9,3.5) node [right] {$D_4$}  (8.5,1) node [right] {$D_5$} (11,1.5) node [right] {$D_6$} (2,3) node [above] {$E_1$} (5,1.5) node [above] {$E_2$} (6,2) node [right] {$E_3$} (10,2) node [above] {$E_4$};
\end{tikzpicture}$$\text{Figure 32}$$ where dotted lines are $(-1)$-curves, $D_1,\ldots,D_6$ are $(-2)$-curves. Moreover, $E\cdot E_1=E\cdot E_2=E\cdot E_4=0$, $E\cdot E_3=1$, $\tilde{E}_1\cdot E_1=\tilde{E}_1\cdot E_4=0$, $\tilde{E}_1\cdot E_2=\tilde{E}_1\cdot E_3=1$. We have \begin{gather*}\varphi^*(H)\equiv L=aD_1+2aD_2+(2a-1-\frac{5}{7}\lambda)E_1+(a+1+\frac{5}{7}\lambda)D_3+\\
+(a+\frac{1}{7}\lambda-\lambda_1)E_2+bD_6+2bD_5+(2b-1-\frac{2}{7}\lambda)E_4+\\
+(b+1+\frac{2}{7}\lambda)D_3+(b-\frac{8}{7}\lambda-\lambda_1)E_3,\end{gather*} where $a+b=1+\lambda+\lambda_1$. Assume that $\lambda<1$. We see that there exist $a$ and $b$ such that $L$ is an effective divisor. Let $Y\rightarrow\tilde{X}$ be the blow up of the intersection point of $D_3$ and $D_4$, and $F$ be the exceptional divisor. Note that there exists $\PP^1$-fibration $g\colon Y\rightarrow\PP^1$ such that $g$ has only two singular fibers $C_1, C_2$ and $$C_1=\bar{D}_1+2\bar{D}_2+\bar{D}_3+2\bar{E}_1+\bar{E}_2,$$ $$C_2=\bar{D}_4+2\bar{D}_5+\bar{D}_6+2\bar{E}_4+\bar{E},$$ where $\bar{D}_1,\ldots,\bar{D}_6$ are the proper transforms of $D_1,\ldots, D_6$, $\bar{E}_1, \bar{E}_2, \bar{E}_3, \bar{E}_4$ are the proper transforms of $E_1,E_2,E_3, E_4$. Then $X$ has an $H$-polar cylinder. So, we may assume that $\lambda\geq 1$. We have the following configuration on $\tilde{X}$.\\
\begin{tikzpicture}
\draw  (0.5,0) -- (0.5,3) (0,0.5) -- (6,0.5) (5.5,0)--(5.5,4) (5.2,3.5)--(9,3.5)  (8.5,4)--(8.5,1) (8,1.5)--(11,1.5); \draw [dashed] (0,2)--(2,2) (3,0)--(3,3) (8,2)--(10,2) (5,2)--(7,2) (6.5,1)--(6.5,3);
\draw (1,0.5) node [above] {$D_2$} (0.5,3) node [above] {$D_1$} (5.5,4) node [above] {$D_3$}  (9,3.5) node [right] {$D_4$}  (8.5,1) node [right] {$D_5$} (11,1.5) node [right] {$D_6$} (3,3) node [above] {$E_2$} (10,2) node [above] {$E_3$} (6.5,1) node [right] {$\tilde{E}_1$} (7,2) node [right] {$E_1$} (1.8,2) node [above] {$E$} ;
\end{tikzpicture}$$\text{Figure 33}$$ where dotted lines are $(-1)$-curves, $D_1,\ldots,D_6$ are $(-2)$-curves. We have \begin{gather*}\varphi^*(H)\equiv L=(2+\frac{4}{7}\lambda)D_3+(a-2+\lambda)E+(a-1+\frac{6}{7}\lambda)D_1+\\
+(a+\frac{5}{7}\lambda)D_2+(a-1)E_2+(b+\frac{3}{7}\lambda)D_4+(b-1+\frac{1}{7}\lambda)D_6+\\
+(2b-2+\frac{2}{7}\lambda)D_5+(2b-3)E_3+cE_1+(c-1+\lambda_1)\tilde{E}_1,\end{gather*} where $a+b+c=4$. Since $\lambda\geq 1$, we see that there exist $a$ and $b$ such that $L$ is an effective divisor. Note that there exists $\PP^1$-fibration $g\colon \tilde{X}\rightarrow\PP^1$ such that $g$ has only three singular fibers $C_1, C_2, C_3$ and $$C_1=E+D_1+D_2+E_2,$$ $$C_2=D_4+D_6+2D_5+2E_3,$$ $$C_3=E_1+\tilde{E}_1.$$ Then $X$ has an $H$-polar cylinder.
\end{proof}

We see that Theorem \ref{glav} follows from Lemmas \ref{Lem1} --- \ref{LemA6}.

\end{document}